\documentclass[11pt,reqno]{amsart}

\usepackage{color}
\usepackage{amsmath, amsthm, amssymb}
\usepackage{amsfonts}

\usepackage{hyperref}

\usepackage[ansinew]{inputenc}
\usepackage[dvips]{epsfig}
\usepackage{graphicx}
\usepackage[english]{babel}

\theoremstyle{plain}
\newtheorem{thm}{Theorem}[section]
\newtheorem{cor}[thm]{Corollary}
\newtheorem{lem}[thm]{Lemma}
\newtheorem{prop}[thm]{Proposition}

\theoremstyle{definition}
\newtheorem{defi}[thm]{Definition}

\theoremstyle{remark}
\newtheorem{rem}[thm]{Remark}

\numberwithin{equation}{section}

\newcommand{\average}{{\mathchoice {\kern1ex\vcenter{\hrule height.4pt
width 6pt depth0pt} \kern-9.7pt} {\kern1ex\vcenter{\hrule
height.4pt width 4.3pt depth0pt} \kern-7pt} {} {} }}

\def\R{\mathbb{R}}

\newcommand{\LL}{\mathcal L}
\newcommand{\oo}{\circ}
\newcommand{\eps}{\varepsilon}

\DeclareSymbolFont{matha}{OML}{txmi}{m}{it}
\DeclareMathSymbol{\varv}{\mathord}{matha}{118}

\begin{document}

\title[Regularity for nonlocal obstacle problems]{Regularity theory for nonlocal obstacle problems \\ with critical and subcritical scaling}

\author{Alessio Figalli}
\address{ETH Z\"urich, Department of Mathematics, Raemistrasse 101, 8092 Z\"urich, Switzerland}
\email{alessio.figalli@math.ethz.ch}

\author{Xavier Ros-Oton}
\address{ICREA, Pg. Llu\'is Companys 23, 08010 Barcelona, Spain \& Universitat de Barcelona, Departament de Matem\`atiques i Inform\`atica, Gran Via de les Corts Catalanes 585, 08007 Barcelona, Spain  \& Centre de Recerca Matem\`atica, Barcelona, Spain}
\email{xros@icrea.cat}

\author{Joaquim Serra}
\address{ETH Z\"urich, Department of Mathematics, Raemistrasse 101, 8092 Z\"urich, Switzerland}
\email{joaquim.serra@math.ethz.ch}

\keywords{Obstacle problem; integro-differential operators; free boundary}
\subjclass[2010]{35R35; 47G20; 35B65.}

\maketitle

\begin{abstract}
Despite significant recent advances in the regularity theory for obstacle problems with integro-differential operators, some fundamental questions remained open. On the one hand, there was a lack of understanding of parabolic problems with critical scaling, such as the obstacle problem for $\partial_t+\sqrt{-\Delta}$. No regularity result for free boundaries was known for parabolic problems with such scaling. On the other hand, optimal regularity estimates for solutions (to both parabolic and elliptic problems) relied strongly on monotonicity formulas and, therefore, were known only in some specific cases. 
In this paper, we present a novel and unified approach to answer these open questions and, at the same time, to treat very general operators, recovering as particular cases most previously known regularity results on nonlocal obstacle problems.
\end{abstract}

\section{Introduction and results}

Free boundary problems appear in several areas of pure and applied mathematics, and have been a central line of research in elliptic and parabolic PDE's during the last fifty years.
The most important and challenging question in this context is to understand the \emph{regularity of free boundaries}.
The development of the regularity theory for free boundaries started in the late seventies with the works of Caffarelli \cite{C-obst}, and since then several ideas and techniques have been developed; see for example the books \cite{Friedman,CS-book,PSU,FR20}.

During the last decade, starting with the works \cite{ACS,S-obst,CSS}, an abundance of new results has been obtained, understanding for the first time \emph{thin} and \emph{nonlocal} free boundary problems.

The motivation for studying such type of problems comes from elasticity (the classical Signorini problem); probability and finance (optimal stopping for jump processes, pricing of options); control problems (boundary heat control); fluid dynamics in biology (osmosis, semipermeable membranes); or interacting energies in physical, biological, and material sciences.
We refer to the classical book of Duvaut and Lions \cite{DL}, as well as to \cite{PS06,Merton,CT} and \cite{CDM,Serfaty}, for a description of these models.

The above-mentioned works \cite{ACS,S-obst,CSS} established for the first time:

- the optimal regularity of solutions, and

- regularity of free boundaries near regular points

\noindent both in the thin obstacle problem, and in the obstacle problem for the fractional Laplacian.
After these results, new methods and techniques have been introduced in \cite{GP,CF,KPS,DS2,DGPT,CRS,JN,ACM,FS17,BFR2,ACMb,CSV17,AR,CSV20b,FJ21,Kuk21,SY21,RT21,Kuk22}, studying various questions such as singular free boundary points, higher regularity of free boundaries, more general nonlocal operators, and the parabolic versions of these problems.
However, despite such significant developments in the last years, some central questions remained open.

On the one hand, there was a lack of understanding of parabolic problems with \emph{critical} scaling, such as the obstacle problem for $\partial_t+\sqrt{-\Delta}$: 
no regularity result for free boundaries was known for any parabolic problem with such scaling.
We note that the case $\partial_t+\sqrt{-\Delta}$ is particularly interesting, because the problem is equivalent to a thin obstacle problem in $\R^{n+1}_+$ with \emph{dynamic} boundary conditions, i.e., 
\[\partial_{x_{n+1}} u=\partial_t u\quad \textrm{on}\ \{x_{n+1}=0\}\cap \{u>\varphi\}.\]
Free boundary problems with dynamic boundary conditions are discussed in \cite{DL} and \cite{Elliot} (see also \cite{AC,ACM,ACMb}), and 
no regularity result for free boundaries was known for any problem of this type.
The main difficulty comes from the critical scaling of the equation, since the equation and free boundary have the same ``hyperbolic'' scaling in time and space.
Because of the traveling wave solutions constructed in \cite{CF}, the structure of the free boundary in such a setting was expected to be much more complicated and rich than in previously known parabolic obstacle problems.

On the other hand, a second important open question was to establish \emph{optimal regularity estimates} for solutions to nonlocal (parabolic and elliptic) obstacle problems.
Indeed, optimal regularity estimates relied strongly on monotonicity formulas, and therefore were only known in very specific situations.
In the elliptic setting, they were only known for the fractional Laplacian, but not for more general nonlocal operators.\footnote{The results in \cite{CRS} establish the regularity of free boundaries and local $C^{1+s}$ estimates near regular points, but not a global nor uniform $C^{1+s}$ estimate for solutions.}
In the parabolic setting, for the fractional Laplacian the optimal regularity of solutions in space was established in \cite{CF}, but even in such case the optimal regularity in time (or in space-time) was open.
It is important to notice that the results in \cite{CRS,BFR2} establish regularity results for free boundaries in these problems, but these are qualitative results, and do not yield in any case optimal regularity estimates for solutions.
Furthermore, still in the parabolic setting, all known results\footnote{The only known result in this direction is the recent work \cite{RT21}, in which the second author and Torres-Latorre studied the \emph{supercritical} case $s<\frac12$. Such case turns out to be completely different, since the time derivative $\partial_t$ dominates; see Remark \ref{rem-super} below.} are for the fractional Laplacian, and used monotonicity formulas \cite{CF} or the extension problem for the fractional Laplacian \cite{BFR2}.
Extending these results to more general nonlocal operators was an open problem, too.

The aim of this paper is to develop a unified approach to the regularity theory of such problems that allow us to answer all these open questions at the same time. Note that, in addition to giving an answer to the open problems mentioned above, we can also recover, as particular cases, all the previously known regularity results on nonlocal obstacle problems from \cite{ACS,CSS,CF,CRS,BFR2,FR18}.

We consider nonlocal operators of the form
\begin{equation}\label{L0}
\LL u(x) = \textrm{p.v.}\int_{\R^n} \bigl(u(x+y)-u(x)\bigr) K(y)\,dy,
\end{equation}
with
\begin{equation} \label{ellipt}
K(y)=K(-y)\qquad \textrm{and}\qquad \frac{\lambda}{|y|^{n+2s}}\leq K(y)\leq \frac{\Lambda}{|y|^{n+2s}}.
\end{equation}
The constants $0<\lambda\leq \Lambda$ are called ellipticity constants, and $s\in(0,1)$.
This is the most typical and natural class of operators of order $2s$; see \cite{basslevin,CS09,Ros16}.
(Notice that some of the results of the paper hold for more general  classes of operators, as considered in Definition \ref{L}.)

\subsection{Main result}

Given $\LL$ of the form \eqref{L0}-\eqref{ellipt}, and given an obstacle $\varphi$ in $\R^n$, we consider the parabolic obstacle problem
\begin{equation}\label{obst-pb}\begin{split}
\min\bigl\{u_t-\LL u,\,u-\varphi\bigr\}&=0\quad\textrm{in}\ \, \R^n\times (0,T),\\
u(0)&=\varphi\quad\textrm{in}\ \, \R^n.
\end{split}\end{equation}
The solution $u(x,t)$ of \eqref{obst-pb} can be constructed as the smallest supersolution lying above the obstacle~$\varphi$; see \cite{CF,RT21}.
We will always assume
\begin{equation}\label{obstacle}
\|\varphi\|_{C^4(\R^n)}\leq C_\oo.
\end{equation}
Throughout the paper, we will denote
\[\mathcal Q_r :=B_r\times (-r^{2s},r^{2s}).\]

To understand the regularity of solutions and to the free boundary for \eqref{obstacle}, we shall first prove a very general result about almost-convex solutions to the obstacle problem with zero obstacle  and a small right hand side.
This result reads as follows:

\begin{thm}[Quantitative estimate]\label{thm-main}
Let $s \in[\frac12,1)$, $\LL$ as in \eqref{L0}-\eqref{ellipt}, with $K$ homogeneous.
Fix $\delta>0$, and given $\eta>0$ small, assume that  $u\in {\rm Lip}(\R^n\times (-1/\eta,1/\eta))$ satisfies:
\begin{itemize}
\item[$\bullet$]  $u$ is nonnegative, monotone, and almost-convex:
\[u\geq0,\quad \partial_t u\geq0,\quad \textrm{and} \quad D^2_{x,t} u \ge - \eta \,{\rm{Id}}\qquad \textrm{in}\quad \mathcal Q_{1/\eta},\]
\[\textrm{with}\quad (0,0)\in\partial\{u>0\}.\]

\item[$\bullet$]  $u$ solves the obstacle problem with zero obstacle  and a small right hand side:
\[\partial_t u-\LL u=f\quad \textrm{in}\quad \{u>0\}\cap \mathcal Q_{1/\eta} \quad \textrm{ and }\quad \partial_t u-\LL u\geq f \quad \textrm{in} \quad \mathcal Q_{1/\eta},\]
\[ \textrm{with} \quad  |\nabla f|+|\partial_{t} f|\leq \eta.\]

\item[$\bullet$]  $u$ has a controlled growth at infinity:
\[R\|\nabla u\|_{L^\infty(\mathcal Q_R \cap \{|t|<1/\eta\})}+R^{2s}\|\partial_t u\|_{L^\infty(\mathcal Q_R\cap \{|t|<1/\eta\})} \le R^{2-\delta} \qquad \textrm{for all}\quad R\ge 1.\]
\end{itemize}

\noindent Then, there exists a 1D solution of the form
\begin{equation}\label{u0}
u_\circ(x,t)=\left\{\begin{array}{rl}
\kappa (x\cdot e+\varv t)_+^{1+\gamma} & \quad \textrm{if}\quad s={\textstyle \frac12} \vspace{3mm} \\
\kappa (x\cdot e)_+^{1+s}& \quad \textrm{if}\quad s>{\textstyle \frac12},
\end{array}\right.
\end{equation}
with $\kappa>0$, $e\in \mathbb S^{n-1}$, $\varv\geq0$, and $\gamma=\gamma(\LL,\varv,e)\geq\frac12$,
such that 
\[
\|u- u_\circ\|_{{\rm Lip} (\mathcal Q_1)} \le \eps(\eta),
\]
where $\eps(\eta)$ is a modulus of continuity\footnote{That is, $\eps: (0,\infty)\to (0,\infty)$ is nondecreasing function with $\lim_{\eta\downarrow 0} \eps(\eta) =0$.} depending only on $n$, $s$, $\delta$, $\lambda$, $\Lambda$.

Moreover, for any given $\kappa_\circ>0$ there exist $\eps_\circ>0$ such that  if $\eps(\eta)< \eps_\circ$ and $\kappa\geq \kappa_\circ>0$, then the free boundary $\partial \{u>0\}$ is a $C^{1,\tau}$ graph in $\mathcal Q_{1/2}$ for some $\tau>0$, and we have the bound
\[|\nabla u|+|\partial_t u|\leq C\big(|x|^s+|t|^s\big)\]
for $(x,t)\in \mathcal Q_1$.
The constants $\eps_\circ$, $C$, and $\tau$,  depend only on  $n$, $s$, $\delta$, $\lambda$, $\Lambda$, and $\kappa_\circ$.
\end{thm}

While the previous theorem holds for $s \in[\frac12,1)$, in the elliptic setting, i.e.
\begin{equation}\label{obst-pb-ell}
\min\bigl\{-\LL u,\,u-\varphi\bigr\}=0\quad\textrm{in}\ \, \R^n,
\end{equation}
the analogous result is valid for all $s\in(0,1)$; see Theorem \ref{thm-main-ell}.

\begin{rem}[On the assumption $s\geq\frac12$]\label{rem-super}
Notice that, in the parabolic setting, the case $s<\frac12$ needs to be excluded if we look for a unified theory for both elliptic and parabolic nonlocal obstacle problems.
Indeed, the case $s<\frac12$ turns out to be completely different, both in terms of the results and the methods to study it.
It was proved very recently in \cite{RT21} that, when $s<\frac12$, non-stationary solutions are automatically $C^{1,1}$ in space and time, independently of the parameter $s$.
Moreover, the proof of such result is independent from the ones in the stationary setting (and from the case $s\geq\frac12$), since it uses very strongly the fact that $\partial_t$ is the dominating term in the equation.
\end{rem}

\subsection{Regularity of free boundaries}

Iterating our main result above and combining it with the explicit 1D profiles in the case $\LL=\sqrt{-\Delta}$, we get the following.

\begin{cor}[Regularity of the free boundary, $s=\frac12$]\label{thm1}
Let $\LL=\sqrt{-\Delta}$, $\varphi$ an obstacle satisfying \eqref{obstacle}, and $u$ the solution to \eqref{obst-pb}.
Then, at each free boundary point $(x_\circ,t_\circ)\in\partial\{u>\varphi\}$ we have the following dichotomy:
\begin{itemize}
\item[(i)] either \vspace{-1mm}
\[\hspace{16mm}0<c\,r^{1+\gamma(x_\circ,t_\circ)}\leq \sup_{\mathcal Q_r(x_\circ,t_\circ)}(u-\varphi)\leq Cr^{1+\gamma(x_\circ,t_\circ)},\qquad \text{for some}\quad\gamma(x_\circ,t_\circ)\in[{\textstyle \frac12},1),\]
\vspace{2mm}
\item[(ii)] or \hspace{20mm} $\displaystyle 0\leq \sup_{\mathcal Q_r(x_\circ,t_\circ)}(u-\varphi)\leq C_\eps r^{2-\eps}$ \qquad for all \ $\eps>0$.
\end{itemize}
In addition, the set of points $(x_\circ,t_\circ)$ satisfying (i) is an open subset of the free boundary and it is a $C^{1,\alpha}$  submanifold in space-time of codimension 1.

Moreover, if we denote by $\nu=(\nu_x,\nu_t)$ the normal vector\footnote{More precisely, $\nu$ is the normal vector to $\partial \{u>\varphi\}$ pointing towards $\{u>\varphi\}$. Notice that since $u_t\geq0$ then we always have $\nu_t\geq0$.} to the free boundary at $(x_\circ,t_\circ)$, then the exponent $\gamma(x_\circ,t_\circ)$ is given by
\[\gamma(x_\circ,t_\circ):=\frac12+\frac{1}{\pi}\arctan(\varv_\circ),\]
where $\varv_\circ:=\nu_t/|\nu_x|\geq 0$ is the {speed} of the free boundary at $(x_\circ,t_\circ)$.
\end{cor}

As said above, this is the first result concerning the regularity of the free boundary for a critical operator such as $\partial_t+\sqrt{-\Delta}$.
Prior to our result, the ``subcritical'' case $s>\frac12$ was understood in \cite{BFR2}, while the ``supercritical'' case $s<\frac12$ was treated in \cite{RT21}  (cf. Remark~\ref{rem-super}).

In the case $s>\frac12$, our new approach allows us to extend the results in \cite{BFR2} to much more general kernels, and 
the results of \cite{CRS} to the parabolic setting.

\begin{cor}[Regularity of the free boundary, $s>\frac12$]\label{thm2}
Let $\LL$ be of the form \eqref{L0}-\eqref{ellipt} with $K$ homogeneous\footnote{The assumption of the kernel $K$ being homogeneous is needed in order to ensure that 1D solutions are homogeneous; see \cite{RS-Duke,CRS}.}, $\varphi$ be an obstacle satisfying \eqref{obstacle}, and $u$ be the solution to \eqref{obst-pb}.
Then, at each free boundary point $(x_\circ,t_\circ)\in\partial\{u>\varphi\}$ we have the following dichotomy:
\begin{itemize}
\item[(i)] either \vspace{-1mm}
\[\hspace{-2mm}0<c\,r^{1+s}\leq \sup_{\mathcal Q_r(x_\circ,t_\circ)}(u-\varphi)\leq Cr^{1+s},\]
\vspace{2mm}
\item[(ii)] or \hspace{35mm} $\displaystyle 0\leq \sup_{\mathcal Q_r(x_\circ,t_\circ)}(u-\varphi)\leq C_\eps r^{2-\eps}$ \qquad for all \ $\eps>0$.
\end{itemize}
In addition, the set of points $(x_\circ,t_\circ)$ satisfying (i) is an open subset of the free boundary and it is $C^{1,\alpha}$ in space-time.
\end{cor}

This result reads exactly as the one in \cite{BFR2} for the fractional Laplacian $(-\Delta)^s$, $s>\frac12$.
Still, as we will see next, the results of the present paper are quantitative in nature (thanks to Theorem \ref{thm-main}), while the ones in \cite{BFR2} (as well as \cite{CRS}) were qualitative.
Thanks to this fact, we can establish several new regularity estimates for solutions, both in the parabolic and elliptic setting.

\subsection{Optimal regularity estimates}

We present here some consequences of our main result, Theorem \ref{thm-main} and its elliptic counterpart (see Theorem \ref{thm-main-ell} below), in terms of optimal regularity estimates for solutions.
In the elliptic case, we answer an open question left in \cite{CRS}:

\begin{cor}[$C^{1+s}$ elliptic estimates]\label{thm-sol-1}
Let $\LL$ be of the form \eqref{L0}-\eqref{ellipt} with $K$ homogeneous, $\varphi$ be an obstacle satisfying \eqref{obstacle}, and $u$ be the solution to \eqref{obst-pb-ell}.
Then $u\in C^{1+s}(\R^n)$ and
\[\|\nabla u\|_{C^{s}(\R^n)} \leq CC_\circ,\]
with $C$ depending only on $n$, $s$, $\lambda$, and $\Lambda$.
\end{cor}

In the parabolic critical case $\partial_t+\sqrt{-\Delta}$, we establish the optimal $C^{3/2}$-regularity of solutions in space-time, answering a question left open in \cite{CF}.

\begin{cor}[$C^{3/2}_{x,t}$ estimates for $s=\frac12$]\label{thm-sol-2}
Let $\LL$ be of the form \eqref{L0}-\eqref{ellipt} with $K$ homogeneous, $\varphi$ be an obstacle satisfying \eqref{obstacle}, and $u$ be the solution to \eqref{obst-pb}.
Then, $u\in C^{3/2}_{x,t}(\R^n\times (0,T))$ and for any $[t_1,t_2]\subset (0,T]$ we have
\[\|\nabla u\|_{C^{1/2}(\R^n\times [t_1,t_2])}+\|\partial_t u\|_{C^{1/2}(\R^n\times [t_1,t_2])} \leq CC_\circ,\]
with $C$ depending only on $n$, $s$, $\lambda$, $\Lambda$, and $t_1$.
\end{cor}

In case $s>\frac12$, the results in \cite{CF} imply that solutions $u$ are $C^{1+s}$ in space and $C^{\frac{1+s}{2s}-\eps}$ in time, for all $\eps>0$.
Here, we improve the regularity in time to the optimal scaling.
Notice that our results hold for the general class of kernels considered in \cite{CRS}, but they are new even for the fractional Laplacian.

\begin{cor}[Further regularity in time, $s>1/2$]\label{thm-sol-3}
Let $\LL$ be of the form \eqref{L0}-\eqref{ellipt} with $K$ homogeneous, $\varphi$ be an obstacle satisfying \eqref{obstacle}, and $u$ be the solution to \eqref{obst-pb}.
\begin{itemize}
\item If $s\in (\frac12,\frac{\sqrt5-1}{2})$  then
$u\in C^{1+s}_{x,t}(\R^n\times (0,T))$ and for any $[t_1,t_2]\subset (0,T]$ we have
\[\hspace{15mm} \|\nabla u\|_{C^{s}(\R^n\times [t_1,t_2])}+\|\partial_t u\|_{C^{s}(\R^n\times [t_1,t_2])} \leq CC_\circ,\]
with $C$ depending only on $n$, $s$, $\lambda$, $\Lambda$, and $t_1$.
\item If $s\in [\frac{\sqrt5-1}{2},1)$ then
$u\in C^{1+s}_{x}\cap C^{\frac{1}{s}-\eps}_t(\R^n\times (0,T))$ for any $\eps>0$, and for any $[t_1,t_2]\subset (0,T]$ we have
\[\hspace{25mm}\|\partial_t u\|_{C^{\frac{1}{s}-1-\eps}_{t}(\R^n\times [t_1,t_2])} \leq C_\eps C_\circ,\]
with  $C_\eps$ depending only on $n$, $s$, $\lambda$, $\Lambda$, $\eps$, and $t_1$.
\end{itemize}
\end{cor}

The regularity estimate for $s<\frac{\sqrt5-1}{2}$ is clearly optimal (in view of the description of solutions in Theorem \ref{thm2}), and we expect the regularity estimate for $s>\frac{\sqrt5-1}{2}$ to be almost-optimal.

We thus find a new threshold at which the regularity of solutions changes and, curiously,  this threshold is at exactly the golden ratio 
\[s={\textstyle\frac{\sqrt5-1}{2}}\approx 0.61803.\]
The reason for this is that, when looking at the regularity of solutions in $t$, the ``worst points'' for $s<\frac12(\sqrt5-1)$ (case (i) above) are the regular ones, while for $s\geq\frac12(\sqrt5-1)$  (case (ii) above) the ``worst regularity'' happens at singular points.

\subsection{Nonsymmetric operators}

The new quantitative methods developed in this paper are very flexible. For instance, the symmetry assumption on the kernels in \eqref{L0} is not needed for some of our results to hold, and we can establish new regularity results for solutions and free boundaries in the non-symmetric case.

As a model case, we consider the elliptic problem for the fractional Laplacian with critical drift ($s=\frac12$) and establish the optimal regularity of solutions, thus answering an open question from \cite{FR18}.

\begin{cor}\label{thm-sol-4}
Let $\LL=\sqrt{-\Delta}+b\cdot \nabla$ with $b\in \R^n$, $\varphi$ an obstacle satisfying \eqref{obstacle}, and $u$ the solution to \eqref{obst-pb}.
Then $u\in C^{1+\gamma_b}(\R^n)$, with 
\[\gamma_b:=\frac12-\frac{1}{\pi}\arctan |b|\qquad \text{and}
\qquad \|u\|_{C^{1+\gamma_b}(\R^n)} \leq CC_\circ,\]
with $C$ depending only on $n$ and $|b|$.
\end{cor}

Notice that $\gamma_b\in(0,\frac12)$ and $\gamma_b\to\frac12$ as $b\to0$. 
The expression of $\gamma_b$ comes from an explicit computation of 1D solutions; see \cite{FR18,DRSV} for more details.

\subsection{Acknowledgements}

AF has been supported by the European Research Council (ERC) under the Grant Agreement No 721675.
XR has been supported by the European Research Council (ERC) under the Grant Agreement No 801867, by the AEI project PID2021-125021NA-I00 (Spain), by AGAUR Grant 2021 SGR 00087 (Catalunya), by MINECO Grant RED2022-134784-T (Spain), and by the SpanishAEI through the Mar\'ia de Maeztu Program for Centers and Units of Excellence in R\&D CEX2020-001084-M.
JS has been supported by Swiss NSF Ambizione Grant PZ00P2 180042 and by the European Research Council (ERC) under the Grant Agreement No 948029.

\subsection{Organization of the paper}

The paper is organized as follows.
In Section \ref{sec2} we prove all our results in the elliptic setting, deducing in particular Corollaries \ref{thm-sol-1} and \ref{thm-sol-4}.
In Section \ref{sec3} we establish a new parabolic boundary Harnack inequality, which plays a crucial role in the proof of our main results in the parabolic setting.
In Section \ref{sec4} we prove Theorem \ref{thm-main}.
Finally, in Section \ref{sec5} we deduce Corollaries~\ref{thm1}, \ref{thm2}, \ref{thm-sol-2}, and \ref{thm-sol-3}

\section{The elliptic case}
\label{sec2}

In this section we prove the analogue of Theorem~\ref{thm-main} in the stationary case. We start from this case because, 
in this setting, the arguments are simpler and are valid for every $s\in(0,1)$.
In addition, the proofs are shorter, since we can rely on several known results from \cite{CRS,RS-bdryH}. 

Actually, thanks to the recent (elliptic) results from \cite{DRSV}, we can establish our results also for non-symmetric operators.
The general class of operators that we consider in the elliptic case is the following.

\begin{defi}\label{L}
Throughout this Section, we consider operators $\LL$ of the form
\begin{equation*}\label{L0-A}
\LL u(x) = \int_{\R^n} \bigl(u(x+y)-u(x)\bigr) K(y)\,dy\qquad \textrm{if}\quad s\in(0,{\textstyle\frac12}),
\end{equation*}
\begin{equation*}\label{L0-B}
\LL u(x) = \textrm{p.v.}\int_{\R^n} \bigl(u(x+y)-u(x)\bigr) K(y)\,dy\,+\,b\cdot\nabla u(x) \qquad \textrm{if}\quad s={\textstyle\frac12},
\end{equation*}
\begin{equation*}\label{L0-C}
\LL u(x) = \int_{\R^n} \bigl(u(x+y)-u(x)-\nabla u(x)\cdot y\bigr) K(y)\,dy \qquad \textrm{if}\quad s\in ({\textstyle\frac12},1),
\end{equation*}
with $b\in\R^n$ satisfying $|b|\leq\Lambda$, and
\begin{equation*} \label{ellipt-A}
\frac{\lambda}{|y|^{n+2s}}\leq K(y)\leq \frac{\Lambda}{|y|^{n+2s}}.
\end{equation*}
If $s=\frac12$ we must add the  standard ``zero-moment assumption''    $\int_{R_{2r}\setminus B_r }yK(y)dy=0$ for all $r>0$, so that the principal value integral defining $\LL$ is well-defined.
\end{defi}

We refer to \cite{DRSV} for some basic interior and boundary regularity estimates for such class of operators.
This is basically the most general scale-invariant class of linear operators of order $2s$ for which we have both interior and boundary Harnack.

\subsection{Main elliptic result}
The main result of this section is the following quantitative estimate.

\begin{thm}\label{thm-main-ell}
Let $s \in(0,1)$, $\LL$ as in Definition \ref{L}, and $\alpha_\circ\in (0,s)\cap (0,1-s)$.


Let $\eta>0$ and suppose that  $u\in {\rm Lip}(\R^n)$ satisfies:
\begin{itemize}
\item[$\bullet$]  $u$ is nonnegative and almost-convex in a large ball:
\[u\geq0\quad \textrm{and} \quad D^2 u \ge - \eta \,{\rm{Id}}\qquad \textrm{in}\quad B_{1/\eta},\quad \textrm{with}\quad  0 \in \partial \{u>0\}.\]

\item[$\bullet$]  $u$ solves the obstacle problem with a small right hand side:
\[\LL u=f\quad \textrm{in}\quad \{u>0\}\cap B_{1/\eta} \quad \textrm{and} \quad \LL u\leq f \quad \textrm{in} \quad B_{1/\eta}, \quad \textrm{with} \quad  |\nabla f|\leq \eta.\]

\item[$\bullet$]   $u$ has a controlled growth at infinity:
\[\|\nabla u\|_{L^\infty(B_R)} \le R^{s+\alpha_\circ} \qquad \textrm{for all}\quad R\ge 1.\]
\end{itemize}
Then:
\begin{itemize}
\item[(i)] There exist $e\in \mathbb S^{n-1}$ and a nonnegative convex 1D solution $u_\circ(x)=U(x\cdot e)$, satisfying 
\begin{equation}\label{1D-profile}
\begin{split} 
\LL(\nabla u_\circ) & =0\qquad \textrm{in}\quad \{x\cdot e>0\}\\
u_\circ&=0\qquad \textrm{in}\quad \{x\cdot e\leq 0\}\\
\|\nabla u_\circ\|_{L^\infty(B_R)}& \leq R^{s+\alpha_\circ}\quad  \text{for all} \quad R\geq1,
\end{split}
\end{equation}
such that 
\[
\|u- u_\circ\|_{{\rm Lip} (B_1)} \le \eps(\eta),
\]
where $\eps(\eta)$ is a modulus of continuity depending only on $n$, $s$, $\alpha_\circ$, $\lambda$, and $\Lambda$.
\item[(ii)]
Moreover, given $\kappa_\circ>0$ exists $\eps>0$ such that  if $\|u_\circ\|_{{\rm Lip} (B_1)}\geq \kappa_\circ>0$ and $\eps(\eta)<\eps_\circ$, then the free boundary $\partial \{u>0\}$ is a $C^{1,\tau}$ graph in $B_{1/2}$, for some $\tau>0$.
\item[(iii)]
If in addition the kernel $K$ of the operator $\LL$ is homogeneous, then $u_\circ$ is homogeneous of degree $\gamma=\gamma(\LL,e)\in (0,2s)\cap (2s-1,1)$ and we have the expansion
\[\qquad \qquad |u-u_\circ|\leq C|x|^{1+\gamma+\tau}\quad \textrm{and}\quad |\nabla u-\nabla u_\circ|\leq C|x|^{\gamma+\tau}\quad \textrm{for}\quad x\in B_1.\]
Furthermore, if $K$ is symmetric, then $\gamma\equiv s$ for all $e\in \mathbb S^{n-1}$.
\end{itemize}
Here, the constants $C$, $\eps_\circ$, and $\tau$ depend only on  $n$, $s$, $\alpha_\circ$, $\lambda$, $\Lambda$, and $\kappa_\circ$.
\end{thm}

Part (i) of this quantitative result is basically equivalent to showing that all blow-ups are~1D (at nondegenerate points).
Part (ii) is essentially the regularity of the free boundary near regular points, and is somewhat independent from (i).
Still, such combined quantitative versions can be iterated and will give us some more information, as we will show later.

\begin{rem}
Once $e\in \mathbb S^{n-1}$ is fixed, the 1D profile $u_\circ$ is uniquely determined, up to a multiplicative constant (see Proposition \ref{prop-classification-ell}).
Moreover, when the kernel $K$ is homogeneous then $u_\circ$ can be computed explicitly, and if $K$ is in addition symmetric then $u_\circ(x)=c(x\cdot e)_+^{1+s}$, as in \cite{CRS}.
\end{rem}

\subsection{Proof of the main elliptic result}

To prove the result we will need several ingredients.
The first one is the (elliptic) boundary Harnack for such class of operators.

\begin{thm}[\cite{RS-bdryH,DRSV}]\label{ellbdryharnack}
Let $s \in(0,1)$ and $\LL$ as in Definition \ref{L}.
Let $\Omega\subset \R^n$ be a Lipschitz graph in $B_1$, with $0\in \partial\Omega$.
Then, there exist  positive constants $\eta$, $C$,  and $\tau$ depending only on $n$, $s$, $\lambda$, $\Lambda$, and the Lipschitz norm of $\partial\Omega$ in $B_1$, such that the following holds.

Let $v_1,v_2$ be weak (or viscosity) solutions of 
\[
\big|\LL v_i  \big| \le  \eta \quad \mbox{in}\quad \Omega \cap B_1,  \qquad v_i \equiv 0 \quad \mbox{in}\quad\Omega^c \cap B_1,
\]
satisfying 
\[v_i\geq0\qquad \textrm{in}\quad \R^n\qquad \textrm{and}\qquad \int_{\R^n}\frac{|v_i(x)|}{1+|x|^{n+2s}}\,dx=1.\]
Then, there exists $\tau>0$ such that 
\[
\left\|\frac{v_1}{v_2}\right\|_{C^{\tau}(\overline\Omega\cap B_{1/2})}\le  C.
\]
\end{thm}

We will also need the following:

\begin{lem}\label{lem-supersol1D}
Let $s\in(0,1)$, $\LL$ as in Definition \ref{L}, and $e\in \mathbb S^{n-1}$.
Then, there exists $\theta>0$ such that 
\[\phi(x):= \exp\big(-|x\cdot e|^{1-\theta}\big)\]
satisfies 
 \[\LL \phi\leq C\qquad\textrm{in}\quad \R^n.\]
The constants $C$ and $\theta$ depend only on $n$, $s$, and the ellipticity constants.
\end{lem}

\begin{proof}
We prove it for $e=e_n$.
Let $\mathcal M_{s,\lambda,\Lambda}^-$ be the extremal operator associated to our class of operators, i.e., $\mathcal M_{s,\lambda,\Lambda}^- w:=\inf_\LL \LL w$, where the infimum is taken among all operators $\LL$ as in Definition \ref{L} (with fixed $s$, $\lambda$, $\Lambda$).
Then, the operator $\mathcal M_{s,\lambda,\Lambda}^-$ is scale invariant of order $2s$, and in particular $\mathcal M_{s,\lambda,\Lambda}^-|x_n|^\beta = c_\beta |x_n|^{\beta-2s}$ for $\beta\in(0,2s)$ (see \cite[Section~2]{RS-Duke}).
Moreover, it is easy to see that $c_\beta\to+\infty$ as $\beta\to 2s$, and in addition $c_\beta>0$ for $s\ge\frac 12$ (by convexity).
Hence, since $c_\beta$ is continuous with respect to $\beta$,  for any $s\in(0,1)$ there is $\theta>0$ such that $s<1-\theta<2s$ and $c_{1-\theta}>0$.

This implies that for any operator $\LL$ as in Definition \ref{L} we have 
\[\LL|x_n|^{1-\theta}\geq c_{1-\theta}|x_n|^{1-\theta-2s}\geq0 \qquad \textrm{in}\quad \R^n,\]
with $c_{1-\theta}>0$.
In particular, since the function $\phi(x)+|x_n|^{1-\theta}$ is of class $C^{2(1-\theta)}\subset C^{2s+\delta}$ for some $\delta>0$, we conclude that the function $\phi$ satisfies $\LL \phi \leq C$ in $\R^n$, as wanted.
\end{proof}

As a consequence of the previous supersolution, we find:

\begin{lem}\label{lem-solution-across}
Let $s\in(0,1)$, $\LL$ as in Definition \ref{L}, $e\in \mathbb S^{n-1}$, and $\Gamma\subset \{x\cdot e=0\}$.
Assume that $w\in {\rm Lip}_{\rm loc}(\R^n)$ is a viscosity solution of
\[\LL w = 0 \qquad \textrm{in}\quad \R^n\setminus \Gamma.\]
Then $\LL w=0$ in $\R^n$.
\end{lem}

\begin{proof}
For any $\varepsilon>0$ we consider the function $w_\varepsilon := w-\varepsilon\phi$, where $\phi$ is given by Lemma~\ref{lem-supersol1D}.

Assume now that a test function $\eta\in C^2$ touches $w_\varepsilon$ from above at $x_\circ\in \R^n$.
Since $w$ is Lipschitz, then by definition of $\phi$ we have that $w_\varepsilon$ has a ``downwards cusp'' on $\{x\cdot e=0\}$, and therefore $x_\circ\notin \{x\cdot e=0\}$.
Thus $\LL\eta(x_\circ) = \LL w(x_\circ)-\varepsilon \LL \phi(x_\circ)\geq -C\varepsilon$.
Since this holds for every test function $\eta\in C^2$, we deduce that $\LL w_\varepsilon \geq -C\varepsilon$ in $\R^n$ in the viscosity sense.
Since $w=\sup_{\varepsilon>0} w_\varepsilon$, we conclude that $\LL w\geq0$ in $\R^n$.

Repeating the same argument with $-w$ instead of $w$, we find $\LL w=0$ in $\R^n$, as desired.
\end{proof}

Thanks to the previous results, we can prove the classification of blow-ups.

\begin{prop}\label{prop-classification-ell}
Let $s \in(0,1)$, $\LL$ as in Definition \ref{L}, and $\alpha_\circ\in (0,s)\cap (0,1-s)$.

Let $u_\circ\in {\rm Lip}(\R^n)$ be a function satisfying:

\begin{itemize}
\item[$\bullet$] $u_\circ$ is nonnegative and convex in $\R^n$:
\[u_\circ\geq0\quad \textrm{and} \quad D^2 u \ge 0\qquad \textrm{in}\quad \R^n,\qquad \textrm{with}\quad u_\circ(0)=|\nabla u_\circ(0)|=0.\]

\item[$\bullet$]  for any given $h \in \R^n$, $u_\circ$ solves
\[\LL(D_h u_\circ)\geq 0 \qquad \textrm{in} \quad \{u_\circ>0\},\]
where 
\[D_h u_\circ(x)={\textstyle \frac{u_\circ(x)-u_\circ(x-h)}{|h|}}.\]

\item[$\bullet$]  $u_\circ$ has a controlled growth at infinity:
\begin{equation}\label{qhiohtiowhw}
\|\nabla u_\circ\|_{L^\infty(B_R)} \le R^{s+\alpha_\circ} \qquad \textrm{for all}\quad R\ge 1.
\end{equation}
\end{itemize}
Then $u_\circ$ is a 1D function, i.e., there exists $e\in \mathbb S^{n-1}$ such that $u_\circ(x)=U(x\cdot e)$.

Moreover, for each $e\in \mathbb S^{n-1}$, the function $u_\circ$ is unique (up to multiplicative constant) and, if the kernel $K$ of the operator $\LL$ is homogeneous, then $u_\circ$ is homogeneous, too.
\end{prop}

\begin{rem}
\label{rem:incremental quotient}
In the sequel, we will implicitly use the following simple observation:
if $u$ is a locally Lipschitz function satisfying
$\LL(D_h u)\geq 0 $ inside $\{u>0\}$
for all $h \in \R^n$, then 
\[\LL(\nabla u)=0\qquad \textrm{in}\quad \{u>0\}.\]
Indeed, given $k \in \{1,\ldots,n\}$ we can choose $h=\epsilon e_k$ to obtain
$$
\LL(\partial_k u)=\lim_{\epsilon \to 0^+}\LL(D_{\epsilon e_k} u) \geq 0 
\quad\text{and}\quad \LL(\partial_k u)=\lim_{\epsilon \to 0^-}\LL(D_{\epsilon e_k} u) \leq 0 \qquad \textrm{in} \quad \{u>0\}.
$$
The same observation applies also to the parabolic case.
\end{rem}

In the proof of Proposition \ref{prop-classification-ell} (and also later on in the paper) we will need the following simple barrier.
\begin{lem}\label{lemsuper1}
Let $s\in(0,1)$ and $\LL$ as in Definition \ref{L}. Given $\eta>0$ there exists $\theta>0$ such that 
\[\Phi(x):=  \bigg( x\cdot e+ \eta |x|\bigg(1-\frac{(x\cdot e)^2}{|x|^2}\bigg)\bigg)_+^\theta,\]
with $e\in \mathbb S^{n-1}$, satisfies 
 \[\LL\Phi\leq -c<0\quad\textrm{in }\mathcal C_\eta \cap B_2,\]
 where $\mathcal C_\eta$ is the cone 
 \[
  \Big\{ \tfrac{x}{|x|}\cdot e\ge- \eta\Big(1- \big(\tfrac{x}{|x|}\cdot e\big)^2\Big)\Big\}.
  \]
The constants $c$ and $\theta$ depend only on $n$, $s$, the ellipticity constants, and $\eta$.
\end{lem}
\begin{proof}
It is a variation (with almost identical proof) of Lemma 4.1 in \cite{RS-C1}. See \cite[Lemma~4.1]{AuR} for more details.
\end{proof}

\begin{rem}\label{remclaricone}
Notice that, given any $\omega\in (0,1)$  (small), the inclusion
\[
\mathcal C_\eta \subset \big\{ \tfrac{x}{|x|}\cdot e \le -1+ \omega\big\}
\]
holds provided $\eta = \eta(\omega)$ is taken sufficiently large.
\end{rem}

\begin{proof}[Proof of Proposition \ref{prop-classification-ell} ]
We follow and simplify the ideas in \cite[Section~4]{CRS}.

First, notice that the set $\{u_\circ=0\}$ is closed and convex.
Then, we separate the proof into two cases:

\vspace{2mm}

\noindent \emph{Case 1}.
Assume that the convex set $\{u_\circ=0\}$ contains a closed convex cone $\Sigma \ni 0$ with nonempty interior.
Then, we can find $n$ independent directions $e_i\in \mathbb S^{n-1}$, $i=1,...,n$, such that $-e_i\in \Sigma\subset \{u_\circ=0\}$, and by convexity of 
$u_\circ$ we deduce that
\[v_i:=\partial_{e_i} u_\circ \geq0\qquad \textrm{in}\quad\R^n.\]
Moreover, since $u_\circ \not\equiv 0,$ at least one of them is not identically zero, say $v_n\not\equiv0$.

We first claim that $v_i$ are continuous functions. Indeed, since $\{u_\circ=0\}$ is a convex set containing the cone $\Sigma$, all the points of its boundary can be touched by the vertex of a translation of the cone $\Sigma$ which is contained in $\{u_\circ=0\}$.

Hence, given any vector $h \in \R^n\setminus \{0\}$, the function $(D_h u_\circ)_+$ is a continuous subsolution vanishing on $\{u_\circ =0\}$ and with growth as in  \eqref{qhiohtiowhw}. 
Now, given $R>2$, let $\psi_R \in C^\infty_c(B_{2R})$ be a smooth cut-off function such that $\psi_R \equiv 1$ in $B_{3R/2}$, and consider the bounded function $(D_h u_\circ)_+\psi_R.$
Thanks to the growth assumption \eqref{qhiohtiowhw} it follows that $\LL((D_h u_\circ)_+\psi_R)\geq -C_R$ in $\R^n$. Hence,
using a large multiple of the supersolution in Lemma \ref{lemsuper1} as barrier (see Remark \ref{remclaricone}) we deduce that, for all  $z\in \partial \{u_\circ >0\}\cap B_R$ and $r \in (0,1)$,  we have 
\[
\sup_{B_r(z)} (D_h u_\circ)_+= \sup_{B_r(z)} (D_h u_\circ)_+\psi_R \le C_R' r^{\theta}.
\]
Since $h$ is arbitrary,  letting $h\to 0$  we obtain  ($u_\circ$ is smooth in the interior of $\{u_\circ>0\}$)
\[
\sup_{B_r(z)} |\nabla u_\circ| \le C_R' r^{\theta} \qquad \mbox{for all}\quad z\in \partial \{u_\circ >0\}\cap B_R, \, r\in (0,1).
\]
Noticing that the gradient of $u$ is smooth in the interior of $\{u_\circ>0\}$ (all partial derivatives satisfy a translation invariant elliptic equation), we conclude that
 $\nabla u_\circ$ is continuous, as claimed.

Hence, recalling \eqref{qhiohtiowhw}, we can apply the boundary Harnack (Theorem \ref{ellbdryharnack} above) to the functions $v_i(2Rx)$ to deduce that $[v_i/v_n]_{C^\tau(B_R)}\leq CR^{-\tau}$, with $C$ independent of $R\geq1$.
Then, letting $R\to\infty$, we conclude the existence of constants $\kappa_i \in \R$ such that 
\[v_i\equiv \kappa_i v_n,\qquad \textrm{for}\quad i=1,...,n-1.\]
This means that $u_\circ$ is a 1D function, as desired.

Moreover, assuming that both $u_{\circ,1}(x)=U_1(e\cdot x)$ and $u_{\circ,2}(x)=U_2(e\cdot x)$ satisfy all the assumptions of $u_\circ$,
then  applying boundary Harnack to $\partial_e u_{\circ,1}$ and $\partial_e u_{\circ,2}$ we deduce that $\partial_e u_{\circ,1}\equiv \kappa\partial_e u_{\circ,2}$ for some constant $\kappa \in \R.$ This proves that  $u_\circ$ is unique, up to multiplicative constant.

\vspace{2mm}

\noindent \emph{Case 2}.
Assume that the convex set $\{u_\circ=0\}$ does \emph{not} contain any convex cone with nonempty interior.
Then, exactly as in \cite{CRS} (see  Lemma \ref{lem-growth-par-blowdown} below, written in the parabolic setting, for a detailed argument), we can find a sequence $R_m\to\infty$ such that
\[u_m(x):=\frac{u_\circ(R_mx)}{R_m\|\nabla u_\circ\|_{L^\infty(B_{R_m})}}\]
satisfies 
\[\|\nabla u_m\|_{L^\infty(B_1)}=1,\quad \quad \|\nabla u_m\|_{L^\infty(B_R)}\leq 2R^{s+\alpha_\circ}\textrm{ for } R\geq1,\qquad \text{$\LL(D_h u_m)=0$ in $\{u_m>0\}$}.\]
By convexity, the nonnegative functions $u_m$ converge (up to a subsequence) locally uniformly to a nonnegative function $u_\infty$ that satisfies
\[\|\nabla u_\infty\|_{L^\infty(B_2)}\geq1\qquad \textrm{and}\qquad \|\nabla u_\infty\|_{L^\infty(B_R)}\leq 2R^{s+\alpha_\circ}\quad\textrm{for all}\quad R\geq1.\]
Also, since by assumption the convex set $\{u_\circ=0\}$ does not contain any cone with nonempty interior, its ``blow-down'' sequence $\{u_m>0\}=\frac{1}{R_m}\{u_\circ=0\}$ converges to a convex set $\Gamma$ that is contained in a hyperplane.
In particular
\[\LL(D_h u_\infty)=0\qquad \textrm{in}\quad \R^n\setminus \Gamma,\]
and since $D_h u_\infty\in {\rm Lip}_{\rm loc}(\R^n)$ it follows by Lemma \ref{lem-solution-across} that $\LL(D_h u_\infty)=0$ in $\R^n$. Hence, letting $h\to 0$, we conclude that
\[\LL(\nabla u_\infty)=0\qquad \textrm{in}\quad \R^n.\]
Thanks to the growth assumption \eqref{qhiohtiowhw}, it follows by Liouville Theorem that $u_\infty(x)=a\cdot x+b$.
However, this contradicts the fact that $u_\infty\geq 0$ and $\|\nabla u_\infty\|_{L^\infty(B_2)}\geq1$.
Thus, Case 2 cannot happen, and the proposition is proved.
\end{proof}

Once we have the classification of blow-ups, we can show the almost-optimal regularity of solutions.
However, we first need the following:

\begin{lem}\label{lem-1d-ell-nonhomogeneous}
Let $s \in(0,1)$, $\LL$ as in Definition \ref{L}, and $e\in \mathbb S^{n-1}$.
Then there exists $\gamma\in \left(0,\min\{2s,1\}\right)$, depending only on $n$, $s$, and the ellipticity constants, such that 
\[\LL(x\cdot e)_+^\gamma \leq 0 \qquad \textrm{in}\quad \{x\cdot e>0\}.\]
Moreover, when the kernel of the operator $\LL$ is even and homogeneous, we may take $\gamma=s$.
\end{lem}

\begin{proof}
When the kernel of the operator $\LL$ is even and homogeneous, the result is proved in \cite[Section~2]{RS-Duke}.
Hence, it suffices to prove the result in the case of general operators as in Definition \ref{L}.

After a rotation, we may assume $e=e_n$.
Let $\mathcal M_{s,\lambda,\Lambda}^-$ be the extremal operator associated to our class of operators, i.e., $\mathcal M_{s,\lambda,\Lambda}^- w:=\inf_L \LL w$, where the infimum is taken among all operators $\LL$ as in Definition \ref{L} (with fixed $s$, $\lambda$, $\Lambda$).
Then, the operator $\mathcal M_{s,\lambda,\Lambda}^-$ is scale invariant of order $2s$, and in particular $\mathcal M_{s,\lambda,\Lambda}^-(x_n)_+^\gamma = c_\gamma x_n^{\gamma-2s}$ in $\{x_n>0\}$ for $\gamma\in[0,2s)$.
Moreover, it is immediate to check that $c_0<0$, and therefore we have $c_\gamma<0$ for $\gamma>0$ small; see \cite[Section~2]{RS-Duke}.
Also, since $(x_n)_+^\gamma$ is convex for $\gamma \geq 1$, it follows that $c_\gamma > 0$ for $\gamma\geq 1.$
Hence, we proved that $\LL(x_n)_+^\gamma \leq c_\gamma x_n^{\gamma-2s}<0$ in $\{x_n>0\}$ for some $\gamma \in \left(0,\min\{2s,1\}\right)$, as desired.
\end{proof}

We also need the following:

\begin{lem}\label{lem-growth-ell}
Assume $w_k\in L^\infty(B_1)$ satisfy 
\[\sup_k \|w_k\|_{L^\infty(B_1)}<\infty \qquad \text{and}\qquad \sup_k\sup_{r\in(0,1)} \frac{\|w_k\|_{L^\infty(B_{r})}}{r^\mu}=\infty\]
for some $\mu \geq 0$.
Then, there are subsequences $w_{k_m}$ and $r_m\to0$ such that $\|w_{k_m}\|_{L^\infty(B_{r_m})} \geq r_m^\mu$ and for which the rescaled functions 
\[\tilde w_m(x):= \frac{w_{k_m}(r_m x)}{\|w_{k_m}\|_{L^\infty(B_{r_m})}}\]
satisfy
\[\big|\tilde w_m(x)\big| \leq 2\big(1+|x|^\mu\big)\qquad \textrm{in}\quad B_{1/r_m}.\]
\end{lem}

\begin{proof}
For every $m\in \mathbb N$, let $k_m\in\mathbb N$ and $r_m \in (\frac1m,1)$ be such that 
\[ r_m^{-\mu} \|w_{k_m}\|_{L^\infty(B_{r_m})} 
\geq \frac12 \sup_k \sup_{r\in( \frac1m,1)} r^{-\mu} \|w_k\|_{L^\infty(B_r)}
\geq \frac12 \sup_k \sup_{r\in (r_m,1)} r^{-\mu} \|w_k\|_{L^\infty(B_r)}.\]
Note that, since $\sup_k \|w_k\|_{L^\infty(B_1)}<\infty$ but 
$$
\sup_k \sup_{r\in( \frac1m,1)} r^{-\mu} \|w_k\|_{L^\infty(B_r)} \to \infty \qquad \text{as }m \to \infty,
$$
necessarily $r_m \to 0$ as $m \to \infty$.
Also, by construction of $r_m$ and $k_m$, we have
\[r_m^{-\mu}\|w_{k_m}\|_{L^\infty(B_{r_m})} \geq \frac12 r^{-\mu}\|w_{k}\|_{L^\infty(B_{r})} \quad \text{for all}\quad r\geq r_m,\  k\in \mathbb N.\]
In particular, for any $R\in (1,r_m^{-1})$ we have
\[\|\tilde w_m\|_{L^\infty(B_R)} = \frac{\|w_{k_m}\|_{L^\infty(B_{R r_m})}}{\|w_{k_m}\|_{L^\infty(B_{r_m})}} \leq 2R^\mu.\]
Since $\|\tilde w_m\|_{L^\infty(B_1)}=1$, the result follows.
\end{proof}

We can now establish the almost-optimal regularity of solutions.

\begin{cor}\label{cor-almost-optimal-ell}
Let $s \in(0,1)$ and $\LL$ as in Definition \ref{L}.
Let $\alpha_\circ\in (0,s)\cap (0,1-s)$, $\gamma\in \left(0,\min\{2s,1\}\right)$ given by Lemma~\ref{lem-1d-ell-nonhomogeneous}, and
$u\in {\rm Lip}_{\rm loc}(\R^n)$, with
$$
\|\nabla u\|_{L^\infty(B_R)} \le R^{s+\alpha_\circ} \qquad \textrm{for all}\quad R\ge 1,
$$
 satisfy
$$
\text{$u\geq0$ and $D^2 u\geq-{\rm Id}$ in $B_2$,}\quad  \text{$\LL u=f$ in $\{u>0\}\cap B_2$},\quad\text{and}\quad \text{$\LL u \leq f$ in $B_2$,}\quad\text{with $|\nabla f|\leq 1$.}
$$
Then, for any $\varepsilon>0$ we have
\[\|\nabla u\|_{C^{0,\gamma-\varepsilon}(B_1)} \leq C_\varepsilon 
,\]
with $C_\varepsilon$ depending only on $n$, $s$, $\varepsilon$, and the ellipticity constants.
\end{cor}

\begin{proof}
Let $\mu:=\gamma-\varepsilon$. Since $\gamma\in \left(0,\min\{2s,1\}\right)$, up to enlarging $\alpha_\circ$ we can assume that $s+\alpha_\circ \geq \mu$.

We will prove the existence of a constant $C>0$ such that, at every free boundary point $x_\circ\in \partial\{u>0\}\cap B_1$, we have
\[|\nabla u(x)|\leq C|x-x_\circ|^\mu.\]
This, combined with interior regularity estimates (see for instance  \cite{CS09,DRSV}), yields the desired result.

Assume by contradiction that such estimate fails.
Then, we can find sequences $u_k$, $\LL_k$, and $f_k$, satisfying the assumptions, with $0\in \partial\{u_k>0\}$, such that 
\[\sup_k\sup_{r\in (0,1)} \frac{\|\nabla u_k\|_{L^\infty(B_{r})}}{r^\mu}=\infty.\]
Note that, by the uniform semiconvexity estimate $D^2u_k \geq -{\rm Id}$ in $B_2$, the functions $u_k$ are uniformly Lipschitz inside $B_1$.
Hence, thanks to Lemma \ref{lem-growth-ell}, there exists sequences $k_m$ and $r_m\to0$ such that $\|\nabla u_{k_m}\|_{L^\infty(B_{r_m})} \geq r_m^\mu$ and the functions
\[\tilde u_m(x) := \frac{u_{k_m}(r_m x)}{r_m\|\nabla u_{k_m}\|_{L^\infty(B_{r_m})}},\qquad 
\nabla \tilde u_m(x) := \frac{\nabla u_{k_m}(r_m x)}{\|\nabla u_{k_m}\|_{L^\infty(B_{r_m})}},\]
satisfy $\|\nabla \tilde u_m\|_{L^\infty(B_1)}=1$ and
\[|\nabla \tilde u_m(x)| \leq 2(1+|x|^\mu)\qquad\textrm{in}\quad B_{1/r_m}.\]
Moreover, we also have
\begin{equation}\label{eq:semiconvex m}
D^2 \tilde u_m \geq - r_m^{1-\mu}{\rm Id} \longrightarrow 0 \qquad \text{in}\quad B_{2/r_m},
\end{equation}
$$
\|\nabla \tilde u_m\|_{L^\infty(B_R)} \le r_m^{s+\alpha_\circ-\mu}R^{s+\alpha_\circ}
\le R^{s+\alpha_\circ} \qquad \textrm{for all}\quad R\ge r_m^{-1},
$$
(recall that $s+\alpha_\circ \geq \mu$) and
$$
\text{$\LL_{k_m} \tilde u_m= f_m$ in $\{u_m>0\}\cap B_{2/r_m}$,}\quad \text{$\LL_{k_m} \tilde u_m\leq  f_m$ in $B_{2/r_m}$,}\quad\text{with $|\nabla f_m|\leq r_m^{2s-\mu}\to 0$.}
$$
In particular, the last two conditions imply that $\LL_{k_m}(D_h \tilde u_m)\geq0$ in $\{\tilde u_m>0\}\cap B_{1/r_m}$, where
$$
D_h \tilde u_m(x)={\textstyle \frac{\tilde u_m(x)-\tilde u_m(x-h)}{|h|}}.
$$
Hence, thanks to \eqref{eq:semiconvex m}, up to a subsequence the functions $\tilde u_m$ will converge locally uniformly in $\R^n$ to a limiting convex function $\tilde u_\circ$
satisfying
$$
\|\nabla \tilde u_\circ\|_{L^\infty(B_2)}\geq 1,\qquad \|\nabla \tilde u_\circ\|_{L^\infty(B_R)} \leq 3R^\mu \qquad \textrm{for all}\quad R\ge 1.
$$
Moreover, using for instance \cite[Lemma 3.2]{DRSV} to take the limit in the equations, we see that $u_\circ$ will satisfy the hypotheses of Proposition \ref{prop-classification-ell}, and therefore it must be a 1D function, say $\tilde u_\circ(x)=U(x\cdot e)$.

Hence, if we consider the 1D function $w:=U'\geq 0$, we see that
$$
\text{$\LL w=0$ in $(0,\infty)$}\quad \text{and}\quad \text{$w=0$ in $(-\infty,0)$},\qquad \text{for some 1D operator $\LL$ as in Definition \ref{L}.}
$$
Also, since $w(t)\leq C(1+t_+)^\mu $, for any $\delta>0$ small we see that $w(t)\leq \delta (t_+)^\gamma$ for $t\geq C\delta^{-1/\varepsilon}$.
Recalling that $\LL(t_+)^\gamma\geq0$ in $(0,\infty)$  (see Lemma \ref{lem-1d-ell-nonhomogeneous}), we can apply the comparison principle in \cite[Lemma 4.1]{DRSV} to deduce that $w\leq \delta(t_+)^\gamma$ on the whole $\R$.
Since $\delta>0$ is arbitrary, this implies that  $w\leq0$, and hence $w\equiv0$ in $\R$.
However, this means that $U\equiv0$ in $\R$ and therefore $\tilde u_\circ\equiv0$ in $\R^n$, a contradiction.
\end{proof}

The next step consists in showing that the free boundary is $C^{1,\tau}$ near nondegenerate points.
To prove this result, we need the following result:

\begin{lem} \label{lem-ell-positivity}
Let $s\in(0,1)$, $\LL$ as in Definition \ref{L}, $\alpha_\circ\in(0,s)$, and $c_\circ>0$.
Then there exist $R_\circ\geq1$ large and $\eps_\circ>0$ small, depending only on $n$, $s$, $\lambda$, $\Lambda$, $c_\circ$, and~$\alpha_\circ$, such that the following holds.

Assume that $E\subset\R^n$ is closed, and $v\in C(\R^n)$ satisfies (in the viscosity sense)
\[\LL v\leq \eps\quad \textrm{in}\quad B_{R_\circ}\setminus E , \qquad v\equiv0\quad \textrm{in}\quad B_{R_0}\cap E,\]
\[ \int_{B_{R_\circ}}v_+ \geq c_\circ>0, \qquad v \geq -\eps_\circ \quad \textrm{in}\quad B_{R_\circ},\qquad \text{and}\qquad |v(x)|\leq |x|^{s+\alpha_\circ}\quad \textrm{in}\quad \R^n\setminus B_{R_\circ}.\]
Then $v\geq0$ in $B_{R_\circ/2}$.
\end{lem}

\begin{proof}
The proof is the same as that of \cite[Lemma 6.2]{CRS}.
\end{proof}

We can now show the $C^{1,\tau}$ regularity of free boundaries.

\begin{lem}\label{lemm-main-ell-2}
Let $s$, $\LL$, $\alpha_\circ$, and $u$, be as in Theorem \ref{thm-main-ell}.
There, for any given $\kappa_\circ>0$, there exist $R_\circ>1$ large and $\eps_\circ>0$ small for which the following holds.

Let $u_\circ(x)=U(x\cdot e)$, $e \in \mathbb S^{n-1}$, be a nonnegative, convex, 1D solution of \eqref{1D-profile}.
Assume that $\|u_\circ\|_{{\rm Lip} (B_{1})}\geq \kappa_\circ>0$ and  $\|u- u_\circ\|_{{\rm Lip} (B_{R_\circ})} \le \eps_\circ$.
Then the free boundary $\partial \{u>0\}$ is a $C^{1,\tau}$ graph in $B_{1/2}$.
The constants $R_\circ,\eps_\circ$ and the bounds on  the $C^{1,\tau}$ norm depend only on  $n$, $s$, $\alpha_\circ$, $\lambda$, $\Lambda$, and $\kappa_\circ$.
\end{lem}

\begin{proof}
By assumption, for any direction $e'\in \mathbb S^{n-1}$ such that $e'\cdot e\geq\frac12$ we have
\[|\partial_{e'} u-\partial_{e'} u_\circ|\leq \eps\qquad\textrm{in}\quad B_{R_\circ}.\]
Also,
\[\partial_{e'} u_\circ \geq0 \quad \textrm{in}\quad \R^n \qquad \textrm{and} \qquad \partial_{e'} u_\circ \geq c_1\kappa_\circ \quad \textrm{in}\quad \{x\cdot e\geq {\textstyle\frac12}\}.\]
Thus, if $\eps_\circ$ is sufficiently small, we have that $v:=\partial_{e'} u$ and $E:=\{u=0\}$ satisfy:
\[|\LL v|\leq \eta\quad \textrm{in}\quad B_{R_\circ}\setminus E, \qquad v\equiv0\quad \textrm{in}\quad B_{R_\circ}\cap E,\]
\[ v \geq c_2\kappa_\circ \quad \textrm{in}\quad \{x\cdot e\geq {\textstyle\frac12}\}\cap B_{R_\circ},\qquad v \geq -\eps_\circ \quad \textrm{in}\quad B_{R_\circ},
\qquad \text{and}\quad |v(x)|\leq |x|^{s+\alpha_\circ}\quad \textrm{in}\quad \R^n\setminus B_{R_\circ}.\]
Hence, choosing $R_\circ$ large enough, it follows from Lemma \ref{lem-ell-positivity}  that $v\geq 0$ in $B_{R_\circ/2}$, i.e.,
\[\partial_{e'} u\geq0 \quad \textrm{in}\quad B_{R_\circ/2}\qquad \text{for all $e'\in \mathbb S^{n-1}$ such that $e'\cdot e\geq\frac12$.}\]
This means that the free boundary $\partial\{u>0\}$ is a Lipschitz graph in $B_{R_\circ/2}$, with Lipschitz constant bounded by 1, which allows us to 
apply Theorem~\ref{ellbdryharnack} to the functions $(\partial_{e'} u)_+$ and $(\partial_{e} u)_+$ to deduce that 
\[\left\|\frac{\partial_{e'} u}{\partial_{e} u}\right\|_{C^\tau(B_{1/2})} \leq C.\]
Choosing $e=e_n$ and $e'=e_n+e_i$ for $i=1,...,n-1$, we conclude that the free boundary $\partial\{u>0\}$ is a $C^{1,\tau}$ graph in $B_{1/2}$, as wanted.
\end{proof}

Finally, we need the following expansion for solutions to elliptic equations in $C^{1,\tau}$ domain (recall Proposition~\ref{prop-classification-ell} for the uniqueness of 1D solutions).

\begin{lem}\label{lemm-main-ell-3}
Let $s$, $\LL$, $\alpha_\circ$, and $u$, be as in Theorem \ref{thm-main-ell}.
Suppose in addition that the kernel $K$ of the operator $\LL$ is homogeneous.

Assume that $\partial\{u>0\}$ is a $C^{1,\tau}$ graph in $B_{1/2}$, with $\nu(0)=e$, and 
let $u_\circ$ be the unique nonnegative, convex 1D solution satisfying \eqref{1D-profile}.
Then $u_\circ$ is homogeneous of degree $1+\gamma$, with $\gamma\in (0,2s)\cap (2s-1,1)$ depending only on $\LL$ and $e$.
Moreover
\[|\nabla u-\nabla u_\circ|\leq C|x|^{\gamma+\tau'}\quad \textrm{for}\quad x\in B_{1/4},\]
with $C$ and $\tau'>0$ depending only on  $n$, $s$, $\alpha_\circ$, $\lambda$, $\Lambda$, $\tau$, and the $C^{1,\tau}$ norm of the graph.
\end{lem}

\begin{proof}
The uniqueness and homogeneity of $u_\circ$ follow from Proposition~\ref{prop-classification-ell}, while the explicit expression of $\gamma$ is proved in \cite[Corollary~4.6]{DRSV}.
The expansion for $\nabla u$ then follows from \cite[Theorem~1.2]{DRSV}.
\end{proof}

Combining the previous results, we can finally prove Theorem \ref{thm-main-ell}.

\begin{proof}[Proof of Theorem \ref{thm-main-ell}]
Let us first prove that, given any $R_\circ\geq1$ and $\eps>0$, for $\eta>0$ small enough, we have that
\begin{equation}\label{contradiction-ell}
\|u-u_\circ\|_{{\rm Lip}(B_{R_\circ})} \leq\eps,
\end{equation}
for some nonnegative, convex, 1D function $u_\circ$ satisfying \eqref{1D-profile}.

Indeed, assuming by contradiction that this is false, we can find  sequences $\eta_k\to0$, operators $\LL_k$, and solutions $u_k$, such that $u_k$ satisfy the hypotheses of the statement but
\begin{equation}\label{contradiction-ell2}
\|u_k-u_\circ\|_{{\rm Lip}(B_{R_\circ})} \geq\eps
\end{equation}
for any $e\in \mathbb S^{n-1}$ and any solution $u_\circ$ of \eqref{1D-profile}.
But then, by Corollary \ref{cor-almost-optimal-ell} and \cite[Lemma 3.2]{DRSV}, up to a subsequence the functions $u_k$ converges in $C^1$ norm in compact sets to a limiting function $u_\infty$ that satisfies the same conditions with $\eta=0$.
Then,  Proposition \ref{prop-classification-ell} implies that $u_\infty$ is a 1D function satisfying \eqref{1D-profile},
which means that we can take $u_\circ=u_\infty$ in \eqref{contradiction-ell2}, a contradiction.
Hence,  \eqref{contradiction-ell} is proved.

Thanks to  \eqref{contradiction-ell}, the $C^{1,\tau}$ regularity of the free boundary follows from Lemma \ref{lemm-main-ell-2}, while the expansion for $\nabla u$ (and hence $u$) at 0 follow from Lemma \ref{lemm-main-ell-3} (taking $\tau$ smaller, if necessary).
\end{proof}

\subsection{Iteration and optimal regularity of solutions}

We now show how to use Theorem \ref{thm-main-ell} to establish optimal regularity estimates for solutions, namely Corollary \ref{thm-sol-1}.
We will actually prove a finer result, which gives a \emph{uniform} estimate of order $1+s+\alpha_\circ$ at all free boundary points.

\begin{cor}\label{cor-sol-ell-1+s+a}
Let $\LL$ be an operator of the form \eqref{L0}-\eqref{ellipt} with $K$ homogeneous, and let $\alpha_\circ\in (0,s)\cap (0,1-s)$.
Let $\varphi$ be an obstacle satisfying \eqref{obstacle}, and $u$ be the solution to~\eqref{obst-pb-ell}.

Then, for every free boundary point $x_\circ\in \{u>\varphi\}$, there exist $c_\circ\geq0$ and $e\in \mathbb S^{n-1}$ such that 
\[\qquad \qquad \left|u(x)-\varphi(x)-c_\circ\big((x-x_\circ)\cdot e\big)_+^{1+s}\right| \leq CC_\circ|x-x_\circ|^{1+s+\alpha_\circ}\quad\textrm{for}\quad x\in B_1(x_\circ),\]
with $C$ depending only on $n$, $s$, $\lambda$, $\Lambda$, and $\alpha_\circ$.

Moreover, if $c_\circ>0$, then the free boundary is a  $C^{1,\alpha_\circ}$ graph in a ball $B_{\rho_\circ}(x_\circ)$, with $C\rho_\circ^{\alpha_\circ}\geq c_\circ$ and $C$ depending only on $n$, $s$, $\lambda$, $\Lambda$, and $\alpha_\circ$.

Finally, we have $u\in C^{1+s}(\R^n)$ and
\[\|\nabla u\|_{C^{s}(\R^n)} \leq CC_\circ,\]
with $C$ depending only on $n$, $s$, $\lambda$, and $\Lambda$.
\end{cor}

\begin{rem}
The result above provides a uniform estimate at {all} free boundary points, which in turn yields a quantitative  (and sharp) relation between the constant $c_\circ$ (quantitative nondegeneracy) and the radius of the ball where the free boundary is smooth (quantitative regularity of the free boundary).
Furthermore, the above expansion for $u$ can be used to prove in addition that 
\[\|(u-\varphi)/d^{1+s}\|_{C^{\alpha_\circ}(\R^n)} + \|\nabla (u-\varphi)/d^s\|_{C^{\alpha_\circ}(\R^n)} \leq CC_\circ,\]
where $d$ is the distance to the free boundary.
We leave the details to the interested reader.
\end{rem}

\begin{proof}[Proof of Corollary \ref{cor-sol-ell-1+s+a}]
Dividing the solution and the obstacle by a constant, if necessary, and up to a translation, we may assume $C_\circ=1$ and $x_\circ=0$.
Moreover, exactly as in \cite{CRS}, we may consider $u\mapsto u-\varphi$, so that $u$ now satisfies:
\begin{equation}
\label{eq:u-phi}
\begin{split}
u\geq0,\quad &\quad D^2 u \ge - C_1{\rm{Id}}\qquad \textrm{in}\quad \R^n,\qquad \|\nabla u\|_{L^\infty(\R^n)} \le C_1,\qquad u(0)=|\nabla u(0)|=0,\\
&\LL u=f\quad \textrm{in}\quad \{u>0\} \quad \textrm{and} \quad \LL u\leq f \quad \textrm{in} \quad \R^n, \quad \textrm{with} \quad  |\nabla f|\leq C_1.
\end{split}
\end{equation}
We now want to apply Theorem \ref{thm-main-ell} iteratively in order to get the desired estimate.

Consider $\kappa_\circ>0$ to be chosen later, and let $\eps_\circ>0$ be the constant given by Theorem~\ref{thm-main-ell}.
For $\eta>0$ small, define the functions
\[w_k(x):=\frac{\eta}{C_1}\frac{u(2^{-k} x)}{(2^{-k})^{1+s+\alpha_\circ}},\qquad k=0,1,2,\ldots\]
Since $1+s+\alpha_\circ<2$ and $s+\alpha_\circ<2s$, it follows that all functions $w_k$ satisfy
\begin{equation}
\label{eq:u-phi k}
\begin{split}
w_k\geq0\quad \textrm{and}& \quad D^2 w_k \ge - \eta{\rm{Id}}\qquad \textrm{in}\quad \R^n,\quad \textrm{with}\quad w_k(0)=|\nabla w_k(0)|=0,\\
\LL w_k=f_k\quad \textrm{in}&\quad \{w_k>0\} \quad \textrm{and} \quad \LL w_k\leq f_k \quad \textrm{in} \quad \R^n, \quad \textrm{with} \quad  |\nabla f_k|\leq \eta.
\end{split}
\end{equation}
Moreover, when $k=0$ we  have
$\|\nabla w_0\|_{L^\infty(\R^n)} \le 1.$

In other words, all the assumptions of Theorem \ref{thm-main-ell}, possibly except for the growth control on $\|\nabla w_k\|_{L^\infty(B_R)}$ (that holds at least for $k=0$), are satisfied by $w_k$.
We then have two possibilities:

\vspace{2mm}

\noindent \emph{Case 1}. Assume that functions $w_k$ satisfy
\[\|\nabla w_k\|_{L^\infty(B_R)} \leq R^{s+\alpha_\circ}\quad \mbox{ for $R\ge1$}, \qquad \textrm{for all}\quad k\geq0.\]
Then, we have
\[\|\nabla u\|_{L^\infty(B_{2^{-k}})} = C_1\eta^{-1}(2^{-k})^{s+\alpha_\circ}\|\nabla w_k\|_{L^\infty(B_1)} \leq  C(2^{-k})^{s+\alpha_\circ},\]
and therefore
\[|\nabla u(x)|\leq C|x|^{s+\alpha_\circ}\quad \textrm{for}\quad x\in B_1.\]
This, in turn, implies that 
\[|u(x)|\leq C|x|^{1+s+\alpha_\circ}\quad \textrm{for}\quad x\in B_1,\]
as wanted.

\vspace{2mm}

\noindent \emph{Case 2}. If we are not in Case 1, then there is a maximal number $k_\circ\in \mathbb N$ such that 
\begin{equation}\label{growth-ell-contr}
\|\nabla w_k\|_{L^\infty(B_R)} \leq R^{s+\alpha_\circ}\quad \mbox{ for $R\ge1$},\qquad \textrm{for all}\quad k\leq k_\circ.
\end{equation}
In particular, in terms of $u$, this implies that
\begin{equation}\label{growth-ell-contr2}
|u(x)| \leq C|x|^{1+s+\alpha_\circ} \qquad \text{for all }x \in B_1\setminus B_{2^{-k_\circ}}.
\end{equation}
We now observe that, thanks to \eqref{growth-ell-contr}, choosing $\eta$ sufficiently small Theorem \ref{thm-main-ell} implies that
\[\|w_{k_\circ}-u_\circ\|_{\rm Lip(B_1)}\leq \eps : = \min \{\eps_\circ, 1/6\},\]
where $u_0$ is a multiple of $(x\cdot e)_+^{1+s}$, that is
\[|\nabla u_\circ(x)|=\kappa(x\cdot e)^s_+, \qquad \textrm{for some }\quad 0\leq \kappa \leq 2.\]
We consider two subcases:\\
(i)\, If $\kappa\leq \frac13$, then by triangle inequality
\[\|\nabla w_{k_\circ}\|_{L^\infty(B_1)}\leq \|\nabla u_\circ\|_{L^\infty(B_1)} + \eps\leq {\textstyle \frac13} + \eps \leq \frac12 <  2^{-s-\alpha_\circ}.\]
Since $\nabla w_{k_\circ+1}(x)=2^{s+\alpha_\circ}\nabla w_{k_\circ}(\frac{x}{2})$, this implies that 
\[\|\nabla w_{k_\circ+1}\|_{L^\infty(B_2)}\leq 1.\]
Since \[\|\nabla w_{k_\circ+1}\|_{L^\infty(B_R)} = 2^{s+\alpha_\circ}\|\nabla w_{k_\circ}\|_{L^\infty(B_{R/2})}  \leq R^{s+\alpha_\circ}\quad \textrm{for}\quad R\geq2,\]
then $w_{k_\circ+1}$ still satisfies the growth condition \eqref{growth-ell-contr}, a contradiction to the definition of $k_\circ$.\\
(ii)\, If instead $\kappa \geq \frac13$, it follows from Theorem  \ref{thm-main-ell} that the free boundary $\partial\{w_{k_\circ}>0\}$ is a $C^{1,\tau}$ graph in $B_1$ and 
\begin{equation}\label{jronya}|\nabla w_{k_\circ}(x)-\nabla u_\circ(x)|\leq C|x|^{s+\tau}\qquad \text{for all}\quad x \in B_1.
\end{equation}
Furthermore, as in \cite{CRS}, we can apply the boundary Harnack estimate in $C^1$ domains from \cite{RS-C1} to deduce that the regularity of the free boundary can be improved to $C^{1,\alpha_\circ}$. Hence, applying the corresponding estimates in $C^{1,\alpha_\circ}$ domains from \cite{RS-C1}, we finally obtain that \eqref{jronya} holds with $\tau=\alpha_\circ$.
This, in turn, implies 
\[\big|w_{k_\circ}(x)-{\textstyle\frac{\kappa}{1+s}}(x\cdot e)_+^{1+s}\big|\leq C|x|^{1+s+\alpha_\circ}\qquad \text{for all}\quad x \in B_1,\]
and rescaling back to $u$ we find
\begin{equation}\label{growth-ell-contr3}
\big|u(x)-c_\circ(x\cdot e)_+^{1+s}\big|\leq C|x|^{1+s+\alpha_\circ}\qquad \text{for all}\quad x \in B_{2^{-k_\circ}},\qquad \text{with}\quad c_\circ = \frac{C_1\kappa }{\eta(1+s)} 2^{-\alpha_\circ k_\circ},
\end{equation}
and that the free boundary $\partial\{u>0\}$ is $C^{1,\alpha_\circ}$ in a ball of radius $2^{-k_\circ}$.
Note that, since 
$$
|c_\circ(x\cdot e)_+^{1+s}|\leq c_\circ |x|^{1+s}=\frac{C_1\kappa }{\eta(1+s)} 2^{-\alpha_\circ k_\circ} |x|^{1+s} \leq C|x|^{1+s+\alpha_\circ} \quad \text{for}\quad |x|\geq 2^{-k_\circ},
$$ 
it follows from \eqref{growth-ell-contr2} and \eqref{growth-ell-contr3} that
$$
\big|u(x)-c_\circ(x\cdot e)_+^{1+s}\big|\leq C|x|^{1+s+\alpha_\circ}\qquad \text{for all}\quad x \in B_{1},
$$
proving the result also in Case 2(ii).

Finally, to conclude the proof, it suffices to observe that in all cases we have 
\[|\nabla u(x)|\leq C|x|^{s}\qquad \textrm{for}\quad x\in B_1,\]
and this implies the uniform $C^{1+s}$ estimate for $u$.
\end{proof}

We now show that the same argument as above can be adapted to the case of non-symmetric operators.
In this case we establish directly the optimal regularity of solutions, without passing through the expansion for $u$ in terms of $u_\circ$.
We recall that, when $\LL$ is as in Definition \ref{L} with $K$ homogeneous, then it has a Fourier symbol $\mathcal A(\xi)+i\mathcal B(\xi)$ associated to it. 
We refer to \cite[(2.6)-(2.7)]{DRSV} for the explicit expression of $\mathcal A$ and $\mathcal B$.

\begin{cor}\label{cor-sol-nonsymm}
Let $s\in (0,1)$ and $\LL$ as in Definition \ref{L}, with $K$ homogeneous.
Let $\varphi$ be an obstacle satisfying \eqref{obstacle}, and $u$ be the solution of \eqref{obst-pb-ell}.
Let $\mathcal A(\xi)+i\mathcal B(\xi)$ be the Fourier symbol of $\LL$, and define
\[\gamma_\LL:=\min_{e\in \mathbb S^{n-1}} \gamma_{\LL,e},\qquad \textrm{where}\qquad \gamma_{\LL,e}:= s-\frac{1}{\pi}\arctan\left(\frac{\mathcal B(e)}{\mathcal A(e)}\right).\]
Then $u\in C^{1+\gamma_\LL}(\R^n)$ with
\[\|u\|_{C^{1+\gamma_\LL}(\R^n)} \leq CC_\circ,\]
with $C$ depending only on $n$, $s$, $\lambda$, and $\Lambda$.
\end{cor}

\begin{proof}[Proof of Corollary \ref{cor-sol-nonsymm}]
As in the proof of Corollary \ref{cor-sol-ell-1+s+a}, we may assume $C_\circ=1$ and $0 \in \partial \{u>0\}$. Also
we may consider $u\mapsto u-\varphi$ so that \eqref{eq:u-phi} holds.

We now want to apply Theorem \ref{thm-main-ell} iteratively in order to prove
\[|u(x)|\leq C|x|^{1+\gamma_\LL}\qquad \textrm{for}\quad x\in B_1.\]
Notice that, in the current setting, for any $e\in \mathbb S^{n-1}$ the solution $u_\circ$ to \eqref{1D-profile} is a multiple of $(x\cdot e)_+^{1+\gamma_{\LL,e}}$,
where the explicit expression for $\gamma_{\LL,e}$ is given by \cite[Corollary 4.6]{DRSV}.
In particular, by definition of $\gamma_\LL$, we have $\gamma_{\LL,e}\geq \gamma_\LL$ for all $e\in \mathbb S^{n-1}$.

Let $\kappa>0$ to be chosen later, and let $\eta>0$ be the constant given by Theorem~\ref{thm-main-ell}.
We define the functions
\[w_k(x):=\frac{\eta}{C_1}\frac{u(2^{-k} x)}{(2^{-k})^{1+\gamma_\LL}},\]
for $k=0,1,2,...,$.
Notice that, since $\gamma_\LL<\min\{1,2s\}$, the functions $w_k$ satisfy \eqref{eq:u-phi k}.
Moreover, when $k=0$ we have $\|\nabla w_0\|_{L^\infty(\R^n)} \le 1.$

Now, as in the proof of Corollary \ref{cor-sol-ell-1+s+a}, we consider two cases: 
if 
\[\|\nabla w_k\|_{L^\infty(B_R)} \leq R^{\gamma_\LL}\qquad \textrm{for all}\quad k\geq0,\]
then in terms of $u$ this implies that $|\nabla u(x)|\leq C|x|^{\gamma_\LL}$ in $B_1$, hence

\[|u(x)|\leq C|x|^{1+\gamma_\LL}\qquad \textrm{for}\quad x\in B_1,\]
as desired.

Alternatively, assume there is a maximal number $k_\circ\in \mathbb N$ such that 
\begin{equation}\label{growth-ell-contr bis}
\|\nabla w_k\|_{L^\infty(B_R)} \leq R^{\gamma_\LL}\qquad \textrm{for all}\quad k\leq k_\circ.
\end{equation}
In particular, in terms of $u$, this implies that
\begin{equation}\label{growth-ell-contr2 bis}
|u(x)| \leq C|x|^{1+\gamma_\LL} \qquad \text{for all }x \in B_1\setminus B_{2^{-k_\circ}}.
\end{equation}
Also, by Theorem \ref{thm-main-ell}, choosing $\eta$ sufficiently small we find
\[\|w_{k_\circ}-u_\circ\|_{\rm Lip(B_1)}\leq \eps \ll 1 ,\qquad |\nabla u_\circ(x)|=A(x\cdot e)^{\gamma_{\LL,e}}_+,\quad 0 \leq A \leq 2.\]
We then have two possibilities:\\
(i)\, If $A\leq \frac13$, then by triangle inequality
\[\|\nabla w_{k_\circ}\|_{L^\infty(B_1)}\leq \|\nabla u_\circ\|_{L^\infty(B_1)} + \eps\leq {\textstyle \frac13} + \eps \leq \frac12 <  2^{-\gamma_\LL},\]
which implies $\|\nabla w_{k_\circ+1}\|_{L^\infty(B_2)}\leq 1$. Hence $w_{k_\circ+1}$ satisfies the growth condition \eqref{growth-ell-contr bis}, contradicting the definition of $k_\circ$.\\
(ii)\, If $A\geq \frac13$, then by Theorem  \ref{thm-main-ell} we have that the free boundary $\partial\{w_{k_\circ}>0\}$ is a $C^{1,\tau}$ graph in $B_1$ and 
\[|\nabla w_{k_\circ}(x)-\nabla u_\circ(x)|\leq C|x|^{\gamma_{\LL,e}+\tau}\qquad \textrm{for}\quad x\in B_1.\]
In particular $|\nabla w_{k_\circ}(x)|\leq C|x|^{\gamma_\LL}$ in $B_1$, that rescaled back yields 
\[|\nabla u(x)|\leq C|x|^{\gamma_\LL}\qquad \textrm{for}\quad x\in B_{2^{-k_\circ}}.\]
Recalling \eqref{growth-ell-contr2 bis}, this concludes the proof.
\end{proof}

Thanks to the previous result, we finally deduce the validity of Corollaries \ref{thm-sol-1} and \ref{thm-sol-4}.

\begin{proof}[Proof of Corollaries \ref{thm-sol-1} and \ref{thm-sol-4}]
Both results are particular cases of Corollary \ref{cor-sol-nonsymm}.
\end{proof}

\section{A parabolic boundary Harnack inequality}
\label{sec3}

In this section we prove a parabolic boundary Harnack inequality in Lipschitz (and also more general) domains.
More precisely, we consider domains satisfying the following definition:

\begin{defi}
We say that a domain $\Omega\subset\R^n\times (-\infty,0)$ satisfies the \emph{interior cone condition  at $(0,0)$ with opening $\theta$ and speed $\omega$} if for some direction $e\in \mathbb S^{n-1}$ there is a ``traveling cone'' of the form:
\[\Sigma_t = \{|x\cdot e| >\cos \theta |x|\} - \omega t e ,\]
with opening angle  $\theta \in (0,\pi/2)$ and speed $\omega>0$, such that $(\Sigma_t \cap B_1)\times \{t\}  \subset \Omega$ all $t<0$ (i.e., for all $t\in (-1/\omega, 0)$).

We say that $\Omega$ satisfies the {\it interior cone condition with opening $\theta$ and speed $\omega$}  in  $Q\subset\R^n \times \R $ if, for all $(x_\circ, t_\circ)\in \Omega \cap Q$, the translation $\Omega-(x_\circ, t_\circ)$ satisfies the previous condition.
\end{defi}

A key result in this paper is the following parabolic boundary Harnack.

\begin{thm}\label{bdryharnack}
For any given $n\ge1$,   $s \in[\frac12,1)$, and positive constants $\lambda\leq\Lambda$ (ellipticity), $\theta\in (0,\pi/2)$ and $\omega>0$  (opening and speed of traveling cone),   $t_\oo>0$, and $\gamma_\oo \in (0,2s)$,
there exist positive constants $R$, $\varepsilon$, $\alpha\in (0,1)$, and $C$, such that the following statement holds.

Suppose that $\LL$ is as in \eqref{L0}-\eqref{ellipt} and let $A\subset \R^n\times [-2t_\oo,0]$  be a closed set such that $A^c\cap Q_1$ satisfies the interior cone condition (with opening $\theta$ and speed $\omega$).
Let $v_i$, $i=1,2,$ be two viscosity solutions of
\[
\big|(\partial_t -\LL) v_i  \big| \le \varepsilon \quad \mbox{in } A^c \cap \big(B_{R}\times(-2t_\oo, 0)\big),  \qquad v_i \equiv 0 \quad \mbox{in  }A \cap \big( B_{R}\times(-2t_\oo, 0)\big),
\]
satisfying
\[ v_i \ge -\varepsilon \quad \mbox{in } B_{R}\times(-2t_\oo, 0),\qquad
|v_i (x,t )| \le C_\oo (1+|x|)^{2s-\gamma_\oo}   \quad \mbox{in }\R^n\times(-2t_\oo,0),\qquad v_i(e_n,0)=1.
\]
Then, setting $Q_1:=B_1 \times (- t_\oo,0)$, we have 
\[
v_i >0\quad \text{in $A^c \cap Q_1$},\qquad \left[\frac{v_1}{v_2}\right]_{C^{\alpha}(A^c\cap Q_1)} + \left[\frac{v_2}{v_1}\right]_{C^{\alpha}(A^c\cap Q_1)} \le  C.
\]
\end{thm}

To prove it, we will need several ingredients.
The main step will be the following.

\begin{prop}\label{step}
For any given $n\ge1$,   $s \in[\frac12,1)$, and positive constants $\lambda\leq \Lambda$ (ellipticity), $\gamma_\oo\in (0,2s)$, $C_\oo$, and  $t_\oo$, there exist positive constants  $R$, $\varepsilon$, and $C$ such that the following statement holds.

 Suppose that $\LL$ is as in \eqref{L0}-\eqref{ellipt} and let $A\subset \R^n\times \R$  be a closed set satisfying
\[
B_{2\delta}(e_n)\times (-2t_\oo, 0) \subset A^c
\]
for some $\delta>0$.
Let $\rho\ge R$, and let $v_i$, $i=1,2$, be two viscosity solutions of
\[
\big|(\partial_t -\LL) v_i  \big| \le \rho^{-\gamma_\oo} \quad \mbox{in } A^c \cap \big(B_{\rho}\times(-2t_\oo, 0) \big),  \qquad v_i \equiv 0 \quad \mbox{in  }A \cap \big(B_{\rho}\times(-2t_\oo, 0) \big)
\]
satisfying
\[ v_i \ge -\varepsilon  (1+|x|^2)^{\frac{2s-\gamma_\oo}{2}} \quad \mbox{in } B_{\rho}\times(-2t_\oo, 0),\qquad 
|v_i (x, t)| \le  (1+|x|^2)^{\frac{2s-\gamma_\oo}{2}}  \quad \mbox{in }(\R^n\setminus B_{\rho})\times(-2t_\oo,0),\]
and $v_i(e_n,0) =c_\oo>0$.
Then
\[
\frac 1 C \le v_i \quad \mbox{in } Q^* : = B_{\delta}(e_n)\times (-5t_\oo/4, 0), \quad \quad  0\le v_i\le C \quad \mbox{in } Q_{1}
\]
and
\[
0<v_1 \le C v_2,\quad  0<v_2\le C v_1    \qquad \mbox{ in } Q_{1},
\]
where $Q_{1}:=B_{1} \times (- t_\oo,0)$.
\end{prop}

To prove  Proposition \ref{step} we need the following auxiliary results:

\begin{lem}[Supersolution]\label{supersol}
Let $s \in[\frac12,1)$, $\LL$ as in \eqref{L0}-\eqref{ellipt}, and
$\gamma_\oo \in (0,2s)$. 
Given $R \ge 1$, there exists a solution $S^1$ of
\[
(\partial_t -\LL) S^1 =R^{-\gamma_\circ}  \qquad \mbox{in } B_R\times (-1,0)
\]
satisfying
\[
S^1(x,t) \le CR^{-\gamma_\oo} \quad \mbox{in } B_{R/4}\times (-1,0),\qquad
S^1 (x,t) \ge  c |x|^{2s-\gamma_\oo} \chi_{\R^n\setminus B_R}  (x),
\]
for some  positive constants $c, C$ depending only on $n$, $s$, $\lambda$, $\Lambda$, and $\gamma_\oo$.
\end{lem}

\begin{proof}
We take  $S^1(x,t) := h(x,t) +R^{-\gamma_\oo} (t+1)$, where $h$ solves
\[
(\partial_t -\LL) h = 0  \quad \mbox{in } \R^n\times (-1,0)
\]
with initial condition $h(x,-1) = |x|^{2s-\gamma_\circ} \chi_{\R^n\setminus B_{R/2}}$ in $\R^n$.
Since the heat kernel $H$ of the operator $\LL$ satisfies 
\begin{equation}\label{HKbounds}
\frac 1 {C} \le  \frac{H(z,t)}{t^{\frac{n}{2s}}+t^{-1}|z|^{n+2s}} \le C
\end{equation}
(see for instance \cite{basslevin, KaKK21}),
for $x\in B_{R/4}$ and $t\in(-1,0)$ we obtain
\[h(x,t)= \int_{\R^n\setminus B_{R/2}} H(x-y,t)|y|^{2s-\gamma_\circ}dy\leq C\int_{\R^n\setminus B_{R/2}} \frac{|y|^{2s-\gamma_\circ}}{t^{-1}|y|^{n+2s}}\,dy \leq Ct R^{-\gamma_\circ},\]
therefore $S^1(x,t) \le CR^{-\gamma_\oo}$ inside $B_{R/4}.$

The lower bound follows from a similar argument, concluding the proof.
\end{proof}

We will also need the following result from \cite[Corollary 4.3]{ChaDav}.

\begin{prop}[half Harnack, \cite{ChaDav}] \label{halfharnack}
Let $s \in[\frac12,1)$ and $\LL$ as in \eqref{L0}-\eqref{ellipt}. 
Let $t_\circ>0$, and let $w$ be a viscosity supersolution of
\[
(\partial_t -\LL)w \ge -\eps \quad \mbox{in } B_1\times(-2t_\oo,0),\qquad \text{with}\quad w\ge 0 \quad \mbox{in  }\,\R^n.\]
Then
\[
\inf_{B_{1}\times(-t_\oo,0)}   w \ge -\varepsilon +c \int_{-2t_\oo}^{-3t_\oo/2} dt \int_{\R^n} dx\, \frac{w(x,t)} {1+|x|^{n+2s}},
\]
for some $c>0$ depending only on $n$, $s$, $\lambda$, $\Lambda$, and  $t_\oo$.
\end{prop}

We also need the following:

\begin{lem}\label{global}
Let $s \in[\frac12,1)$ and $\LL$ as in \eqref{L0}-\eqref{ellipt}. 
Given $t_\oo>0$, there exists  $C>0$ depending only on $n$,  $s$, $\lambda$, $\Lambda$, and $t_\oo$, such that the following holds.

Let $\rho\ge 1$, $A\subset \R^n \times \R$ be a closed set, and $w$ a viscosity solution of
\[
\begin{cases}
(\partial_t -\LL)w =0 \quad &\mbox{in } A^c\cap \big(B_\rho \times(-2t_\oo,0)\big)
\\
w \equiv 0  &\mbox{in } \big( (A\cap B_\rho) \cup (\R^n\setminus  B_\rho)\big)\times(-2t_\oo,0) .
\end{cases}
\]
Then
\[
\sup_{B_2 \times (- t_\oo,0)} w \le  C  \int_{\R^n} \frac{w^+(x,-t_1)} {(1+|x|)^{n+2s}}dx \qquad \text{for all $-t_1\in [-2t_\oo,-3t_\oo/2]$.}
\]
\end{lem}

\begin{proof}
Observe that $(\partial_t -\LL)w^+\le 0$ in $\R^n\times (-2t_\oo,0)$. As a consequence  we obtain
\[
w^+(x,-t_1+t) \le  \int_{\R^n} w^+(y,-t_1) H(x-y,t) dy  \quad \mbox{for } t\in [-t_1,0].
\]
Using the heat kernel bounds \eqref{HKbounds} we obtain
\[
\sup_{(x,t) \in B_2 \times (- t_\oo,0)} w(x,t)  \le \sup_{(x,t) \in B_2 \times (- t_\oo,0)}  \int_{\R^n} w^+(y,-t_1) H(x-y,t-t_1) dy \le C \int_{\R^n} \frac{w^+(y,-t_1)} {(1+|y|)^{n+\sigma}}dx,
\]
where we used that, for $t\in (- t_\oo,0) $ and $t_1\in [3t_\oo/2,2t_\oo]$, we have $t_\oo/2 \le t_1-t\le 2t_\oo$.
The lemma follows.
\end{proof}

We can now give the:

\begin{proof}[Proof of Proposition \ref{step}]
We divide the proof into three steps.

\vspace{2mm}

\noindent
- {\em Step 1.} Fix $\eps>0$ small to be chosen later and let $R=R_\eps := \eps^{-2/\gamma_\oo}$.  We claim that if $\rho\ge R_\eps$ and $i \in \{1,2\}$, then
\begin{equation}\label{a1}
\inf_{Q^*}   v_i \ge -C\varepsilon+c \int_{-2t_\oo}^{-3t_\oo/2} dt \int_{\R^n} dx\, \frac{v_i^+(x,t)} {(1+|x|)^{n+2s}}
\end{equation}
for some constants $c>0$ small and  $C>0$ large (recall $Q^* :=  B_{\delta}(e_n)\times (-5t_\oo/4, 0) \subset A^c$).

Indeed, it suffices to apply Proposition \ref{halfharnack} (rescalled) to the function
\[
w(x,t) :=  v_i(x,t) +  \eps (1+|x|^2)^{\frac 12 (2s-\gamma_\oo/2)},
\]
which is nonnegative\footnote{Notice $\eps (1+|x|^2)^{\frac 12 (2s-\gamma_\oo/2)} \ge  \eps R_\circ^{\gamma_\oo/2}  (1+|x|^2)^{\frac 12 (2s-\gamma_\oo)} = (1+|x|^2)^{\frac 12 (2s-\gamma_\oo)}$ for $|x|\ge R_\oo$.} in all of $\R^n \times (-2t_\oo, 0)$, to get
$$
\inf_{Q^*}   w\ge -\varepsilon +c \int_{-2t_\oo}^{-3t_\oo/2} dt \int_{\R^n} dx\, \frac{w(x,t)} {(1+|x|)^{n+2s}}.
$$
This implies that
$$
\inf_{Q^*}   v_i\ge -C\varepsilon +c \int_{-2t_\oo}^{-3t_\oo/2} dt \int_{\R^n} dx\, \frac{v_i(x,t)} {(1+|x|)^{n+2s}}.
$$
Also, noticing that 
\[ |v_i-v_i^+|\leq\varepsilon  (1+|x|^2)^{\frac{2s-\gamma_\oo}{2}} \quad \mbox{in } B_{\rho}\times(-2t_\oo, 0),\qquad 
|v_i-v_i^+| \le  (1+|x|^2)^{\frac{2s-\gamma_\oo}{2}}  \quad \mbox{in }(\R^n\setminus B_{\rho})\times(-2t_\oo,0),\]
we easily get that
$$
\bigg| \int_{-2t_\oo}^{-3t_\oo/2} dt \int_{\R^n} dx\, \frac{v_i(x,t)} {(1+|x|)^{n+2s}}-
\int_{-2t_\oo}^{-3t_\oo/2} dt \int_{\R^n} dx\, \frac{v_i^+(x,t)} {(1+|x|)^{n+2s}}\biggr|\leq C\eps.
$$
so \eqref{a1} follows.

\vspace{2mm}

\noindent
- {\em Step 2.}
We prove that
\begin{equation}\label{goal1}
1\le \sup_{B_2\times (-3t_\circ/2,0)}  v_i \le  C_{1} \big(\inf_{Q^*}  v_i + \eps\big) \le 2C_1.
\end{equation}
To this aim, choose $t_1\in (3t_\oo/2, 2t_\oo)$ such that
\begin{equation}\label{choose t1}
\int_{-2t_\oo}^{-3t_\oo/2} dt \int_{\R^n} dx\, \frac{v_i^+(x,t)} {(1+|x|)^{n+2s}}  \ge \frac{t_\oo}{2}\int_{\R^n} dx\, \frac{v_i^+(x,-t_1)} {(1+|x|)^{n+2s}}
\end{equation}
and decompose
\[
v_i = v_i^{\rm main} +v_i^{\rm error},
\]
where
$v_i^{\rm main}$ is the solution of
\[
\begin{cases}
(\partial_t -\LL) v_i^{\rm main} =0  \quad & \mbox{in }(B_\rho \setminus A)\times (-t_1,0)
\\
v_i^{\rm main} =0  				 & \mbox{in } \big(A\cup (\R^n\setminus B_\rho)\big)\times (-t_1,0)
\\
v_i^{\rm main} = v_i       & \mbox{in } B_\rho\times \{-t_1\}
\end{cases}
\]
and $v_i^{\rm error}$ satisfies
\[
\begin{cases}
\big|(\partial_t -\LL) v_i^{\rm error} \big|\le\eps  \quad & \mbox{in }(B_\rho \setminus A)\times (-t_1,0)
\\
|v_i^{\rm error}| \le C_\oo (1+|x|)^{2s-\gamma_\oo} 	& \mbox{in } \big(A\cup (\R^n\setminus B_\rho)\big)\times (-t_1,0)
\\
v_i^{\rm error} = 0      & \mbox{in } B_\rho\times \{-t_1\}.
\end{cases}
\]
Note that, since $v_i =0$  	inside $\big(A\cup (\R^n\setminus B_\rho)\big)\times (-t_1,0)$, then also $v_i^{\rm error}$ vanished inside this set.
Hence, choosing as barrier a rescaling of the function $S$ provided by 
Lemma \ref{supersol}, if $\rho\ge R_\eps= \eps^{-1/\gamma_\oo}$ we get
\[
\sup_{B_2\times (-3t_\circ/2,0)}\big| v_i^{\rm error} \big| \le C\eps.
\]
On the other hand, by Lemma \ref{global} we have
\[
\sup_{B_2\times (-3t_\circ/2,0)} v_i^{\rm main}   \le  C \int_{\R^n} dx\, \frac{v_i^+(x,-t_1)} {(1+|x|)^{n+2s}}.
\]
Combining this with  \eqref{choose t1} and \eqref{a1}, we conclude that 
\[
\sup_{B_2\times (-3t_\circ/2,0)}  v_i \le \int_{-2t_\oo}^{-3t_\oo/2} dt \int_{\R^n} dx\, \frac{v_i^+(x,t)} {(1+|x|)^{n+2s}}  \le  C_1\big(\inf_{Q^*}  v_i + \eps\big).
\]
Recalling that $v_i(e_n,0)=1$ and $\eps \in (0,1)$, we obtain \eqref{goal1}.

\vspace{2mm}

\noindent
- {\em Step 3.}
Finally, we want to prove that 
\[v_1 \le  C v_2 \qquad \mbox{in } Q_{1} = B_1\times(-t_\oo,0)\]
Let $\eta \in C^\infty_c( B_{3/2}\times (-\tfrac 54t_\oo,0] )$ be nonnegative cutoff function with $\eta=1$ in $\overline{B_{1}} \times [-t_\oo,0]$, and define
\[
w(x,t) :=  v_1(x,t) \chi_{B_2}(x) + (2C_1+1)(\eta(x,t)-1),
\]
where $C_1$ is the constant in \eqref{goal1}.
Since $v_1(x,t) \le 2C_1$ in $B_2\times(-3t_\oo/2,0)$, we have
\[
w(x,t) \le -1   \qquad \mbox{in }\big(B_{3/2}^c\times(-\tfrac 54,0)\big) \cup \R^n\times\{-\tfrac 54t_\oo\}.
\]
In addition,
\[
(\partial_t -\LL)w \le (\partial_t -\LL)v_1 +C \le \varepsilon +C\le C \qquad \mbox{in }A^c \cap \big( B_{3/2}\times(-2t_\oo,0)\big).
\]
Let $\xi(x,t)=\xi(x) := \chi_{B_{\delta}(e_n)} (x)$.
Since $(\partial_t -\LL) \xi(x,t) \le -c <0$   for $(x,t)\in (B_1\setminus B_{\delta}(e_n))\times\R$,
for $C_2$ large enough we have
\[
(\partial_t -\LL) \big(w+C_2\xi\big) \le -1 \qquad \mbox{in  } (B_1\setminus B_{\delta}(e_n)) \times\R.
\]
Furthermore, by \eqref{goal1} we see that $\inf_{Q^*}  v_i \geq \frac{1}{2C_1}$ provided $\eps$ is sufficiently small.
In particular, we can choose $C_3$ large enough so that
\[w+C_2\xi \le C_3v_2 \qquad \mbox{in  }  Q^* = B_{\delta}(e_n) \times(-\tfrac 54t_\oo,0)\]
Combining all these estimates together, this proves that
\[
V(x,t):=C_3v_2(t,x)\chi_{B_2}(x) -w(x,t)-C_2\xi(x,t) \ge 0 \qquad \mbox{in } \big((B_{3/2}^c \cup B_{\delta}(e_n)) \times(-\tfrac 5 4 t_\oo,0) \big)\cup (\R^n \times\{-\tfrac 54 t_\oo\}),
\]
provided $\varepsilon$ is sufficiently small.
Hence, since
\[
(\partial_t -\LL) \big(C_3v_2-w-C_2\xi\big) \ge 1-C_3 \varepsilon \qquad \mbox{in } \bigl(B_{3/2} \setminus  B_{\delta}(e_n)\bigr) \times(-\tfrac 54 t_\oo,0),
\]
it follows that
\[
(\partial_t -\LL) V \ge 1-C_3 \varepsilon- C_3\big|(\partial_t -\LL) \big(v_2\chi_{B_2^c})\big|\geq 1-C_4\eps \qquad \mbox{in } \bigl(B_{3/2} \setminus  B_{\delta}(e_n)\bigr) \times(-\tfrac 54 t_\oo,0).
\]
Taking $\varepsilon$ small so that $1-C_4 \varepsilon> 0$, it  follows from the maximum principle that
\[
C_3v_2-w-C\xi \ge 0\qquad \mbox{in }  B_{3/2} \times(-\tfrac 54t_\oo,0).
\]
In particular,
\[
v_1 = w \le w +C\xi \le C_3v_2 \qquad \mbox{in }Q_{1},
\]
as desired.

Finally, notice that the exact same argument with $w(x,t) =  \eta(x,t)-1$ (i.e., replacing both $v_1$ by $0$ and $2C_1$ by~$1$ in the previous argument) shows that $v_2\ge0$ in $Q_{1}$, and then  
$v_2>0$ in $Q_{1}\setminus A$ by the strong maximum principle (since, by assumption, $v_2$ is nonzero).
\end{proof}

We now construct a subsolution to prove a nondegeneracy property in moving Lipschitz domains.

\begin{lem}[Subsolution supported in a traveling cone] \label{subsolcone}
Let $s \in[\frac12,1)$ and $\LL$ as in \eqref{L0}-\eqref{ellipt}. 
Given $\omega_\oo\ge0$, $e_\oo\in \mathbb S^{n-1}$, and $\theta_\oo\in(0,\pi)$, there are positive constants $\gamma$ and $c$, depending only on $n$, $s$, $\lambda$, $\Lambda$, $\omega_\oo$, and $\theta_\oo$, such that the following statement holds.

Consider the traveling cone
\[\Sigma_t := \big\{ x\in \R^n \,:\,\angle(e_\oo,\textstyle \frac{x}{|x|}) \le \theta_\oo\big\} - \omega_\oo t e_\oo\]
and fix a smooth $1$-homogeneous function $\psi: \Sigma_0 \to (0, \infty)$  such that:\\
- $\psi(x) = {\rm dist}(x, \R^n \setminus \Sigma_0)$ for all $x\in \{ \tfrac{9}{10} \theta_\oo\le  \angle(e_\oo,\textstyle \frac{x}{|x|})\le \theta_\oo\}$;\\
- $\nabla \psi \cdot e_\circ >0$ in $\Sigma_0$.  \\
Then the ``traveling wave'' $\varphi=\varphi_\gamma$ given by $\varphi(x,t) := \big(\psi(x+ \omega_\oo t e_\oo)\big)_+^{2s-\gamma}$ satisfies
\begin{equation}\label{subsol}
(\partial_t -\LL)\varphi \le -c <0 \
\end{equation}
in $B_1\times (-1,0)$.
\end{lem}
\begin{proof}
By translation invariance in $t$ we just need to show \eqref{subsol} in $B_{1+\omega_\oo}\times\{0\}$. Then, by scaling, it is enough to prove \eqref{subsol} just in $B_{1}\times\{0\}$ (up to changing $c$ and the ellipticity constants).
Since $(\partial_t -\LL)\varphi \le -c <0$ in $(B_1\setminus \Sigma_0 )  \times \{0\}$ (note that $\varphi\geq 0$ vanishes at those points and $\LL$ is nonlocal), it suffices to prove \eqref{subsol} for $(x,0)\in (\Sigma_0\cap B_1)\times \{0\}$.

We claim that it suffices to show that
\[
(\partial_t -\LL)\varphi(x_\oo,0)<-1 \quad \mbox{for all } x_\oo\in  \Sigma_0  \mbox{ with }\psi(x_\circ)   = M:=\max_{\overline B_1} \psi.
\]
Indeed, given
$(x,0) \in (\Sigma_0\cap B_1)\times \{0\}$ we have $\psi(x) \in (0, M]$, hence ---by homogeneity--- there is $x_\oo \in \Sigma_0  \cap \{\psi = M\}$ and  $r\in(0,1)$ such that $x= rx_\oo$.
Therefore, defining $\tilde \varphi(x,t) = \varphi(rx, r t)$ and noticing that  (again by homogeneity of $\psi$ and using the definition of $\varphi$) $\tilde \varphi(x,t) = r^{2s-\gamma}\varphi(x,t)$ we obtain
\[
\begin{split}
(\partial_t -\LL)\varphi(x_\oo,0) =  r^{\gamma-2s}&(\partial_t -\LL)\tilde \varphi(x_\oo,0) =
= r^{\gamma-2s}(r\partial_t -r^{2s} \LL) \varphi(x,0) \ge  r^{\gamma}  (\partial_t -\LL)\varphi(x,0),
\end{split}
\]
where we used  $\partial_t \varphi\ge 0$ (since $e_\circ \cdot \nabla \psi >0$) and $2s\ge 1$.
Thus
\[
(\partial_t -\LL)\varphi(x,0) \le r^{-\gamma}(\partial_t -\LL)\varphi(x_\oo,0),
\]
and therefore  it suffices to  show $(\partial_t -\LL)\varphi(x_\oo,0)\leq -1$, as claimed.

To show that $(\partial_t -\LL)\varphi(x_\oo,0)\le -1$ for $\gamma>0$ small enough, it is useful to think of the following dichotomy (although they are treated almost identically):\\
- either $x_\circ$ belongs to a compact subset of $\Sigma_\circ$;\\
- or $|x_\circ|$ is very large and therefore, since $\psi(x_\circ)=M$, $x_\circ$belongs to the cone  $\{ \tfrac{9}{10} \theta_\oo\le  \angle(e_\oo,\textstyle \frac{x}{|x|})\le \theta_\oo\}$ where $\psi= {\rm dist}(\,\cdot\, , \R^n \setminus \Sigma_0)$. In particular, ${\rm dist}(x_\circ, \R^n \setminus \Sigma_0)=M.$\\
In both cases it is simple to show that there exists $\rho_\circ=\rho_\circ(\theta_\circ,M)>0$ such that
 \[
 |\partial_t\varphi(x_\oo, 0)| +  \|\varphi(x_\oo+\,\cdot\,, 0) \|_{C^2(B') } \le C
 \qquad \text{in $B':= \big\{ |x| \le  \rho_\circ\big\}$},
\]
with $C$ independent of  $x_\oo$ and $\gamma$. In addition, keeping again in mind the previous dichotomy,  in both cases we have 
\[
\min_{x_\circ\in \{\psi=M\} } \int_{\R^n\setminus B'} \frac{\varphi(x_\oo+y,0) } { |y|^{n+2s}}\, dy \rightarrow +\infty \qquad \text{as $\gamma\downarrow 0$.}
\]
Then the lemma follows from the following simple bound, choosing $\gamma$ sufficiently small:
\begin{multline*}
(\partial_t -\LL)\varphi(x_\oo,0) \le  |\partial_t \varphi(x_\oo, 0)| - \int_{\R^n} (\varphi(x_\circ +y,0) +\varphi(x_\circ -y,0) -2\varphi (x_\circ,0))K(y) 
\\
\le C +  C\int_{B'} |y|^2 \frac{\Lambda}{|y|^{n+2s}} dy +  \int_{\R^n \setminus B' } 2\varphi (x_\circ,0 )\frac{\Lambda}{|y|^{n+2s}}dy - \int_{\R^n \setminus B'} \varphi(x_\circ +y,0) \frac{  \lambda} {|y|^{n+2s}} dy.
\end{multline*}
\end{proof}

We can now prove our parabolic boundary Harnack.

\begin{proof}[Proof of Theorem \ref{bdryharnack}]
First we note that $v_1$ and $v_2$ play symmetric roles in the theorem. Also, as a consequence of Proposition \ref{step}, $v_i>0$ in $A^c\cap Q_{1}$ provided $\varepsilon>0$ is small enough.
Our goal will be to prove that, in parabolic cylinders centered at $(0,0)$, the quotient $v_1/v_2$ decays geometrically. More precisely, setting
\[
 Q_r :=  B_r \cap (- r^{2s} t_\oo,0)
 \]
we shall prove that
\begin{equation}\label{osc}
{\rm osc}_{A^c\cap Q_r } \left( \frac{v_1}{v_2} \right) \le  r^{\alpha'} \qquad \mbox{for }r\in(0,\bar r),
\end{equation}
provided that $R$  is chosen large enough, and $\varepsilon$, $\alpha'$, and $\bar r$  are small positive constants.
Since $(0,0)$ can be replaced by any other point in $A^c\cap Q_1$,  \eqref{osc} will hold at every point in $A^c\cap Q_1$  with uniform constants, implying the theorem. 

 To prove \eqref{osc} we use the subsolution $\varphi$ from Lemma \ref{subsolcone}, and we then rescale and iterate Proposition \ref{step} along a sequence of geometric scales, as explained next. We split the argument into three steps.

\vspace{2mm}

\noindent
- {\em Step 1.} We first show that
\begin{equation}\label{controlsubsol}
v_i (re_n,0)\ge c_1r^{2s-\gamma} \qquad \mbox{for } r\in(0,1/2),
\end{equation}
where $c_1$ and $\gamma$ are positive constants.

Recall that, by assumption, $A^c$ satisfies the interior cone property at $(0,0)$, where $e_\circ$, $\theta$, and  $\omega$ refer to the traveling cone's direction, opening, and speed, respectively. 
 Assume without loss of generality  that  $\omega \ge \frac{2}{t_\oo}$ and $e_\oo = e_n$.  
Setting $\mathcal C_\delta:=  \cup_{r>0}\chi_{B_{r\delta}(re_n)}$, for $\delta>0$ small (comparable with $\theta$) we have that $(B_2\cap (\mathcal C_\delta - e_\circ \omega t))\times\{t\} \subset A^c$ for all $t\in [-2t_\circ, t_\circ]$.

Let $\varphi$ be a subsolution as in Lemma \ref{subsolcone}, with $\varphi(0,t)$ supported in the (spatial) cone $\mathcal C_{\delta}$ and traveling  in the $-e_\oo$ direction at speed $\omega_\circ =\omega$. 
Recall that the subsolutions $\varphi(\cdot, 0)$ is a  spatially $(2s-\gamma)$-homogeneous function, and that $\varphi(\,\cdot\,,t)\equiv 0$ in $B_{2}$ for $t\le -t_\oo$.  

Now let $\eta\in C^\infty_c (B_{3/2})$ be a nonnegative spatial cut-off function such that $\eta = 1$  in $B_1$, and define
\[
w(x,t) :=  \big( \varphi(x,t)\chi_{B_2}(x)  + (\eta(x)-1)\big( \max_{B_2} \varphi(\,\cdot\, , 0)\big) \big)_+ + C_2 \chi_{B_{\delta/4}(e_n)} (x).
\]
For $C_2$ large enough, we have
\[
(\partial_t - \mathcal L) w\ge 1 \qquad {\rm in }\ (B_{3/2} \setminus B_{\delta/2}(e_n))\times (-t_\oo, 0).
\]
On the other hand, by construction, $w=0$ on $B_{3/2}\times \{t_\oo\}$ and inside $(\R^n \setminus B_{3/2})\times [-t_\circ, 0] $ .
Also, by Proposition \ref{step}, we have $\frac 1{C'} \le   v_i$ in $Q^* = B_{\delta}(e_n) \times(-5t_\oo/4,0)$.
Using that $(\partial_t - \mathcal L) v_i\le \eps$ in $B_1\times(-t_\oo,0) \setminus A$  and that the support of $w$ is contained in the complement of $A$, applying the maximum principle to $w$ and $C'C_2v_i$, if $\eps< \frac{1}{C'C_2}$ we obtain
\[
w(x,t) \le C v_i(x) \qquad \mbox{for  }(x,t)\in B_1  \times (-t_\oo, 0).
\]
Evaluating at $(re_n,0)$, this proves \eqref{controlsubsol}.

From now on, fixed $\gamma$ as in \eqref{controlsubsol}, we assume without loss of generality that $\gamma_\circ<\gamma$.\footnote{Note that if the assumptions of the theorem are satisfied for some $\gamma_\circ$, then they are also satisfied with $\gamma_\circ$ smaller.
}

\vspace{2mm}

\noindent
- {\em Step 2}. We now show that there exists $C>0$ such that, for all $r>0$ small and $R\le 1/r $, we have
\begin{equation}\label{whohfiowhoh1}
\sup_{B_{Rr}} v_i \le  Cv_i(re_n,0) R^{2s-\gamma}.
\end{equation}

Indeed, consider the functions
\[
\bar v_i(x,t) : = \frac{v_i(\bar r x, \bar r ^{2s}t)}{v_i(\bar r e_n,0)}.
\]
Since $\gamma_\oo<\gamma$, combining the assumption $|v_i (x,t)| \le C_\oo (1+|x|)^{2s-\gamma_\oo}$ with \eqref{controlsubsol}
it follows that, if $\bar r  \in (0,1)$, then the functions $\bar v_i(x,t)$ 
satisfy the assumptions of Proposition  \ref{step} with $\rho= \bar r^{-1}R$ and with uniform constants (i.e., not degenerating as $\bar r$ goes to zero).\footnote{Notice that here we are using $2s\ge 1$: indeed,   the set where $\bar v_i$ vanishes is the rescaling of the set $A$, namely $\{( x/\bar r ,t/ \bar r^{2s}) \, : \, (x,t) \in A\}$. Such a set satisfies the interior cone with opening angle and speed independent  $\bar r \in (0,1)$ if and only if $2s\ge 1$.}
Hence, applying Proposition \ref{step} we deduce that 
$\frac 1 {C'} \le \bar v_i $ in $Q^*$ and  $0\le \bar v_i\le C'$ in $Q_1$.  This allows us to repeat the subsolution argument of Step 1 with $v_i$ replaced by $\bar v_i$, so to obtain 
\begin{equation}\label{howihowhoh}
\bar v_i(e_n/R, 0)  \ge c_1\bar v_i(e_n, 0) (1/R)^{2s-\gamma} \qquad \text{for all $R\ge 1$.}
\end{equation}
Choosing $R$ such that $Rr= \bar r$, this yields
\[
\begin{split}
\frac{\sup_{Q_{Rr}} v_i } {v_i(\bar r e_n,0) }&=   \sup_{Q_1} \bar v_i   \le C' =  C'\bar v_i(e_n,0)  \le \frac{C'}{c_1}R^{2s-\gamma} \bar v_i(e_n/R,0)  = \frac{C'}{c_1}R^{2s-\gamma} \frac{ v_i(re_n,0)  } {v_i(\bar r e_n,0) },
\end{split}
\]
proving \eqref{whohfiowhoh1}.

\vspace{2mm}

\noindent
- {\em Step 3}.
We obtain the geometrically improving ``sandwich-type'' estimates
\begin{equation}\label{indu1}
m_j v_1 \le v_2  \le  M_j v_1 \quad \mbox{in } Q_{\rho^{-j}}   , \quad j\ge 1
\end{equation}
where
\begin{equation}\label{indu2}
0\le M_j -m_j = C_3(1-\eta)^j
\end{equation}
for some positive constants $\varrho$,  $C_3$ (both large)  and $\eta$ (small).

Indeed, thanks to Proposition \ref{step}, both \eqref{indu1} and \eqref{indu2} are satisfied for $j=1$ (provided $\rho\ge 2$), for some positive constants $M_1$, $m_1$, and $C_3$. 
We now proceed by induction: assume that \eqref{indu1} and \eqref{indu2} hold for $1\le j\le k$, and let us prove that they also hold for $j=k+1$.

We first show the validity of \eqref{indu1} in $Q_{\rho^{-k-1}}$. 
We consider two cases:\\
- if
\begin{equation}\label{case1}
(v_2 -m_k v_1)(\rho^{-k-1}e_n,0) \ge (M_k v_1- v_2) (\rho^{-k-1}e_n,0)
\end{equation}
then we prove that \eqref{indu1} holds $j= k+1$ for $m_{k+1}=m_k + \eta(1-\eta)^k$ and $M_{k+1}=M_k$;\\
-  if
\begin{equation}\label{case2}
(v_2 -m_k v_1)(\rho^{-k-1}e_n,0) < (M_k v_1- v_2) (\rho^{-k-1}e_n,0)
\end{equation}
then we prove that \eqref{indu1} holds $j= k+1$ for $m_{k+1}=m_k$ and $M_{k+1}=M_k- \eta(1-\eta)^k$.

Assume that we are in the first case.
We begin by noticing that, as a consequence of \eqref{case1}, we have 
\begin{multline*}
(v_2 -m_k v_1)(\rho^{-k-1}e_n,0) + (v_2 -m_k v_1)(\rho^{-k-1}e_n,0)\\
\ge 
(v_2 -m_k v_1)(\rho^{-k-1}e_n,0) +(M_k v_1- v_2) (\rho^{-k-1}e_n,0)=(M_k -m_k)   v_1 (\rho^{-k-1}e_n,0) ,
\end{multline*}
that is,
$$
(v_2 -m_k v_1)(\rho^{-k-1}e_n,0) \ge  \frac 12  (M_k -m_k)   v_1 (\rho^{-k-1}e_n,0).
$$
Hence, since $\gamma_\oo \in (0,\gamma)$, then
 \eqref{indu2}  and \eqref{controlsubsol} yield
\begin{equation}\begin{split}
(v_2 -m_k v_1)(\rho^{-k-1}e_n,0) 
\ge \frac 1 2 C_3 (1-\eta)^k c_1 \rho^{-(k+1)(2s-\gamma)}
\ge \rho^{-(k+1)(2s-\gamma_\oo)},
\end{split}
\end{equation}
provided that $\rho$ is chosen large enough.
Also, using again \eqref{controlsubsol},
\begin{equation}
v_1(\rho^{-k-1}e_n,0) \ge c_1 \rho^{-(k+1)(2s-\gamma)}
\ge \rho^{-(k+1)(2s-\gamma_\oo)}.
\end{equation}
Let us consider the functions
\[
\tilde v_1(x,t) : =  \frac{v_1(\rho^{-(k+1)} x, \rho^{-2s(k+1)}t )}{v_1(\rho^{-k-1}e_n,0)},\qquad  \tilde v_2(x,t) :=  \frac{(v_2 -m_k v_1) (\rho^{-(k+1)} x, \rho^{-2s(k+1)}t )}{(v_2 -m_k v_1)(\rho^{-k-1}e_n,0)},
\]
and show that they satisfy the assumptions of  Proposition \ref{step}  (with $c_\oo =1$) if $\eta>0$ is small enough.

Indeed we already argued in Step 2 that, since $2s\ge 1$, parabolic rescaling preserves the interior cone condition. Also, by construction and by \eqref{whohfiowhoh1} we have $\tilde v_i(e_n,0) =1$ and
\[
 \tilde v_1(x) \le C(1+|x|)^{2s-\gamma}  \le (1+|x|)^{2s-\gamma_\circ}  \qquad \mbox{in  }(\R^n \setminus B_\rho)\times (-2t_\oo, 0),
\]
provided $\varrho$ is chosen large enough (here we use again $\gamma_\circ<\gamma$).

We want now to obtain a similar bound for $v_2(x,t)$, and this is slightly more subtle. We note that, thanks to \eqref{case1}, we have
\[
(v_2 -m_k v_1)(\rho^{-k-1}e_n,0) \ge (M_k -m_k)v_1(\rho^{-k-1}e_n,0) 
\]
Also, by induction hypothesis,  $m_j v_1 \le v_2  \le  M_j v_1$ in  $Q_{\rho^{-j}}$ for all $j\le k$. Thus  
\[
|v_2-m_k v_1| \le (M_j-m_k )v_1 \le (M_j-m_j) v_1 = (1-\eta)^{j-k}  (M_k-m_k) v_1\qquad \mbox{in }  Q_{\rho^{-j}} .
\]
Now, given $(x,t)$,  select  the maximal $j$ such that $(x,t) \in  Q_{\rho^{-j}}$. Hence, we obtain 
 \[
(v_2-m_k v_1) (x,t) \le C v_1(x,t) (M_k-m_k)  \bigg(1 +\frac{|x| + |t|^{1/2s}}{\rho^{-k}}\bigg)^{\delta},
\]
where $\delta= \delta(\eta)\downarrow 0$ as $\eta \downarrow 0$. Hence, using again \eqref{whohfiowhoh1}, we obtain 
\[
\begin{split}
 \tilde v_2(x,t) &=  \frac{(v_2 -m_k v_1) (\rho^{-(k+1)} x, \rho^{-2s(k+1)}t )}{(v_2 -m_k v_1)(\rho^{-k-1}e_n,0)} \le  \frac{C v_1 (\rho^{-(k+1)} x, \rho^{-2s(k+1)}t ) \big(1 +|x| + |t|^{1/2s}\big)^{\delta}}{v_1(\rho^{-k-1}e_n,0) }
\\
 &\le C(1+|x|)^{2s-\gamma+\delta} \le (1+|x|)^{2s-\gamma_\oo}    \qquad \mbox{in  }(\R^n \setminus B_\rho)\times (-2t_\oo, 0),
 \end{split}
\]
where we choose $\delta< \gamma-\gamma_\oo$ and $\rho$ large enough to absorb the constant $C$ in the last inequality.

Finally, as a consequence of the inductive hypothesis ---namely that \eqref{indu1} holds for $j=k$--- the  functions $\tilde v_1$ and $\tilde v_2$ are both nonnegative in $B_\rho\times (-2t_\oo, 0)$. 

Having verified that  $\tilde v_1$ and $\tilde v_2$  satisfy  the assumptions of Proposition \ref{step}  we conclude that 
$\frac 1{C} \tilde v_1\le \tilde v_2$  in $Q_{1}$, that is
\[\begin{split}
\frac{1}{C}  \frac{v_1(\rho^{-(k+1)} x, \rho^{-2s(k+1)}t )}{v_1(\rho^{-k-1}e_n,0)} & \le  \frac{(v_2 -m_k v_1) (\rho^{-(k+1)} x, \rho^{-2s(k+1)}t )}{(v_2 -m_k v_1)(\rho^{-k-1}e_n,0)} \\
&\le
\frac{(v_2 -m_k v_1) (\rho^{-(k+1)} x, \rho^{-2s(k+1)}t )}{ \frac 1 2(M_k-m_k)v_1(\rho^{-k-1}e_n,0) } \qquad \mbox{in }Q_1,
\end{split}\]
or equivalently
\[
\frac{1}{2C}(M_k-m_k) v_1 \le
v_2 -m_k v_1\qquad \mbox{in }Q_{\rho^{-k-1}},
\] 
as desired.

This proves the validity of the inductive step in the case \eqref{case1}. The case \eqref{case2} can be proved similarly and is left to the interested reader.
\end{proof}

\section{Main parabolic result}
\label{sec4}

The goal of this Section is to prove Theorem \ref{thm-main}.
The proof will require several steps.

\subsection{Classification of blow-ups for $s=\frac12$}

Our first main goal will be to classify blow-ups in the critical case $s=\frac12$.
For this, the new parabolic boundary Harnack  from Theorem \ref{bdryharnack} will be crucial to establish the following:

\begin{prop}\label{prop-uniqueness-positive-sols-critical}
Let $s=\frac12$ and $\LL$ as in \eqref{L0}-\eqref{ellipt}, with $K$ homogeneous.

Let $\Sigma\subset \R^n\times\R_-$ be any closed convex cone with nonempty interior, and with vertex at $(0,0)$.
Let $w_1,w_2\in C(\R^n\times \R_-)$ be positive solutions of 
\[\partial_t w_i - \LL w_i=0\quad\textrm{in}\quad \Sigma^c,\qquad \textrm{with}\quad 
w_i\equiv0 \quad\textrm{in}\quad \Sigma.\]
Then $w_1\equiv \kappa w_2$ in $\R^n\times \R_-$ for some constant $\kappa$.
\end{prop}

\begin{proof}
The result follows from the parabolic boundary Harnack we that we proved in Theorem~\ref{bdryharnack}.
Indeed, by convexity the set $\Sigma^c\subset \R^n\times\R_-$ satisfies the interior cone condition\footnote{Here, it is very important that our boundary Harnack is not only for Lipschitz domains, but for general domains satisfying the interior cone condition. For example, it might happen that $\Sigma$ is a very degenerate cone for $t=0$, but the boundary Harnack still holds.}.
Thus, for every $R\geq1$ we can apply Theorem \ref{bdryharnack} to the functions $w_i(2Rx,2Rt)/C_R^{(i)}$, with $C_R^{(i)}:=w_i(2Re_n,0)$, to deduce that
\begin{equation}\label{eoiirfn}
\left[\frac{w_1}{w_2}\right]_{C^\tau(Q_R\cap \Sigma^c)}\leq CR^{-\tau} \frac{C_R^{(1)}}{C_R^{(2)}},\
\end{equation}
with $C$ independent of $R\geq1$.
Moreover, by Proposition \ref{step} we also know that 
\[ \frac{w_2}{C_R^{(2)}} \leq C\frac{w_1}{C_R^{(1)}}\qquad {\rm in}\quad Q_R,\]
and therefore, evaluating this inequality at some arbitrary point in $Q_1\cap \Sigma^c$, we deduce that $C_R^{(1)}/C_R^{(2)}$ is uniformly bounded with respect to $R\geq1$.
Thus, letting $R\to\infty$ in \eqref{eoiirfn}, we deduce that  $\left[\frac{w_1}{w_2}\right]_{C^\tau(\Sigma^c)}=0$. Since both functions vanish outside $\Sigma^c$, we conclude that
\[w_1\equiv \kappa w_2\qquad \textrm{in}\quad \R^n\times \R_-\]
for some $\kappa \in \R.$
\end{proof}

We now prove a collection of technical lemmas (still for the case $s=\frac12$) that will be needed later.

\begin{lem}\label{lem-supersol1D-critical}
Let $s=\frac12$ and $\LL$ as in \eqref{L0}-\eqref{ellipt}, with $K$ homogeneous.
Let $e\in \mathbb S^{n-1}$ and $\varv\in[0,\varv_\circ]$ for some constant $\varv_\circ>0$.
Then, there exists $\theta>0$ such that 
\[\psi(x,t):= \exp\big(-|x\cdot e+\varv t|^{1-\theta}\big)\]
satisfies 
 \[\partial_t \psi - \LL\psi\geq -C\qquad\textrm{in}\quad \R^n\times \R.\]
The constants $C$ and $\theta$ depend only on $n$, $s$, $\varv_\circ$, and the ellipticity constants.
\end{lem}

\begin{proof}
We prove it for $e=e_n$.
Let $\mathcal M_{\lambda,\Lambda}^-$ be the extremal operator associated to our class of operators, i.e., $\mathcal M_{\lambda,\Lambda}^- w:=\inf_\LL \LL w$, where the infimum is taken among all operators $\LL$ as in Definition \ref{L} (with fixed $s=\frac12$, $\lambda$, $\Lambda$).
Then, the operator $\partial_t-\mathcal M_{\lambda,\Lambda}^-$ is scale invariant of order $1$, and
\[(\partial_t-\mathcal M_{\lambda,\Lambda}^-)|x_n+\varv t|^\beta = -(c_\beta-v\,{\rm sign}(x_n+\varv t)) |x_n+\varv t|^{\beta-1}\]
for $\beta\in(0,1)$.
Moreover, it is easy to see that $c_\beta\to+\infty$ as $\beta\to 1$, uniformly in $\varv\in[0,\varv_\circ]$.
Hence, there exists $\theta>0$ small such that $c_{1-\theta}>1+\varv_\circ$.
This implies that, for any operator $\LL$ as in Definition \ref{L}, we have 
\[(\partial_t -\LL)|x_n+\varv t|^{1-\theta}\leq -|x_n+\varv t|^{-\theta}\leq0 \qquad \textrm{in}\quad \R^n.\]
In particular, since $\psi$ is bounded and the difference between $\psi(x,t)$ and $-|x_n+\varv t|^{1-\theta}$ is of class $C^{1,1/2}$ for $0<\theta\leq 1/4$, then the function $\psi$ satisfies $(\partial_t -\LL)\psi \geq -C$ in $\R^n$, as wanted.
\end{proof}

We next show the following.

\begin{lem}\label{lem-solution-across-critical}
Let $s=\frac12$ and $\LL$ as in \eqref{L0}-\eqref{ellipt}, with $K$ homogeneous.
Let $e\in \mathbb S^{n-1}$, $\varv \geq0$, and $\Gamma\subset \{x\cdot e+\varv t=0\}\subset \R^n\times(-\infty,0)$.
Assume $w\in {\rm Lip}_{\rm loc}(\R^n\times (-T,0))$ is a viscosity solution of
\[\partial_t w - \LL w \leq 0 \qquad \textrm{in}\quad \big(\R^n\times (-T,0)\big)\setminus \Gamma.\]
Then $\partial_t w - \LL w\leq0$ in $\R^n\times (-T,0)$.
\end{lem}

\begin{proof}
Let $\psi(x,t)$ be given by Lemma~\ref{lem-supersol1D-critical}, and for 
 any $\varepsilon>0$ consider the function $w_\varepsilon := w-\varepsilon\psi$.

Assume now that a test function $\eta\in C^2$ touches $w_\varepsilon$ from above at $(x_\circ,t_\circ)\in \R^n\times (-T,0)$.
Since $w$ is Lipschitz, it follows from the definition of $\psi$ (which has a H\"older cusp along $\{x\cdot e+\varv t=0\}$) that the point $(x_\circ,t_\circ)$ cannot belong to the set $\{x\cdot e+\varv t=0\}$.
Hence, thanks to our assumption and Lemma~\ref{lem-supersol1D-critical}, $$(\partial_t -\LL)\eta(x_\circ,t_\circ) = (\partial_t -\LL)w(x_\circ,t_\circ)-\varepsilon (\partial_t -\LL)\psi(x_\circ,t_\circ)\leq C\varepsilon.$$
This implies that $(\partial_t -\LL)w_\varepsilon \leq C\varepsilon$ in $\R^n\times (-T,0)$ in the viscosity sense.
Since $w=\sup_{\varepsilon>0} w_\varepsilon$, we conclude that $(\partial_t -\LL)w\leq0$ in $\R^n\times (-T,0)$ in the viscosity sense.
\end{proof}

We will also need the following 1D computation.

\begin{lem}\label{lem-1D-critical}
Let $s=\frac12$ and $\LL$ as in \eqref{L0}-\eqref{ellipt}, with $K$ homogeneous.
Let $e\in \mathbb S^{n-1}$, $\varv\geq0$, and assume that the function
\[u_\circ(x,t) = (x\cdot e+\varv t)_+^{1+\gamma}\]
solves $\partial_t u_\circ-\LL u_\circ = 0$ in $\{x\cdot e+\varv t>0\}$, for some $\gamma\in(0,1)$.
Then the exponent $\gamma$ is given by
\begin{equation}\label{gamma-exp}
\gamma(\LL,e,\varv)=\frac12+\frac{1}{\pi}\arctan\left(\frac{\varv}{\mathcal A(e)}\right),
\end{equation}
where $\mathcal A(\xi)$ is the Fourier symbol of the operator $-\LL$.
\end{lem}

\begin{proof}
Notice that for such function $u_\circ$ we have $\partial_t u_\circ(x,0) = \varv(e\cdot \nabla u_\circ)(x,0)$, hence the function $w(x):=u_\circ(x,0)$ solves $-\LL w+\varv e\cdot \nabla w=0$ in $\{x\cdot e>0\}$.
Since the Fourier symbol of the operator $-L+\varv e\cdot \nabla$ is given by $\mathcal A(\xi)+\varv e\cdot\xi$, the value of the exponent $\gamma$ follows from \cite[Corollary~4.6]{DRSV}.
\end{proof}

We will also use the following:

\begin{lem}\label{lem-growth-par-blowdown}
Let $s,\mu>0$, and let $w\in {\rm Lip}_{\rm loc}(\mathcal Q_\infty)$ be such that
\[R\|\nabla w\|_{L^\infty(\mathcal Q_R)}+R^{2s}\|\partial_t w\|_{L^\infty(\mathcal Q_R)} \leq CR^\mu\qquad \textrm{for all}\quad R\geq1.\]
Then there is a sequence $R_m\to\infty$ for which the rescaled functions 
\[\tilde w_m(x,t):= \frac{w(R_m x,R_m^{2s}t)}{R_m\|\nabla w\|_{L^\infty(\mathcal Q_{R_m})}+R_m^{2s}\|\partial_t w\|_{L^\infty(\mathcal Q_{R_m})}}\]
satisfy\footnote{Notice that, by construction, the functions $\tilde w_m$ satisfy $\|\nabla \tilde w_m\|_{L^\infty(\mathcal Q_1)} + \|\partial_t \tilde w_m\|_{L^\infty(\mathcal Q_1)}=1$.}
\[R\|\nabla \tilde w_m\|_{L^\infty(\mathcal Q_R)} + R^{2s}\|\partial_t \tilde w_m\|_{L^\infty(\mathcal Q_R)} \leq 2R^{\mu}\qquad \textrm{for all}\quad R\geq1.\]
\end{lem}

\begin{proof}
For $R\geq1$ consider the quantity
\[\theta(R):= \sup_{\rho\geq R} \frac{\rho\|\nabla w\|_{L^\infty(\mathcal Q_\rho)}+\rho^{2s}\|\partial_t w\|_{L^\infty(\mathcal Q_\rho)}}{\rho^{\mu}}<\infty.\]
By definition of $\theta$, for all all $m\in\mathbb N$ there is $R_m\geq m$ such that 
\[ \frac{R_m\|\nabla w\|_{L^\infty(\mathcal Q_{R_m})}+R_m^{2s}\|\partial_t w\|_{L^\infty(\mathcal Q_{R_m})}}{R_m^{\mu}} \geq \frac12 \theta(m).\]
Then, since $\theta$ is nonincreasing, such sequence $R_m$ satisfies 
\begin{align*}
R\|\nabla \tilde w_m\|_{L^\infty(\mathcal Q_R)} + R^{2s}\|\partial_t \tilde w_m\|_{L^\infty(\mathcal Q_R)}&= \frac{R_mR\|\nabla w\|_{L^\infty(\mathcal Q_{R_mR})}+(R_mR)^{2s}\|\partial_t w\|_{L^\infty(\mathcal Q_{R_mR})}}{R_m\|\nabla w\|_{L^\infty(\mathcal Q_{R_m})}+R_m^{2s}\|\partial_t w\|_{L^\infty(\mathcal Q_{R_m})}} \\
& \leq \frac{(R_mR)^{\mu}\theta(R_mR)}{\frac12R_m^\mu\theta(m)} \leq 2R^\mu \qquad \textrm{for all}\quad R\geq1,
\end{align*}
as wanted.
\end{proof}

We can now prove the following classification result for blow-ups, which is new even in the special case of $\partial_t+\sqrt{-\Delta}$.

\begin{prop}\label{prop-classification}
Let $s=\frac12$ and $\LL$ as in \eqref{L0}-\eqref{ellipt}, with $K$ homogeneous.

Let $u_\circ\in {\rm Lip}_{\rm loc}(\R^n\times\R)$ be a function satisfying:
\begin{itemize}
\item[$\bullet$]  $u_\circ$ is nonnegative, monotone, and convex:
\[u_\circ\geq0,\quad \partial_t u_\circ\geq0,\quad \textrm{and} \quad D^2_{x,t} u_\circ \ge 0\qquad \textrm{in}\quad \R^n\times \R,\]
\[\textrm{with}\quad (0,0)\in \partial\{u_\circ>0\}.\]

\item[$\bullet$]   $u_\circ$ solves
\[(\partial_t -\LL)(D_{h,\tau} u_\circ)\leq0\qquad \textrm{in}\quad \{u_\circ>0\}\]
for all $h\in \R^n$ and $\tau\in \R$, where
\[D_{h,\tau} u_\circ(x,t)={\textstyle \frac{u_\circ(x,t)-u_\circ(x-h,t-\tau)}{|h|+|\tau|}}.\]

\item[$\bullet$] $u_\circ$ has a controlled growth at infinity: there exists $\delta>0$ such that
\[\|\nabla u_\circ\|_{L^\infty(\mathcal Q_R)}+\|\partial_t u_\circ\|_{L^\infty(\mathcal Q_R)} \le R^{1-\delta} \qquad \textrm{for all}\quad R\ge 1.\]
\end{itemize}
Then, up to a translation,
\[u_\circ(x,t)=\kappa(x\cdot e+\varv t)_+^{1+\gamma}\]
for some $e\in \mathbb S^{n-1}$, $\varv\geq0$, $\kappa\in[0,1]$, and with $\gamma\in[\frac12,1)$ given by \eqref{gamma-exp}.
\end{prop}

\begin{proof}
First, notice that the set $\{u_\circ=0\}\ni 0$ is a convex subset of $\R^n\times\R$.
Then, we consider a ``blow-down'' $u_\infty$ of our function $u_\circ$, as follows.

By Lemma \ref{lem-growth-par-blowdown}, we can find a sequence $R_m\to\infty$ such that
\[u_m(x):=\frac{u_\circ(R_mx,R_mt)}{R_m\|\nabla u_\circ\|_{L^\infty(\mathcal Q_{R_m})}+R_m\|\partial_t u_\circ\|_{L^\infty(\mathcal Q_{R_m})}}\]
satisfies 
\[\|\nabla u_m\|_{L^\infty(\mathcal Q_1)}+\|\partial_t u_m\|_{L^\infty(\mathcal Q_1)}=1,\]
\[\|\nabla u_m\|_{L^\infty(\mathcal Q_R)}+\|\partial_t u_m\|_{L^\infty(\mathcal Q_R)}\leq 2R^{1-\delta}\qquad\textrm{for all}\quad R\geq1,\]
and $(\partial_t -\LL)(D_{h,\tau} u_m)=0$ in $\{u_m>0\}=\frac{1}{R_m}\{u_\circ>0\}$.
Moreover, by convexity, the nondegeneracy of the gradient implies
\[\|u_m\|_{L^\infty(\mathcal Q_2)}\geq1.\]
Also, still by convexity, the functions $u_m$ converge (up to a subsequence) locally uniformly in $\R^n\times \R$ to a function $u_\infty(x,t)$ that satisfies
\[\|u_\infty\|_{L^\infty(\mathcal Q_2)}\geq1\quad \textrm{and}\quad \|\nabla u_\infty\|_{L^\infty(\mathcal Q_R)}+\|\partial_t u_\infty\|_{L^\infty(\mathcal Q_R)}\leq 2R^{1-\delta}\qquad\textrm{for all}\quad R\geq1.\]
Moreover, the  ``blow-down'' sequence $\frac{1}{R_m}\{u_\circ=0\}$ converges to a closed convex cone $\Sigma=\{u_\infty=0\}$ with vertex at the origin.
Furthermore, since $\partial_t u_\infty\geq0$, then $\Sigma$ satisfies a monotonicity property in time, too.

We now separate the proof into two cases:

\vspace{2mm}

\noindent \emph{Case 1}.
Assume that the convex cone $\Sigma$ has nonempty interior.
Then there exist $n+1$ independent directions $\omega_i\in \mathbb S^{n}$, $i=1,...,n+1$, such that $-\omega_i\in \mathring{\Sigma}$.
Thus, by convexity of $u_\infty$, 
\[v_i:=\partial_{\omega_i}  u_\infty \geq0\qquad \textrm{in}\quad\R^n\times \R.\]
Moreover, at least one of them is not identically zero, say $v_n\not\equiv0$.

We first claim that the functions $v_i$ are continuous functions. Indeed, fix $t_\oo <0$ and let $K_\oo: = \{x\in \R^n \ : \ u(x, t_\circ) =0\}$.
Since $\partial_t u_\circ\ge 0$ and $u_\circ \ge 0$, the zero set (as a subset of space-time) $\{u_\circ=0\}$ contains the cylinder $K_\circ \times (-\infty, t_\circ]$. 
Also, the functions $(D_{h,\tau}  u_\circ)_+$ are continuous subsolutions that vanish on $K_\circ \times (-\infty, t_\circ]$. Hence, by standard barrier arguments,\footnote{For instance, one may use a constant-in-time barrier obtained by truncating the elliptic homogeneous supersolution from Lemma \ref{lemsuper1}; see e.g. the proof of Theorem 4.1 in \cite{AuR} for a very similar argument.} for every $R>0$ we obtain
\[
(D_{h,0}  u_\circ)_+(x,t) \le C' d_{K_\oo}^\theta(x) \mbox{ for  $|x|\le R$ and $t\le t_\oo<0$}
\]
where $d_{K_\oo}$ is the distance to the convex set $K_\oo$, and the constants $C'$ and $\theta>0$ possibly depend on $R$ and $t_\circ$. 
Since the partial derivatives of $u_\circ$ are smooth inside $\{u_\circ >0\}$ (they satisfy a parabolic translation invariant equation, recall Remark~\ref{rem:incremental quotient}), letting $|h|\to 0$ and $h/|h| \to \pm \omega_i$, thanks to the arbitrariness of $R$ and $t_\oo$ we deduce that the functions $v_i$ vanish  continuously on  the boundary $\partial \{u_\circ >0\}$, and the claim follows.

Second, since $v_n$ is not identically zero, we can apply Proposition \ref{prop-uniqueness-positive-sols-critical} to deduce that 
\[v_i\equiv \kappa_i v_n\qquad \textrm{in}\quad \R^n\times (-\infty,0],\quad \textrm{for}\quad i=1,...,n.\]
This means that $u_\infty$ is a 1D function for $t\leq0$, i.e.,  $u_\infty(x,t)=U(x\cdot e+\varv t)$ in $\R^n\times (-\infty,0]$.
Therefore, given $(h, \tau) \in \R^n\times \R$ parallel to the hyperplane $\{x\cdot e+\varv t=0\}$, consider the function $w:=(h,\tau)\cdot(\nabla u_\infty, \partial_t u_\infty)$. Then $\partial_t w-\LL w=0$ in $\R^n\times (0,\infty)$ (cp. Remark~\ref{rem:incremental quotient}) and $w\equiv0$ in $\R^n\times \{t=0\}$.
By uniqueness of solutions to such initial value problem, we deduce that $w\equiv0$ in $\R^n\times \R$. Since $(h, \tau) \in \R^n\times \R$ is an arbitrary vector tangent to $\{x\cdot e+\varv t=0\}$,
 this proves that $u_\infty(x,t)=U(x\cdot e+\varv t)$ in $\R^n\times \R$ and $\{u_\infty=0\}\supset \{x\cdot e+\varv t\leq0\}.$

As a consequence, since $0\in\partial\{u_\circ>0\}$, it follows by convexity that $\{u_\circ=0\} \supset \{u_m=0\}$ for every $m\ge 1$, therefore $$\{u_\circ=0\} \supset \{u_m=0\} \to \{u_\infty=0\}\supset \{x\cdot e+\varv t\leq0\}.$$ Hence $\{u_\circ=0\}\supset \{x\cdot e+\varv t=0\}$, and by the convexity of $u_\circ$   we deduce that $u_\circ$ is a 1D function for the form $u_\circ(x,t)=U_\circ(x\cdot e+\varv t)$ (see \cite[Lemma 5.28]{FR20}).

Finally, thanks to Lemma \ref{lem-1D-critical} we find that $u_\circ(x,t)=\kappa(x\cdot e+\varv t)_+^{1+\gamma}$ with $\gamma\in[\frac12,1)$, as desired.

\vspace{2mm}

\noindent \emph{Case 2}.
Assume that the cone $\Sigma$ has empty interior.
Since by convexity it is contained in a hyperplane $\Gamma\subset\R^n\times\R$, it follows that
\[(\partial_t -\LL)(D_{h,\tau} u_\infty)\leq 0\qquad \textrm{in}\quad (\R^n\times \R)\setminus \Gamma.\]
Hence, since $D_{h,\tau} u_\infty\in {\rm Lip}_{\rm loc}(\R^n\times \R)$,  Lemma \ref{lem-solution-across-critical} implies $(\partial_t -\LL)(D_{h,\tau} u_\infty)\leq 0$ in $\R^n\times \R$, and hence
\[(\partial_t -\LL)(\nabla_{x,t} u_\infty)=0\qquad \textrm{in}\quad \R^n\times \R\]
(cp. Remark~\ref{rem:incremental quotient}).
Thanks to the growth control on $\nabla_{x,t} u_\infty$, the Liouville theorem for nonlocal parabolic equations implies that $u_\infty$ is affine.
However, this contradicts the fact that $u_\infty(0)=0$, $u_\circ\geq0$, and $\|u_\infty\|_{L^\infty(\mathcal Q_2)}\geq1$.
Thus, Case 2 cannot happen and the proposition is proved.
\end{proof}

\subsection{Classification of blow-ups for $s>\frac12$}

We next establish the classification of blow-ups for $s>\frac12$.
In this case, the scaling is subcritical.

We first need the following (simpler) version of Lemma \ref{lem-solution-across-critical}.

\begin{lem}\label{lem-solution-across-subcritical}
Let $s\in(\frac12,1)$ and $\LL$ as in \eqref{L0}-\eqref{ellipt}, with $K$ homogeneous.
Let $e\in \mathbb S^{n-1}$, and $\Gamma\subset \{x\cdot e=0\}\subset \R^n$.
Assume $w\in {\rm Lip}_{\rm loc}(\R^n\times \R)$ is a viscosity solution of
\[\partial_t w - \LL w = 0 \qquad \textrm{in}\quad \big(\R^n\setminus \Gamma\big)\times\R.\]
Then $\partial_t w - \LL w=0$ in $\R^n\times \R$.
\end{lem}

\begin{proof}
The proof is analogous to the one of Lemma \ref{lem-solution-across}.
\end{proof}

The classification of blow-ups for $s>\frac12$ is contained in the following result.

\begin{prop}\label{prop-classification>1/2}
Let $s\in(\frac12,1)$ and $\LL$ as in \eqref{L0}-\eqref{ellipt}, with $K$ homogeneous.

Let $u_\circ\in {\rm Lip}(\R^n\times(-\infty,0))$ be a function satisfying:

\begin{itemize}

\item[$\bullet$]  $u_\circ$ is nonnegative, monotone, and convex:
\[u_\circ\geq0,\quad \partial_t u_\circ\geq0,\quad \textrm{and} \quad D^2_{x,t} u_\circ \ge 0\qquad \textrm{in}\quad \R^n\times (-\infty,0),\]
\[\textrm{with}\quad (0,0)\in \partial\{u_\circ>0\}.\]

\item[$\bullet$]  $u_\circ$ solves
\[(\partial_t -\LL)(D_{h,\tau} u_\circ)\leq0\qquad \textrm{in}\quad \{u_\circ>0\}\]
for all $h\in \R^n$ and $\tau \in \R$, where
\[D_{h,\tau} u_\circ(x,t)={\textstyle \frac{u_\circ(x,t)-u_\circ(x-h,t-\tau)}{|h|+|\tau|}}.\]

\item[$\bullet$]  $u_\circ$ has a controlled growth at infinity: there exists $\delta>0$ such that
\[R\|\nabla u_\circ\|_{L^\infty(B_R\times (-R^{2s},R^{2s}))} + R^{2s}\|\partial_t u_\circ\|_{L^\infty(B_R\times (-R^{2s},R^{2s}))}  \le R^{2-\delta} \qquad \textrm{for all}\quad R\ge 1.\]
\end{itemize}
Then, up to a translation,
\[u_\circ(x,t)=\kappa(x\cdot e)_+^{1+s}\]
for some $e\in \mathbb S^{n-1}$ and $\kappa\in[0,1]$.
\end{prop}

This result was known only for the fractional Laplacian \cite{BFR2}, and the proof in such a case used crucially in some steps the extension property for the fractional Laplacian, as well as the regularity of solutions obtained in \cite{CF}.
Here, instead, we establish the result by combining ideas from \cite{BFR2} with the ones used in the proof of Proposition \ref{prop-classification}.

\begin{proof}[Proof of Proposition \ref{prop-classification>1/2}]
First, the set $\{u_\circ=0\}\ni (0,0)$ is a convex subset of $\R^n\times\R$.

If such set contains the whole line $\{x_1=\ldots=x_n=0\}$, then it follows by convexity (see, e.g. \cite[Lemma 5.28]{FR20}) that the function $u_\circ$ is independent of $t$, so the result is a consequence of Proposition~\ref{prop-classification-ell}.

Otherwise, if the convex set $\{u_\circ=0\}$ does not contain the line $\{x_1=\ldots=x_n=0\}$, then there exist $\varv,M>0$ such that
\begin{equation}\label{aldgjh}
\{u_\circ=0\}\subseteq \{x\cdot e+\varv t\leq M\}.
\end{equation}
Let us now consider the blow-down sequence $\tilde u_m$ given by Lemma \ref{lem-growth-par-blowdown}, with $R_m\to\infty$.
Such a sequence satisfies the same assumptions as $u_\circ$. Also, it follows from \eqref{aldgjh} that
\[\{\tilde u_m=0\}=\{(x,t):u_\circ(R_mx,R_m^{2s}t)=0\}\subseteq \{R_m^{1-2s}x\cdot e+\varv t\leq MR_m^{-2s}\}.\]
By convexity of the functions $\tilde u_m$, up to a subsequence we have that $\tilde u_m\to u_\infty$ locally uniformly in $\R^n\times \R$, where $u_\infty$ satisfies the same assumptions as $u_\circ$ and, in addition,
\[\{u_\infty=0\}\subset \{t\leq0\}.\]
Furthermore, by the construction of $\tilde u_m$ in Lemma \ref{lem-growth-par-blowdown}, 
we deduce that $\|u_\infty\|_{L^\infty(\mathcal Q_2)}\geq1$.
We now separate the proof into two cases:

\vspace{2mm}

\noindent \emph{Case 1}.
Assume first that $u_\infty$ is not identically zero for $t\leq0$.
In this case, since the set $\{u_\infty=0\}\cap \{t\leq0\}$ is the blow-down (with a parabolic scaling) of a convex set, it follows that 
$\{u_\infty=0\}$ is a cone of the form $\Sigma\times \R_-$, where $\Sigma\subset\R^n$ is a convex cone with vertex at the origin.
Therefore we can apply the boundary Harnack Theorem \ref{bdryharnack}, which exactly as in the proof of Proposition \ref{prop-classification} yields $u_\infty(x,t)=U(x\cdot e)$ for $t\leq 0$.

Thus, we proved that $\partial_t u_\infty\equiv0$ for $t\leq 0$.
Also, since $u_\infty$ never vanishes for positive times,
$$
(\partial_t -\LL)(\partial_t u_\infty)= 0\qquad \text{in }\R^n\times\R_+
$$
(cp. Remark~\ref{rem:incremental quotient}).
Hence, by uniqueness of solutions to the initial value problem, we deduce $\partial_t u_\infty\equiv0$ in $\R^n\times \R$, a contradiction to the fact that $u_\infty>0$ for $t>0$.

\vspace{2mm}

\noindent \emph{Case 2}.
Assume  that $u_\infty\equiv 0$ for all $t\leq0$.
Then $\partial_t u_\infty\equiv0$ for $t\leq 0$, and we conclude as at the end of Case 1.
\end{proof}

\subsection{Almost-optimal regularity estimates}

Once we have the classification of blow-ups we can show the almost-optimal regularity of solutions.
For this, we first need the following result, which is a simple variant of Lemma~\ref{lem-growth-ell}.

\begin{lem}\label{lem-growth-par}
Let $s,\mu>0$, and let $\mathcal Q_r:=B_{r}\times (-r^{2s},r^{2s})$.
Let $w_k\in {\rm Lip}(\mathcal Q_1)$ be a sequence of functions such that
\begin{equation}
\label{eq:unif Lip}
\sup_k \|\nabla w_k\|_{L^\infty(\mathcal Q_{1})}+\|\partial_t w_k\|_{L^\infty(\mathcal Q_{1})}<\infty
\end{equation}
but
\[\sup_k\sup_{r\in (0,1)} \frac{r\|\nabla w_k\|_{L^\infty(\mathcal Q_r)}+r^{2s}\|\partial_t w_k\|_{L^\infty(\mathcal Q_r)}}{r^\mu}=\infty.\]
Then, there are subsequences $w_{k_m}$ and $r_m\to0$ such that $$r_m^{1-\mu}\|\nabla w_{k_m}\|_{L^\infty(\mathcal Q_{r_m})}+r_m^{2s-\mu}\|\partial_t w_{k_m}\|_{L^\infty(\mathcal Q_{r_m})}\geq 1$$ and for which the rescaled functions 
\[\tilde w_m(x,t):= \frac{w_{k_m}(r_m x,r_m^{2s}t)}{r_m\|\nabla w_{k_m}\|_{L^\infty(\mathcal Q_{r_m})}+r_m^{2s}\|\partial_t w_{k_m}\|_{L^\infty(\mathcal Q_{r_m})}}\]
satisfy\footnote{Notice that, by construction, the functions $\tilde w_m$ satisfy $\|\nabla \tilde w_m\|_{L^\infty(\mathcal Q_1)} + \|\partial_t \tilde w_m\|_{L^\infty(\mathcal Q_1)}=1$.}
\[R\|\nabla \tilde w_m\|_{L^\infty(\mathcal Q_R)} + R^{2s}\|\partial_t \tilde w_m\|_{L^\infty(\mathcal Q_R)} \leq 2R^{\mu}\qquad \textrm{for all}\quad R\in (1,r_m^{-1}).\]
\end{lem}

\begin{proof}
For every $m\in \mathbb N$ let $k_m$ and $r_m \geq \frac1m$ be such that 
\[ \begin{split}
r_m^{1-\mu} \|\nabla w_{k_m}&\|_{L^\infty(\mathcal Q_{r_m})} +r_m^{2s-\mu} \|\partial_t w_{k_m}\|_{L^\infty(\mathcal Q_{r_m})} \geq \\
&\geq \frac12 \sup_k \sup_{r\in (\frac1m,1)} \left(r^{1-\mu}\|\nabla w_k\|_{L^\infty(\mathcal Q_{r})}+r^{2s-\mu}\|\partial_t w_k\|_{L^\infty(\mathcal Q_{r})}\right) \\
& \geq \frac12 \sup_k \sup_{r\in (r_m,1)} \left(r^{1-\mu}\|\nabla w_k\|_{L^\infty(\mathcal Q_{r})}+r^{2s-\mu}\|\partial_t w_k\|_{L^\infty(\mathcal Q_{r})}\right).
\end{split}\]
As in the proof of Lemma~\ref{lem-growth-ell}, it follows from \eqref{eq:unif Lip} that $r_m \to 0$ as $m\to \infty$. 
Also, by construction of $r_m$ and $k_m$, \[\begin{split} 
r_m^{1-\mu} \|\nabla w_{k_m}&\|_{L^\infty(\mathcal Q_{r_m})} +r_m^{2s-\mu} \|\partial_t w_{k_m}\|_{L^\infty(\mathcal Q_{r_m})} 
\geq \\
& \geq\frac12 \left(r^{1-\mu}\|\nabla w_k\|_{L^\infty(\mathcal Q_{r})}+r^{2s-\mu}\|\partial_t w_k\|_{L^\infty(\mathcal Q_{r})}\right)
 \end{split}\]
for all $r\in (r_m,1)$ and for all $k$.
Thus, for any $R\in (1,r_m^{-1})$ we have
\[\begin{split} 
R\|\nabla \tilde w_m &\|_{L^\infty(\mathcal Q_R)}+ R^{2s}\|\partial_t \tilde w_m\|_{L^\infty(\mathcal Q_R)} = \\
&= \frac{r_mR\|\nabla w_{k_m}\|_{L^\infty(\mathcal Q_{R r_m})}+(r_mR)^{2s}\|\partial_t w_{k_m}\|_{L^\infty(\mathcal Q_{R r_m})}}{r_m\|\nabla w_{k_m}\|_{L^\infty(\mathcal Q_{r_m})}+r_m^{2s}\|\partial_t w_{k_m}\|_{L^\infty(\mathcal Q_{r_m})}}
\leq 2R^\mu,
\end{split}\]
and we are done.
\end{proof}

We can now establish the almost-optimal regularity of solutions.
We first recall the notion of the parabolic H\"older seminorm: given $\beta\in (0,1)$,
\begin{equation}\label{Holder-par}
\|w\|_{C^\beta_{\rm par}(K)}:= \sup_{(x,t),(y,\tau)\in K} \frac{|w(x,t)-w(y,\tau)|}{|x-y|^{\beta}+|t-\tau|^{\frac{\beta}{2s}}}.
\end{equation}

\begin{cor}\label{cor-almost-optimal-par}
Let $s \in[\frac12,1)$ and $\LL$ as in \eqref{L0}-\eqref{ellipt}, with $K$ homogeneous.
Let $\delta>0$, and let $u\in {\rm Lip}_{\rm loc}(\R^n\times(-2,2))$, with 
\begin{equation}
\label{eq:growth R}
R\|\nabla u\|_{L^\infty(\mathcal Q_R \cap \{|t|<2\})}+R^{2s}\|\partial_t u\|_{L^\infty(\mathcal Q_R\cap \{|t|<2\})} \le R^{2-\delta} \qquad \textrm{for all}\quad R\ge 1,
\end{equation}
satisfy $u\geq0$, $\partial_t u\geq0$, and $D^2_{x,t} u\geq-{\rm Id}$ in $\mathcal Q_2$, $\partial_t u-\LL u=f$ in $\{u>0\}\cap \mathcal Q_2$ and $u_t-\LL u \geq f$ in~$\mathcal Q_2$, with $|\nabla f|+|\partial_t f|\leq 1$.
Then, for any $\varepsilon>0$ we have
\[\|u\|_{C^{1+s-\varepsilon}_{\rm par}(\mathcal Q_1)} := \|\nabla u\|_{C^{s-\varepsilon}_{\rm par}(\mathcal Q_1)}+\|\partial_t u\|_{C^{1-s-\varepsilon}_{\rm par}(\mathcal Q_1)}\leq C_\varepsilon,\]
with $C$ depending only on $n$, $s$, $\varepsilon$, and the ellipticity constants.
\end{cor}

\begin{proof}
Let $\mu:=1+s-\varepsilon$. Up to reducing $\delta$, we can assume that $1-\delta\geq s$ (in particular, $1-\delta \geq \mu-1$).

\vspace{2mm}

\noindent
- \emph{Step 1.}
We first prove that, at every free boundary point $(x_\circ,t_\circ)\in \partial\{u>0\}\cap \mathcal Q_1$, we have
\begin{equation}\label{wqergij}
r\|\nabla u\|_{L^\infty(\mathcal Q_r(x_\circ,t_\circ))}+r^{2s}\|\partial_t u\|_{L^\infty(\mathcal Q_r(x_\circ,t_\circ))} \leq Cr^\mu,
\end{equation}
for $r\in(0,1)$, with $C$ depending only on $n$, $s$, $\varepsilon$, $\lambda$, and $\Lambda$. 

The proof is very similar to the one of Corollary~\ref{cor-almost-optimal-ell}.
Indeed, assume by contradiction that \eqref{wqergij} fails.
Then, we can find sequences $u_k$, $\LL_k$, and $f_k$, satisfying the assumptions, with $0\in \partial\{u_k>0\}$, and such that 
\[\sup_k\sup_{r\in(0,1)} \frac{r\|\nabla u_k\|_{L^\infty(\mathcal Q_{r})}+r^{2s}\|\partial_t u_k\|_{L^\infty(\mathcal Q_{r})}}{r^\mu}=\infty.\]
Also, the uniform semiconvexity assumption $D^2_{x,t} u_k\geq-{\rm Id}$ 
implies that the functions $u_k$ are uniformly Lipschitz in $\mathcal Q_1$.
Hence, thanks to Lemma \ref{lem-growth-par}, there exist sequences $k_m$ and $r_m\to0$ such that the functions $\tilde u_m(x,t)$
satisfy $\|\nabla \tilde u_m\|_{L^\infty(\mathcal Q_1)}+\|\partial_t \tilde u_m\|_{L^\infty(\mathcal Q_1)}=1$ and
\[R\|\nabla \tilde u_m\|_{L^\infty(\mathcal Q_R)} + R^{2s}\|\partial_t \tilde u_m\|_{L^\infty(\mathcal Q_R)} \leq CR^{\mu}\qquad \textrm{for all}\quad R\in (1,r_m^{-1}).\]
Moreover
\[D^2_{x,t} \tilde u_m \geq - r_m^{2-\mu}{\rm Id} \longrightarrow 0 \qquad \text{in }\mathcal Q_{2/r_m},\]
\[R\|\nabla u\|_{L^\infty(\mathcal Q_R \cap \{|t|<r_m^{-2s}\})}+R^{2s}\|\partial_t u\|_{L^\infty(\mathcal Q_R\cap \{|t|<r_m^{-2s}\})} \le R^{2-\delta} \qquad \textrm{for all}\quad R\ge r_m^{-1},\]
$$
\text{$(\partial_t-\LL_{k_m}) \tilde u_m= f_m$ in $\{u_m>0\}\cap \mathcal Q_{2/r_m}$},\qquad
\text{$(\partial_t-\LL_{k_m}) \tilde u_m\geq  f_m$ in $\mathcal Q_{2/r_m}$,}
$$
$$
\text{ $|\nabla f_m|\leq r_m^{1+2s-\mu}\to 0$},\qquad\text{and}\qquad\text{$|\partial_t f_m|\leq r_m^{4s-\mu}\to 0$}.
$$
These last two conditions imply that $(\partial_t-\LL_{k_m})(D_{h,\tau} \tilde u_m)\leq r_m^{1+2s-\mu}\to0$ in $\{\tilde u_m>0\}\cap \mathcal Q_{1/r_m}$, where
\[D_{h,\tau} \tilde u_m(x,t)={\textstyle \frac{\tilde u_m(x,t)-\tilde u_m(x-h,t-\tau)}{|h|+|\tau|}}.\]
Hence, by semi-convexity, a subsequence of the functions $\tilde u_m$ will converge locally uniformly in $\R^n\times \R$ to a limiting convex function $\tilde u_\circ$
satisfying
\[R\|\nabla \tilde u_\circ\|_{L^\infty(\mathcal Q_R)} + R^{2s}\|\partial_t \tilde u_\circ\|_{L^\infty(\mathcal Q_R)} \leq CR^{\mu}\qquad \textrm{for all}\quad R\in \geq 1.\]
 Using Lemma \ref{lem-limits-par} we see that $\tilde u_\circ$ satisfies the hypotheses of Proposition \ref{prop-classification} or \ref{prop-classification>1/2}, so it follows from the classification of blow-ups and the growth assumption above that $\tilde u_\circ\equiv0$. 
 
On the other hand, by convexity we see that $\|\nabla \tilde u_\circ\|_{L^\infty(\mathcal Q_2)}+\|\partial_t \tilde u_\circ\|_{L^\infty(\mathcal Q_2)}\geq 1$, a contradiction that proves 
\eqref{wqergij}.

\vspace{2mm}

\noindent
- \emph{Step 2.}
We now combine \eqref{wqergij} with interior regularity estimates to establish the result.
While in the elliptic case this is rather standard, here the argument is slightly more delicate and we provided all details.

Let $(x_1,t_1)$ be any point in $\{u>0\}\cap \mathcal Q_1$, and let $r>0$ be largest number for which $\mathcal Q_r(x_1,t_1)\subset \{u>0\}$.
Let $(x_\circ,t_\circ)\in \partial\{u>0\}\cap \partial \mathcal Q_r(x_1,t_1)$.
Then, since $(\partial_t -\LL)(\nabla u)=\nabla f$ in $\mathcal Q_r$, by interior regularity estimates for nonlocal parabolic equations (see for instance \cite{ChaDav}) we have 
\[r\|D^2_x u\|_{L^\infty(\mathcal Q_{r/2}(x_1,t_1))} \leq C\bigg(r^{2s}\|\nabla f\|_{L^\infty}+\sup_{R\geq1} R^{\varepsilon-2s}\|\nabla u\|_{L^\infty(\mathcal Q_{rR}(x_1,t_1)\cap \{|t|<2\} )}\bigg),\]
and analogous estimates hold for $\nabla \partial_t u$ and $\partial_{tt}u$.
By \eqref{wqergij}, it follows that for $R\in (1,r^{-1})$ we have
\[\|\nabla u\|_{L^\infty(\mathcal Q_{rR}(x_1,t_1))} \leq C(rR)^{\mu-1},\]
that combined with \eqref{eq:growth R} gives (without loss of generality, we can assume that $\varepsilon \leq \delta$)
\[\sup_{R\geq1} R^{\varepsilon-2s}\|\nabla u\|_{L^\infty(\mathcal Q_{rR}(x_1,t_1) \cap \{|t|<2\})} \leq r^{\mu-1}.\]
Since $\|\nabla f\|_{L^\infty}\leq 1$, this yields
\[\|D^2_x u\|_{L^\infty(\mathcal Q_{r/2}(x_1,t_1))} \leq Cr^{\mu-2}.\]
Moreover, with the exact same argument (using the regularity estimates for$\nabla \partial_t u$ and $\partial_{tt}u$),  we find 
\[\|\nabla \partial_t u\|_{L^\infty(\mathcal Q_{r/2}(x_1,t_1))} \leq Cr^{\mu-1-2s}\qquad \textrm{and}\qquad 
\|\partial_{tt}u\|_{L^\infty(\mathcal Q_{r/2}(x_1,t_1))} \leq Cr^{\mu-4s}.\]
Since these bounds hold at all points $(x_1,t_1)\in \{u>0\}\cap \mathcal Q_1$, we conclude that
\[\|\nabla u\|_{C^{\mu-1}_{\rm par}(\mathcal Q_1)} + \|\partial_t u\|_{C^{\mu-2s}_{\rm par}(\mathcal Q_1)} \leq C,\]
as wanted.
\end{proof}

\subsection{Regularity of the free boundary}

The next step is to show that the free boundary is $C^{1,\tau}$ near nondegenerate points.
Recall that $\mathcal Q_R=B_R\times  (-R^{2s},R^{2s})$.

%

\begin{prop}\label{lemm-main-par-2}
Let $s$, $\LL$, $\delta$, $u$, $u_\circ$, and $\kappa$ be as in Theorem \ref{thm-main}, and let $\rho_\circ\geq1$.
Assume that $\kappa\geq \kappa_\circ>0$ and
\[
\|u- u_\circ\|_{{\rm Lip} (\mathcal Q_{R_\circ})} \le \eps,
\]
with $\eps>0$ small enough.
Then, if $R_\circ$ is large enough, the free boundary $\partial \{u>0\}$ is a $C^{1,\tau}$ graph in $\mathcal Q_{\rho_\circ}$, with constants depending only on  $n$, $s$, $\delta$, $\lambda$, $\Lambda$, $\rho_\circ$, and $\kappa_\circ$.
\end{prop}

\begin{proof}
By assumption, we have
\[|\partial_t u-\partial_t u_\circ|+|\nabla u-\nabla u_\circ| \leq \varepsilon\qquad\textrm{in}\quad \mathcal Q_{R_\circ}.\]
In particular, for any direction $e'\in \mathbb S^{n-1}$ such that $e'\cdot e\geq\frac12$ we have
\[|\partial_{e'} u-\partial_{e'} u_\circ|\leq \eps\qquad\textrm{in}\quad \mathcal Q_{R_\circ},\]
\[\partial_{e'} u_\circ \geq0 \quad \textrm{in}\quad \R^{n+1}, \qquad \textrm{and} \quad \partial_{e'} u_\circ \geq c_1\kappa \quad \textrm{in}\quad \{x\cdot e+\varv t\geq {\textstyle\frac12}\}.\]
Recall also that $\varv\leq \varv_\circ$, with $\varv_\circ$ depending only on $\delta$, $\lambda$, and  $\Lambda$.

Thus, if $\eps$ is small, we have that $v:=\partial_{e'} u$ \,and $E:=\{u=0\}\cap \mathcal Q_{R_\circ}$ satisfy
\[|\partial_t v-\LL v|\leq \eta\quad \textrm{in}\quad \mathcal Q_{R_\circ}\setminus E, \qquad  v\equiv0\quad \textrm{in}\quad E,\]
\[ v \geq c_2\kappa>0 \quad \textrm{in}\quad \{x\cdot e+\varv t\geq {\textstyle\frac12}\}\cap \mathcal Q_{R_\circ}, \qquad v \geq -\eps \quad \textrm{in}\quad \mathcal Q_{R_\circ},\]
and 
\[|v(x,t)|\leq |x|^{1-\delta}+|t|^{\frac{1-\delta}{2s}}\qquad \textrm{in}\quad \R^{n+1}\setminus \mathcal Q_{R_\circ}.\]
This means that, given any $\rho_\circ>1$, if $\eta$ is small enough we can apply Proposition \ref{step} to the (same) functions $v_i(x,t):=v(\rho_\circ x,\rho_\circ^{2s}t)$, $i=1,2$,  to deduce that $v\geq 0$ in $\mathcal Q_{\rho_\circ/2}$.
That is,
\[\partial_{e'} u\geq0 \qquad \textrm{in}\quad \mathcal Q_{\rho_\circ/2}\]
for all $e'\in \mathbb S^{n-1}$ such that $e'\cdot e\geq\frac12$.
Since we also have $\partial_t u\geq0$, this means that the free boundary $\partial\{u>0\}$ is a Lipschitz graph in $\mathcal Q_{\rho_\circ/2}$.

Finally, taking $\rho_\circ>1$ large enough, we can apply the boundary Harnack (Theorem~\ref{bdryharnack}) to the functions $\partial_{e'} u$ and $\partial_{e} u$, 
and to $\partial_{t} u$ and $\partial_{e} u$, to deduce that 
\[\left\|\frac{\partial_{e'} u}{\partial_{e} u}\right\|_{C^\tau(\mathcal Q_{1/2})} + \left\|\frac{\partial_t u}{\partial_{e} u}\right\|_{C^\tau(\mathcal Q_{1/2})} \leq C.\]
This yields that the free boundary $\partial\{u>0\}$ is a $C^{1,\tau}$ graph in $\mathcal Q_{1/2}$, as wanted.
\end{proof}

Finally, we will need the following bound for solutions to parabolic equations in $C^{1,\tau}$ domains.

\begin{lem}\label{lemm-main-par-3}
Let $s$, $\LL$, $\delta$, and $u$, be as in Theorem \ref{thm-main}.
Assume that $\partial\{u>0\}$ is a $C^{1,\tau}$ graph in $\mathcal Q_{1/2}$.
Then
\[|\nabla u|+|\partial_t u|\leq C\big(|x|^s+|t|^s\big)\qquad \textrm{for}\quad (x,t)\in \mathcal Q_{1/4},\]
with $C$ depending only on  $n$, $s$, $\delta$, $\lambda$, $\Lambda$, $\tau$, and the $C^{1,\tau}$ norm of the graph.
\end{lem}

\begin{proof}
Notice that all derivatives of $u$ are solutions to a linear equation inside the domain $\Omega=\{u>0\}$, and they vanish in $\Omega^c$. 

Since $\Omega$ is monotone nondecreasing in time, we can use a supersolution for cylindrical (i.e., constant in time) domains to prove the bound for $t\leq0$.  
In case of $C^{1,1}$ domains this was done in \cite[Lemma 4.3]{FR17}, and the exact same argument works in $C^{1,\tau}$ domains by using \cite[Proposition 1.1]{RS-C1}.

Once we have the bound for $t\leq0$, since the domain is $C^{1,\tau}$ (in particular Lipschitz), we can use the same argument at any boundary point to deduce the validity of the desired estimate inside $\mathcal Q_{1/2}$, as wanted.
\end{proof}

\subsection{Proof of the main result}

Combining the previous results, we are essentially ready to prove our main parabolic theorem.
We just need a simple stability result contained in the next lemma.

\begin{lem} \label{lem-limits-par}
Let $s\in(0,1)$, and let $\lambda$ and $\Lambda$ be fixed positive constants.
Let $\{\LL_k\}_{k\geq1}$ be any sequence of operators of the form \eqref{L0}-\eqref{ellipt}.
Then, a subsequence of $\{\LL_k\}$ converges weakly to an operator $\LL$ of the same form.

Moreover, let $(u_k)$ and $(f_k)$ be sequences of functions satisfying, in the weak sense,
\[\partial_t u_k - \LL_k u_k = f_k \ \textrm{ in } \ \Omega\times (t_1,t_2)\]
for a given bounded domain $\Omega \subset\R^n$, and suppose that:
\begin{enumerate}
\item $u_k\to u$ uniformly in compact sets of $\R^n\times (t_1,t_2)$;
\item $f_k \to f$ uniformly in $\Omega\times (t_1,t_2)$;
\item $|u_k(x,t)| \leq M\left(1+|x|^{2s-\epsilon}\right)$ for all $x\in \R^n$ and $t\in (t_1,t_2)$, for some $M,\epsilon > 0$.
\end{enumerate}
Then $u$ satisfies
\[\partial_t u-\LL u = f \quad \textrm{ in } \Omega\times (t_1,t_2)\]
in the weak sense.
\end{lem}

\begin{proof}
The proof is very similar to that of \cite[Lemma 3.1]{FR17} and \cite[Lemma~3.2]{DRSV}, so we leave the details to the interested reader.
\end{proof}

\begin{proof}[Proof of Theorem \ref{thm-main}]
We first prove that, given $R_\circ\geq1$ and $\eps>0$, for $\eta>0$ small enough we have
\begin{equation}\label{contradiction-par}
\|u-u_\circ\|_{{\rm Lip}(\mathcal Q_{R_\circ})} \leq\eps,
\end{equation}
for some $u_\circ$ as in Theorem \ref{thm-main}.

Indeed, assume by contradiction that there is no $\eta>0$ for which \eqref{contradiction-par} holds.
Then, we have a sequence $\eta_k\to0$, and sequences of operators $\LL_k$ and solutions $u_k$, such that 
\[\|u_k-u_\circ\|_{{\rm Lip}(\mathcal Q_{R_\circ})} \geq\eps\]
for any $e\in \mathbb S^{n-1}$ and any $u_\circ$ as in Theorem \ref{thm-main}.
Then, by Corollary \ref{cor-almost-optimal-par} and Lemma~\ref{lem-limits-par}, up to a subsequence the functions $u_k$ converges in $C^1_{\rm loc}$ to a limiting solution $u$,
with operator $\LL$ as in  \eqref{L0}-\eqref{ellipt}, that satisfies the assumptions of the theorem with $\eta=0$.
However, by Proposition \ref{prop-classification} (if $s=\frac12$) or Proposition \ref{prop-classification>1/2} (if $s>\frac12$), it follows that $u$ is a 1D function satisfying \eqref{1D-profile}.
This means that we can take $u_\circ=u$ in \eqref{contradiction-par}, a contradiction.
Hence,  \eqref{contradiction-par} is proved.

Thanks to \eqref{contradiction-par}, the $C^{1,\tau}$ regularity of the free boundary follows from Lemma \ref{lemm-main-par-2}, and the bounds for $\nabla u$ and $\partial_t u$ at 0 follow from Lemma \ref{lemm-main-par-3}.
\end{proof}

\section{Optimal regularity of solutions}
\label{sec5}

We now prove Corollaries \ref{thm1} and \ref{thm2}, as well as the optimal regularity estimates from Corollaries \ref{thm-sol-2} and \ref{thm-sol-3}.

\begin{proof}[Proof of Corollaries \ref{thm1} and \ref{thm2}]
We begin by replacing $u$ with $u-\varphi$, so that $u$ now satisfies
\[u\geq0, \quad \partial_t u\geq0\quad \textrm{and} \quad D^2_{x,t} u \ge - C_1C_\circ{\rm{Id}}\qquad \textrm{in}\quad \R^n\times (-1,1),\]
\[\partial_t u-\LL u=f(x)\quad \textrm{in}\quad \{u>0\} \qquad \textrm{and} \qquad \partial_t-\LL u\geq f \quad \textrm{in} \quad \R^n\times(-1,1), \quad \textrm{with} \quad  |\nabla f|\leq C_1C_\circ,\]
\[\|\nabla u\|_{L^\infty(\R^n\times(-1,1))}+\|\partial_t u\|_{L^\infty(\R^n\times(-1,1))} \le C_1.\]
(Note that the semiconvexity of solutions follows from \cite[Lemma 2.1]{BFR2} or \cite[Proposition~2.4]{RT21}.)

We will prove at the same time Corollaries \ref{thm1} and \ref{thm2}, and in addition that, for every free boundary point $(x_\circ,t_\circ)$, we have
\begin{equation}\label{claim-final}
\begin{split} & |\nabla u|\leq C\big(|x-x_\circ|^s+|t-t_\circ|^{\min\left\{s,\,\frac{1-\delta}{2s}\right\}}\big)
\\
& |\partial_t u|\leq C\big(|x-x_\circ|^{\min\left\{s,\,2-2s-\delta\right\}}+|t-t_\circ|^{\min\left\{s,\,\frac{2-2s-\delta}{2s}\right\}}\big),\end{split}
\end{equation}
with $C$ depending only on $n$, $s$, $\delta>0$, and the ellipticity constants.

Dividing by a constant if necessary, and up to a translation, we may assume $C_\circ=1$ and $(x_\circ,t_\circ)=(0,0)$.
We now want to apply Theorem \ref{thm-main} iteratively in order to get the desired estimate.

Let $\kappa>0$ to be chosen later, and let $\eta>0$ be the constant given by Theorem~\ref{thm-main}.
We fix $k_\circ\in\mathbb N$ and define the functions
\[w_k(x,t):=\frac{\eta}{2^{k_\circ}C_1}\frac{u(2^{-k} x,2^{-2sk}t)}{(2^{-k})^{2-\delta}},\qquad k \in \mathbb N.\]
Note that, if $k$ is large enough, then $w_k$ satisfies
\[w_k\geq0, \quad \partial_t w_k\geq0\quad \textrm{and} \quad D^2_{x,t} w_k \ge - \eta{\rm{Id}}\qquad \textrm{in}\quad \R^n\times(-2^{2sk},2^{2sk}),\]
\[\partial_t w_k-\LL w_k=f_k\quad \textrm{in}\quad \{w_k>0\} \qquad \textrm{and} \qquad \partial_t w_k-\LL w_k\geq f_k \quad \textrm{in} \quad \R^n\times (-2^{2sk},2^{2sk}),\]
with $|\nabla f_k|\leq \eta$.
Moreover, 
\[\|\nabla w_{k_\circ}\|_{L^\infty(\R^n\times (-2^{2sk_\circ},2^{2sk_\circ}))} + \|\partial_t w_{k_\circ}\|_{L^\infty(\R^n\times (-2^{2sk_\circ},2^{2sk_\circ}))} \le 1.\]

In other words, for $k\geq k_\circ\gg 1$, all the assumptions of Theorem \ref{thm-main}, except possibly for the growth control on $\nabla w_k$ and $\partial_t w_k$ (which holds at least for $k=k_\circ$), are satisfied by $w_k$.

We then have two possibilities:

\vspace{2mm}

\noindent \emph{Case 1}. Assume that the functions $w_k$ satisfy 
\[R\|\nabla w_k\|_{L^\infty(\mathcal Q_R)} + R^{2s}\|\partial_t w_k\|_{L^\infty(\mathcal Q_R)} \leq R^{2-\delta}\qquad \textrm{for all}\quad R\geq 1,\quad k\geq k_\circ.\]
Then
\[\|\nabla u\|_{L^\infty(\mathcal Q_{2^{-k}})} = 2^{k_\circ}C_1\eta^{-1}(2^{-k})^{1-\delta}\|\nabla w_k\|_{L^\infty(\mathcal Q_1)} \leq  C(2^{-k})^{1-\delta},\]
and 
\[\|\partial_t u\|_{L^\infty(\mathcal Q_{2^{-k}})} = 2^{k_\circ}C_1\eta^{-1}(2^{-k})^{2-2s-\delta}\|\nabla w_k\|_{L^\infty(\mathcal Q_1)} \leq  C(2^{-k})^{2-2s-\delta}.\]
Therefore
\[|\nabla u|\leq C\big(|x|^{1-\delta}+|t|^{\frac{1-\delta}{2s}}\big)
\quad \textrm{and}\quad |\partial_t u|\leq C\big(|x|^{2-2s-\delta}+|t|^{\frac{2-2s-\delta}{2s}}\big),
\]
so \eqref{claim-final} follows.
Moreover, this implies 
\[u(x,t)\leq C\big(|x|^{2-\delta}+|t|^{\frac{2-\delta}{2s}}\big)\qquad \textrm{for}\quad (x,t)\in \mathcal Q_1,\]
and thus we have a \emph{non-regular point} ((ii) in Corollaries \ref{thm1} or \ref{thm2}).

\vspace{2mm}

\noindent \emph{Case 2}. If we are not in Case 1, then there is a maximal number $k_1\geq k_\circ$ such that 
\begin{equation}\label{growth-ell-contr tris}
R\|\nabla w_k\|_{L^\infty(\mathcal Q_R)} + R^{2s}\|\partial_t w_k\|_{L^\infty(\mathcal Q_R)} \leq R^{2-\delta}\qquad \textrm{for all}\quad R\geq1,\quad k_\circ\leq k\leq k_1.
\end{equation}
Then, by Theorem \ref{thm-main}, we find
\[\|w_{k_1}-u_\circ\|_{\rm Lip(\mathcal Q_1)}\leq \eps.\]
Moreover, $u_\circ$ is a multiple of $(x\cdot e+\varv t)_+^{1+\gamma}$ with $\|u_\circ\|_{{\rm Lip}(\mathcal Q_1)}\leq 1$, and therefore we have 
\[|\nabla u_\circ(x,t)|+|\partial_t u_\circ(x,t)|=A(x\cdot e+\varv t)^\gamma_+, \qquad \textrm{with}\quad 0\leq A \leq 1.\]
We claim that $A \geq \frac{1}5$.
Indeed, if not, then by triangle inequality
\[\|\nabla w_{k_1}\|_{L^\infty(\mathcal Q_1)} + \|\partial_t w_{k_1}\|_{L^\infty(\mathcal Q_1)}\leq \|\nabla u_\circ\|_{L^\infty(\mathcal Q_1)} +\|\partial_t u_\circ\|_{L^\infty(\mathcal Q_1)} + \eps\leq {\textstyle \frac15} + \eps \leq {\textstyle \frac14}.\]
Since $\nabla w_{k_1+1}(x)=2^{1-\delta}\nabla w_{k_1}(\frac{x}{2})$ and $\partial_t w_{k_1+1}(x)=2^{2-2s-\delta}\partial_t w_{k_1}(\frac{x}{2})$, this implies that 
\[2\|\nabla w_{k_1+1}\|_{L^\infty(\mathcal Q_2)}+2^{2s}\|\partial_t w_{k_1+1}\|_{L^\infty(\mathcal Q_2)}\leq 1.\]
Since 
\[\begin{split}
R\|\nabla &w_{k_1+1}\|_{L^\infty(\mathcal Q_R)}+R^{2s}\|\partial_t w_{k_1+1}\|_{L^\infty(\mathcal Q_R)} = \\
& = 2^{2-\delta}\left\{(R/2)\|\nabla w_{k_1}\|_{L^\infty(\mathcal Q_{R/2})}+(R/2)^{2s}\|\partial_t w_{k_1}\|_{L^\infty(\mathcal Q_{R/2})} \right\}\leq R^{2-\delta}\quad \textrm{for}\quad R\geq2,
\end{split}\]
then $w_{k_1+1}$ still satisfies the growth condition \eqref{growth-ell-contr tris}, a contradiction to the definition of~$k_1$.

Thanks to the claim (i.e., $A\geq \frac15$) we can apply 
Theorem  \ref{thm-main} to deduce that the free boundary $\partial\{w_{k_1}>0\}$ is a $C^{1,\tau}$ graph in $\mathcal Q_1$, and 
\[|\nabla w_{k_1}|+|\partial_t w_{k_1}|\leq C\big(|x|^s+|t|^s\big) \qquad \text{for}
\quad (x,t)\in \mathcal Q_1.\]
Since 
\[u(x,t) = C_2 R^{\delta-2} w_{k_1}(Rx,R^{2s}t),\]
with $R=2^{k_1}$,  we deduce that
\[|\nabla u|\leq CR^{\delta-1}\big(|Rx|^s+|R^{2s} t|^s\big) \leq CR^{\delta-1}\big(|Rx|^s+|R^{2s} t|^{\min\{s,\, \frac{1-\delta}{2s}\}}\big) \leq C\big(|x|^s+|t|^{\min\{s,\, \frac{1-\delta}{2s}\}}\big),\]
for all $(Rx,R^{2s}t)\in \mathcal Q_1$.
Similarly, we get 
\[|\partial_t u|\leq CR^{\delta+2s-2}\big(|Rx|^s+|R^{2s} t|^{s}\big),\]
and \eqref{claim-final} follows.

Finally, since the free boundary $\partial\{u>0\}$ is $C^{1,\tau}$ in a neighborhood of the origin, and thanks to a standard barrier argument (see Lemma \ref{supersolution} for the construction of the sub- and supersolutions in the critical case $s=\frac12$) we deduce that 
\[0<cr^{\gamma_\circ} \leq \sup_{\mathcal Q_r} u \leq Cr^{\gamma_\circ}\]
for $r>0$ small, where $\gamma_\circ:=\gamma\big(\LL,\frac{\nu_x}{|\nu_x|},\frac{\nu_t}{|\nu_x|}\big)$ is given by \eqref{gamma-exp}, and  $\nu=(\nu_x,\nu_t)$ is the inward unit normal to the free boundary. 
This means that we have a \emph{regular point} ((i) in Corollaries \ref{thm1} or \ref{thm2}), and we are done.
\end{proof}

Finally, we prove the optimal regularity of solutions.

\begin{proof}[Proof of Corollary \ref{thm-sol-2}]
As in the previous proof, we replace $u$ with $u-\varphi$.
Since $s=\frac12$, then \eqref{claim-final} holds at every free boundary point. 
Hence, combining it with interior regularity estimates for linear parabolic equations, the desired $C^{3/2}_{x,t}$ estimate follows.
\end{proof}

Notice that the previous proof does not work when $s>\frac12$.
Indeed, the reason for this is the parabolic scaling: even if we had \eqref{claim-final} at every free boundary point, then by interior regularity estimates we would only get that derivatives of $u$ are $C^s_{\rm par}$ (i.e., $C^s_x$ and $C^{1/2}_t$, recall \ref{Holder-par}), which is not the optimal regularity in $t$ when $s>\frac12$.
Because of this, some extra ideas are needed.

\begin{proof}[Proof of Corollary \ref{thm-sol-3}]
As before, we replace $u$ with $u-\varphi$.
Let $\mu:=\min\{s,1/s-1-\varepsilon\}$, with $\varepsilon>0$.
We want to prove that $\partial_t u\in C^\mu_t(\R^n\times [t_1,t_2])$.

For this, let $\rho=\rho(t_1)>0$ be such that $\mathcal Q_{2\rho}(x_1,t_1)\subset \R^n\times (0,T]$ for any $x_1\in \R^n$.
We consider a cutoff function $\psi\in C^\infty_c(\mathcal Q_{2\rho}(x_1,t_1))$ with $\psi\equiv1$ in $Q_\rho(x_1,t_1)$.
By the semiconvexity of solutions we know that $\partial_{tt} u\geq -C$. Thus
\[0\leq \int_{Q_{2\rho}(x_1,t_1)} \big(\partial_{tt} u +C\big)\psi
 = \int_{Q_{2\rho}(x_1,t_1)} \big( u\partial_{tt} \psi +C\psi\big) \leq C_1,\]
and thus 
\[\int_{Q_\rho(x_1,t_1)} |\partial_{tt} u| \leq C_2,\]
with $C_2$ independent of $x_1$.
Then, we define 
\[w(x,t):= \frac{\partial_t u(x,t+h)-\partial_t u(x,t)}{|h|^\mu} = \frac{1}{|h|^\mu}\int_0^h \partial_{tt} u(x,t+\zeta)\,d\zeta,\]
and notice that 
\[\int_{Q_{\rho/2}(x_1,t_1)} |w| \leq C_3,\]
as long as $|h|<\rho/2$.
In particular, this yields
\begin{equation}\label{bound-L1}
\int_{t_1-\rho^{2s}}^{t_2+\rho^{2s}}\int_{\R^n} \frac{|w(x,t)|}{1+|x|^{n+2s}}\,dx\,dt \leq C_4.
\end{equation}
On the other hand, thanks to \eqref{claim-final} we have that $w$ is uniformly bounded on the contact set, namely
\begin{equation}\label{bound-contactset}
|w(x,t)| = \frac{\partial_t u(x,t+h)}{|h|^\mu} \leq C_5\qquad \textrm{for}\quad (x,t)\in \{u=0\}.
\end{equation}
Furthermore, since $w$ satisfies 
\[|(\partial_t -\LL)w| \leq C_6 \qquad \textrm{in}\quad \{u>0\},\]
the function $\tilde w:= \max\{w,C_5\}$
satisfies
\[(\partial_t -\LL)w \leq C_7 \qquad \textrm{in}\quad \R^n\times (0,T).\]
In other words $w$ is a subsolution, so it follows from \eqref{bound-L1} and \cite[Lemma A.3]{RT21} that
\[\sup_{B_1\times[t_1,t_2]} w \leq C(C_4+C_7)=:C_8.\]
Applying the same argument with any ball $B_1(z)$ instead of $B_1$, recalling the definition of $w$ we deduce that
\[\big|\partial_t u(x,t+h)-\partial_t u(x,t)\big| \leq C_8|h|^\mu,\]
which gives the desired regularity in $t$.

When $s<\frac{\sqrt{5}-1}{2}$ we can repeat the exact same argument used above for any spacial second derivative $\partial_{\xi\xi} u$ with $\xi\in \mathbb S^n$, and we obtain $C^{1+s}$ regularity in all directions.

Instead, when $s\geq\frac{\sqrt{5}-1}{2}$, we combine \eqref{claim-final} with interior regularity estimates to get 
\[\|\nabla u\|_{C^s_{\rm par}(\R^n\times [t_1,t_2])}+\|\partial_t u\|_{C^{\min\{s,2-2s-\varepsilon\}}_{\rm par}(\R^n\times [t_1,t_2])} \leq C,\]
where $C^\beta_{\rm par}$ is defined in \eqref{Holder-par}.
\end{proof}

\appendix
\section{Barriers in moving domains}

The aim of this appendix is to construct sub- and supersolutions for linear parabolic equations in moving domains in case $s=\frac12$.
We start with the following simple result.

\begin{lem}\label{prop-supersolution}
Let $\Omega\subset \R^n\times \R$ be a bounded $C^{1,\tau}$ domain with $(0,0)\in \partial\Omega$, and let $d(x,t)={\rm dist}((x,t),\Omega^c)$.
Let $\rho$ be a regularized distance function, satisfying
\[C_\Omega^{-1}d\leq \rho\leq C_\Omega d,\qquad \|\rho\|_{C^{1,\alpha}(\overline\Omega)}\leq C_\Omega,\qquad |D^2\rho|\leq C_\Omega d^{\alpha-1}\qquad {\mbox{and}}\qquad |D^3\rho|\leq C_\Omega d^{\alpha-2}.\]
Let $\LL$ be any operator of the form \eqref{L0}-\eqref{ellipt}, with $s=\frac12$,
let~$\nu=(\nu_x,\nu_t)$ be the inward unit normal to~$\partial\Omega$,
let $\gamma_{\LL,\nu}:=\gamma\big(\LL,\frac{\nu_x}{|\nu_x|},\frac{\nu_t}{|\nu_x|}\big)$ be given by \eqref{gamma-exp}, and let $\gamma_\circ:=\gamma_{\LL,\nu(0)}$.

Then, for any $\varepsilon>0$, we have
\[(\partial_t -\LL)(\rho^{\gamma_\circ-\varepsilon}) \geq c_0d^{\gamma_\circ-\varepsilon-2s}>0\qquad \textrm{in} \ \{0<d(x,t)\leq \delta\}\cap \mathcal Q_{\delta}\]
and
\[(\partial_t -\LL)(\rho^{\gamma_\circ+\varepsilon}) \leq -c_0d^{\gamma_\circ-\varepsilon-2s}<0\qquad \textrm{in} \ \{0<d(x,t)\leq \delta\}\cap \mathcal Q_{\delta}\]
The constants $c_0>0$ and $\delta>0$ depend only on $\Omega$, $\varepsilon$, and the 
ellipticity constants.
\end{lem}

\begin{proof}
The proof is a minor modification of the one in \cite[Proposition 4.8]{DRSV}.
\end{proof}

We will also need the following:

\begin{prop}[Approximate solution]\label{prop-approx-solution}
Let $\Omega\subset \R^n\times \R$ be a bounded $C^{1,\tau}$ domain with $(0,0)\in \partial\Omega$.
Let $d$, $\rho$, $\LL$, $\nu$, $\gamma_{\LL,\nu}$, and $\gamma_\circ$ be as in Proposition \ref{prop-supersolution}.

Let $\bar\Gamma(x,t)$ be a function
that coincides with $\gamma_{\LL,\nu(x,t)}$ on $\partial\Omega$
and satisfies $|D^2\bar\Gamma(x,t)|\leq Cd^{\tau-2}$ inside $\Omega$.
Assume in addition that $\bar \Gamma\geq \gamma_\circ-\varepsilon/2$ in $\Omega\cap \mathcal Q_1$.

Let $\phi$ be a function that coincides with~$\rho^{\bar\Gamma}$
in a neighborhood of $\partial\Omega$, and such that $\|\phi\|_{C^{\gamma_\circ-\delta}}\leq C$.
Then
\begin{equation}\label{primo00}
\big|(\partial_t -\LL)\phi(x,t)\big| \leq Cd^{\gamma_\circ+\tau-\varepsilon-2s}\quad \textrm{for } (x,t)\in \Omega,
\end{equation}
as long as the exponent above is negative.

The constants $C$ depend only on $\varepsilon$, $\Omega$, and the ellipticity constants.
\end{prop}

\begin{proof}
The proof is a minor modification of the one in \cite[Proposition 4.10]{DRSV}.
\end{proof}

As a consequence of the previous Lemmas, we can now construct exact sub- and supersolutions in moving $C^{1,\tau}$ domains.

\begin{lem}[Sub- and Supersolutions]\label{supersolution}
Let $\Omega\subset \R^n\times \R$ be a bounded $C^{1,\tau}$ domain, let $d(x,t)={\rm dist}((x,t),\Omega^c)$, and let $\LL$, $\nu$, $\gamma_\circ$, and $\bar\Gamma$ be as in Proposition \ref{prop-approx-solution}.

Then there exist $\delta_\circ>0$ and two functions $\Phi_1$, $\Phi_2$ satisfying
\[(\partial_t -\LL)\Phi_1 \leq -1\quad \textrm{in} \ \mathcal Q_{\delta_\circ},\qquad \qquad (\partial_t -\LL)\Phi_2 \geq 1\quad \textrm{in} \ \mathcal Q_{\delta_\circ},\] 
and
\[C^{-1}d^{\bar\Gamma}\leq \Phi_i \leq Cd^{\bar\Gamma} \quad \textrm{in} \ \mathcal Q_1.\]
In particular, we have 
\[C^{-1}r^{\gamma_\circ} \leq \sup_{\mathcal Q_r} \Phi_i \leq Cr^{\gamma_\circ}\]
for $r>0$ small.
\end{lem}

\begin{proof}
It suffices to take
\[\Phi_1:=M\phi-\rho^{\gamma_\circ+\varepsilon}\qquad \textrm{and}\qquad 
\Phi_2:=M\phi+\rho^{\gamma_\circ+\varepsilon},\]
with $M>0$ large enough, $\varepsilon>0$ small enough, and $\phi,\rho$ given by Lemmas \ref{prop-approx-solution} and \ref{prop-supersolution}.
\end{proof}


\begin{thebibliography}{00}

\bibitem[AbR20]{AR} N. Abatangelo, X. Ros-Oton, \emph{Obstacle problems for integro-differential operators: higher regularity of free boundaries}, Adv. Math. \textbf{360} (2020), 106931, 61pp.


\bibitem[AC10]{AC} I. Athanasopoulos, L. Caffarelli, \emph{Continuity of the temperature in boundary heat control problems}, Adv. Math. \textbf{224} (2010), 293-315.

\bibitem[ACM18]{ACM} I. Athanasopoulos, L. Caffarelli, E. Milakis, \emph{On the regularity of the non-dynamic parabolic fractional obstacle problem}, J. Differential Equations \textbf{265} (2018), 2614-2647. 


\bibitem[ACM19]{ACMb} I. Athanasopoulos, L. Caffarelli, E. Milakis, \emph{Parabolic obstacle problems, quasi-convexity and regularity}, Ann. Sc. Norm. Super. Pisa Cl. Sci. \textbf{19} (2019), 781-825. 

\bibitem[ACS08]{ACS}  I. Athanasopoulos, L. Caffarelli, S. Salsa, \emph{The structure of the free boundary for lower dimensional obstacle problems}, Amer. J. Math. \textbf{130} (2008) 485-498.

\bibitem[AuR20]{AuR} A. Audrito, X. Ros-Oton, \emph{The Dirichlet problem for nonlocal elliptic operators with $C^\alpha$ exterior data}, 
Proc. Amer. Math. Soc. \textbf{148} (2020), 4455-4470.



\bibitem[BFR18]{BFR2} B. Barrios, A. Figalli, X. Ros-Oton, \emph{Free boundary regularity in the parabolic fractional obstacle problem}, Comm. Pure Appl. Math. \textbf{71} (2018), 2129-2159.

\bibitem[BL02]{basslevin} R. Bass, D. Levin, \emph{Transition probabilities for symmetric jump processes}, Trans. Amer. Math. Soc. \textbf{354} (2002), 2933-2953

\bibitem[Caf77]{C-obst} L. Caffarelli, \emph{The regularity of free boundaries in higher dimensions}, Acta Math. \textbf{139} (1977), 155-184.



\bibitem[CF13]{CF} L. Caffarelli, A. Figalli, \emph{Regularity of solutions to the parabolic fractional obstacle problem}, J. Reine Angew. Math. \textbf{680} (2013), 191-233.

\bibitem[CRS17]{CRS} L. Caffarelli, X. Ros-Oton, J. Serra, \emph{Obstacle problems for integro-differential operators: regularity of solutions and free boundaries}, Invent. Math. \textbf{208} (2017), 1155-1211.

\bibitem[CS05]{CS-book} L. Caffarelli, S. Salsa, \emph{A Geometric Approach to Free Boundary Problems}, Graduate Studies in Mathematics, Vol. {68}. American Mathematical Society, Providence, RI, 2005.

\bibitem[CSS08]{CSS} L. Caffarelli, S. Salsa, L. Silvestre, \emph{Regularity estimates for the solution and the free boundary of the obstacle problem for the fractional Laplacian}, Invent. Math. \textbf{171} (2008), 425-461.

\bibitem[CS09]{CS09} L. Caffarelli, L. Silvestre, \textit{Regularity theory for fully nonlinear integro-differential equations}, Comm. Pure Appl. Math. \textbf{62} (2009), 597-638.

\bibitem[CDM16]{CDM} J. A. Carrillo, M. G. Delgadino, A. Mellet, \emph{Regularity of local minimizers of the interaction energy via obstacle problems}, Comm. Math. Phys. \textbf{343} (2016), 747-781.

\bibitem[CD14]{ChaDav} H. Chang-Lara, G. D\'avila, \emph{H\"older estimates for non-local parabolic equations with critical drift}, J.~Differential Equations \textbf{260} (2016), 4237-4284.

\bibitem[CSV20]{CSV17} M. Colombo, L. Spolaor, B. Valichkov, \emph{Direct epiperimetric inequalities for the thin obstacle problem and applications}, Comm. Pure Appl. Math. \textbf{73} (2020), 384-420.


\bibitem[CSV20b]{CSV20b} M. Colombo, L. Spolaor, B. Valichkov, \emph{On the asymptotic behavior of the solutions to parabolic variational inequalities}, J. Reine Angew. Math. \textbf{768} (2020), 149-182.



\bibitem[CT04]{CT} R. Cont, P. Tankov, \emph{Financial Modelling With Jump Processes}, Chapman \& Hall, Boca Raton, FL, 2004.

\bibitem[DGPT17]{DGPT} D. Danielli, N. Garofalo, A. Petrosyan, T. To, \emph{Optimal regularity and the free boundary in the parabolic Signorini problem}, Mem. Amer. Math. Soc. \textbf{249} (2017), Num. 1181.

\bibitem[DS16]{DS2} D. De Silva, O. Savin, \emph{Boundary Harnack estimates in slit domains and applications to thin free boundary problems}, Rev. Mat. Iberoam. \textbf{32} (2016), 891-912.


\bibitem[DRSV22]{DRSV} S. Dipierro, Xavier Ros-Oton, J. Serra, E. Valdinoci, \emph{Non-symmetric stable operators: regularity theory and integration by parts}, Adv. Math. \textbf{401} (2022), 108321, 100pag.




\bibitem[DL76]{DL} G. Duvaut, J. L. Lions, \textit{Inequalities in Mechanics and Physics}, Grundlehren der Mathematischen Wissenschaften, Vol. 219. Springer-Verlag, 1976.

\bibitem[ERV17]{Elliot} C. Elliott, T. Ranner, C. Venkataraman, \emph{Coupled bulk-surface free boundary problems arising from a mathematical model of receptor-ligand dynamics}, SIAM J. Math. Anal. \textbf{49} (2017), 360-397.

\bibitem[FJ21]{FJ21} X. Fernandez-Real, Y. Jhaveri, \emph{On the singular set in the thin obstacle problem: higher order blow-ups and the very thin obstacle problem}, Anal. PDE \textbf{14} (2021), 1599-1669.

\bibitem[FR17]{FR17} X. Fernandez-Real, X. Ros-Oton, \emph{Regularity theory for general stable operators: parabolic equations}, J. Funct. Anal. \textbf{272} (2017), 4165-4221.

\bibitem[FR18]{FR18} X. Fernandez-Real, X. Ros-Oton, \emph{The obstacle problem for the fractional Laplacian with critical drift}, Math. Ann. \textbf{371} (2018), 1683-1735.

\bibitem[FR22]{FR20} X. Fernandez-Real, X. Ros-Oton, \emph{Regularity Theory for Elliptic PDE}, Zurich Lectures in Advanced Mathematics. EMS books, 2022.


\bibitem[Fri82]{Friedman} A. Friedman, \emph{Variational Principles and Free Boundary Problems}, Wiley, New York, 1982.

\bibitem[FS18]{FS17} M. Focardi, E. Spadaro, \emph{On the measure and the structure of the free boundary of the lower dimensional obstacle problem}, Arch. Ration. Mech. Anal. \textbf{230} (2018), 125-184.

\bibitem[GP09]{GP} N. Garofalo, A. Petrosyan, \emph{Some new monotonicity formulas and the singular set in the lower dimensional obstacle problem}, Invent. Math. \textbf{177} (2009), 415-461.

\bibitem[GPPS16]{GPPS} N. Garofalo, A. Petrosyan, C. Pop, M. Smit Vega Garcia, \emph{Regularity of the free boundary for the obstacle problem for the fractional Laplacian with drift},
Ann. Inst. H. Poincar\'e Anal. Non Line\'aire. \textbf{34} (2017), 533-570.

\bibitem[JN17]{JN} Y. Jhaveri, R. Neumayer, \emph{Higher regularity of the free boundary in the obstacle problem for the fractional Laplacian}, Adv. Math. \textbf{311} (2017), 748-795.


\bibitem[KPS15]{KPS} H. Koch, A. Petrosyan, W. Shi, \emph{Higher regularity of the free boundary in the elliptic Signorini problem}, Nonlinear Anal. \textbf{126} (2015), 3-44.

\bibitem[KRS19]{KRS} H. Koch, A. R\"uland, W. Shi, \emph{Higher regularity for the fractional thin obstacle problem}, New York J. Math. \textbf{25} (2019) 745-838.

\bibitem[Kuk21]{Kuk21} T. Kukuljan, \emph{The fractional obstacle problem with drift: higher regularity of free boundaries}, J. Funct. Anal. \textbf{281} (2021), Paper No. 109114, 74 pp.


\bibitem[Kuk22]{Kuk22} T. Kukuljan, \emph{$C^{2,\alpha}$ regularity of free boundaries in parabolic non-local obstacle problems}, Calc. Var. Partial Differential Equations \textbf{62} (2023), no. 2, Paper No. 36, 40 pp.

\bibitem[KKK21]{KaKK21} M. Kassmann, K-Y Kim,  T. Kumagai,  \emph{Heat kernel bounds for nonlocal operators with singular kernels}, J. Math. Pures Appl. \textbf{164} (2022), 1-26.


\bibitem[Mer76]{Merton} R. Merton, \emph{Option pricing when the underlying stock returns are discontinuous}, J. Finan. Econ.~\textbf{5} (1976), 125-144.

\bibitem[PS06]{PS06} G. Peskir, A. Shiryaev, \emph{Optimal Stopping and Free-Boundary Problems}, Lectures in Math., ETH Z\"urich, Birkh\"auser 2006.

\bibitem[PSU12]{PSU} A. Petrosyan, H. Shahgholian, N. Uraltseva, \textit{Regularity of free boundaries in obstacle-type problems}, Graduate Studies in Mathematics, Vol. {136}. American Mathematical Society, Providence, RI, 2012.



\bibitem[Ros16]{Ros16} X. Ros-Oton,  \emph{Nonlocal elliptic equations in bounded domains: a survey}, Publ. Mat. \textbf{60} (2016), 3-26.






\bibitem[RS16]{RS-Duke} X. Ros-Oton, J. Serra, \emph{Boundary regularity for fully nonlinear integro-differential equations}, Duke Math. J. \textbf{165} (2016), 2079-2154.

\bibitem[RS17]{RS-C1} X. Ros-Oton, J. Serra, \emph{Boundary regularity estimates for nonlocal elliptic equations in $C^1$ and $C^{1,\alpha}$ domains}, Ann. Mat. Pura Appl. \textbf{196} (2017), 1637-1668.



\bibitem[RS19]{RS-bdryH} X. Ros-Oton, J. Serra, \emph{The boundary Harnack principle for nonlocal elliptic equations in non-divergence form}, Potential Anal. \textbf{51} (2019), 315-331.


\bibitem[RT24]{RT21} X. Ros-Oton, C. Torres-Latorre, \emph{Optimal regularity for supercritical parabolic obstacle problems}, Comm. Pure Appl. Math. (2024), to appear.


\bibitem[Sal09]{Sal09} S. Salsa, \emph{The Problems of the Obstacle in Lower Dimension and for the Fractional Laplacian}, Springer Lecture Notes in Mathematics, 2045, Cetraro, Italy (2009).

\bibitem[SY23]{SY21} O. Savin, H. Yu, \emph{Contact points with integer frequencies in the thin obstacle problem}, Comm. Pure Appl. Math. (2023), to appear.




\bibitem[Ser18]{Serfaty} S. Serfaty, \emph{Systems of points with Coulomb interactions}, In \emph{Proceedings of the ICM 2018}.

\bibitem[Sil07]{S-obst} L. Silvestre, \emph{Regularity of the obstacle problem for a fractional power of the Laplace operator}, Comm. Pure Appl. Math. \textbf{60} (2007), 67-112.


\end{thebibliography}
\end{document}